\RequirePackage{fix-cm}
\newif\ifusepgfplot
\usepgfplottrue %comment out to bypass pgfplot and use figures in directory figs/

\documentclass[nopreprintline,3p]{elsarticle}

\usepackage[dvipsnames]{xcolor}
\usepackage{amssymb}
\usepackage{amsmath}
\usepackage{amsthm}
\usepackage{stmaryrd}
\usepackage{placeins}
\usepackage{multirow}
\usepackage{bm}
\usepackage{bbm}
\usepackage{pgfplots}
\pgfplotsset{compat=1.17}
\pgfplotstableset{col sep=comma}
\usepackage{pgfplotstable}
\usepackage{booktabs}
\usepackage{hyperref}
\usepackage{tikz}
\usepackage{caption}
\usepackage{subcaption}
\usepackage{algorithm}
\usepackage[algo2e]{algorithm2e} 
\usepackage{upgreek}
\usepackage{textgreek}
\usepackage{orcidlink}
\usepackage[misc]{ifsym}
\usepackage{pifont}% http://ctan.org/pkg/pifont

\usetikzlibrary{shapes,arrows,calc,external}
\usepgfplotslibrary{colorbrewer}

\DeclareMathOperator{\adj}{adj}
\DeclareMathOperator{\cond}{cond}

\newtheorem{theorem}{Theorem}
\newtheorem{lemma}[theorem]{Lemma}
\newtheorem{proposition}[theorem]{Proposition}

\newtheorem{example}{Example}
\newdefinition{definition}{Definition}
\newdefinition{remark}{Remark}

\newenvironment{customremark}[1]{%
  \renewcommand\theremark{#1}%
  \remark
}{\endremark}

\newenvironment{customexample}[1]{%
  \renewcommand\theexample{#1}%
  \example
}{\endexample}

\newcommand{\method}[3]{#1\text{DRK}#2\text{-}#3}
\newcommand{\amethod}[3]{A#1\text{DRK}#2\text{-}#3}

\newcommand{\dotPhi}[1]{\frac{\mathrm d^{#1}}{\mathrm dt^{#1}}{\Phi}}

\usepackage{todonotes}

\newcommand{\logLogSlopeTriangle}[7]
{
    \pgfplotsextra
    {
        \pgfkeysgetvalue{/pgfplots/xmin}{\xmin}
        \pgfkeysgetvalue{/pgfplots/xmax}{\xmax}
        \pgfkeysgetvalue{/pgfplots/ymin}{\ymin}
        \pgfkeysgetvalue{/pgfplots/ymax}{\ymax}

        \pgfmathsetmacro{\xArel}{#1}
        \pgfmathsetmacro{\yArel}{#3}
        \pgfmathsetmacro{\xBrel}{#1-(#5*#2)}
        \pgfmathsetmacro{\yBrel}{\yArel}
        \ifnum #6=1
			\pgfmathsetmacro{\xCrel}{\xArel}
		\else
			\pgfmathsetmacro{\xCrel}{\xBrel}
		\fi

        \pgfmathsetmacro{\lnxB}{\xmin*(1-(#1-(#5*#2)))+\xmax*(#1-(#5*#2))} % in [xmin,xmax].
        \pgfmathsetmacro{\lnxA}{\xmin*(1-#1)+\xmax*#1} % in [xmin,xmax].
        \pgfmathsetmacro{\lnyA}{\ymin*(1-#3)+\ymax*#3} % in [ymin,ymax].
        \pgfmathsetmacro{\lnyC}{\lnyA+#4*(\lnxA-\lnxB)}
        \pgfmathsetmacro{\yCrel}{\lnyC-\ymin)/(\ymax-\ymin)}

        \coordinate (A) at (rel axis cs:\xArel,\yArel);
        \coordinate (B) at (rel axis cs:\xBrel,\yBrel);
        \coordinate (C) at (rel axis cs:\xCrel,\yCrel);

		\ifnum #5=1
			\ifnum #6=1
				\def\loc{west};
			\else
				\def\loc{east};
			\fi
		\else
			\ifnum #6=1
				\def\loc{east};
			\else
				\def\loc{west};
			\fi
		\fi

        \ifnum #6 = 1
        	\draw[#7]	(A)--
            			(B)-- 
            			(C)-- node[pos=0.5,anchor=\loc] {#4}
           				cycle;
        \else
        	\draw[#7]	(A)--
            			(B)-- node[pos=0.5,anchor=\loc] {#4}
            			(C)-- 
           				cycle;
        \fi
    }
}

\definecolor{grassGreen}{RGB}{26,224,72}
\definecolor{mintGreen}{RGB}{73,186,142}
\definecolor{lightViolet}{RGB}{238,70,238}
\definecolor{lightPurple}{RGB}{171,0,255}
\definecolor{mediumPurple}{RGB}{103,0,154}

\pgfdeclareplotmark{mystar4*}{
	\node[star, star points=4, star point ratio=1.9, rotate=45, inner sep=0pt,minimum size=4.5pt,draw=lightPurple,fill=lightPurple!25!white, fill opacity=0.7, solid] {};
}

\pgfdeclareplotmark{mystar5*}{
    \node[star,star point ratio=1.9,inner sep=0pt,minimum size=4.5pt,draw=mediumPurple,fill=mediumPurple!25!white, fill opacity=0.7, solid] {};
}

\pgfdeclareplotmark{myflower3*}{
	\node[cloud, cloud puffs=3, cloud puff arc=250, inner sep=0pt,minimum size=3.5pt,draw=lightViolet,fill=lightViolet!25!white, fill opacity=0.7, solid] {};
}

\pgfdeclareplotmark{myflower4*}{
	\node[cloud, cloud puffs=4, cloud puff arc=250, inner sep=0pt,minimum size=3.5pt,draw=magenta,fill=magenta!25!white, fill opacity=0.7, solid] {};
}

\pgfplotscreateplotcyclelist{rainbow}{
	red, every mark/.append style={fill=red!25!white, fill opacity=0.7}, mark=pentagon*\\
	orange, every mark/.append style={fill=orange!25!white,	 fill opacity=0.7}, mark=triangle*\\
	grassGreen, every mark/.append style={fill=grassGreen!25!white, fill opacity=0.7}, mark=square*\\
	mintGreen, every mark/.append style={fill=mintGreen!25!white, fill opacity=0.7, solid}, mark=diamond*, dash pattern=on 1pt off 0.5pt\\
	blue, every mark/.append style={fill=blue!25!white, fill opacity=0.7, solid}, mark=*, dash pattern=on 1pt off 0.5pt\\
	lightPurple, mark=mystar4*\\
	mediumPurple, mark=mystar5*\\
}

\pgfplotscreateplotcyclelist{rainbow2}{
	red, every mark/.append style={fill=red!25!white, fill opacity=0.7}, mark=pentagon*\\
	orange, every mark/.append style={fill=orange!25!white,	 fill opacity=0.7}, mark=triangle*\\
	grassGreen, every mark/.append style={fill=grassGreen!25!white, fill opacity=0.7}, mark=square*\\
	mintGreen, every mark/.append style={fill=mintGreen!25!white, fill opacity=0.7, solid}, mark=diamond*, dash pattern=on 1pt off 0.5pt\\
	blue, every mark/.append style={fill=blue!25!white, fill opacity=0.7, solid}, mark=*, dash pattern=on 1pt off 0.5pt\\
	lightViolet, mark=myflower3*, densely dashed\\
	magenta, mark=myflower4*, densely dashed\\
}

\allowdisplaybreaks

\usepackage{ifthen}
\newboolean{compilefromscratch}
\setboolean{compilefromscratch}{false}

\begin{document}
\setlength{\emergencystretch}{1.6em}

\begin{frontmatter}



\title{Jacobian-free implicit MDRK methods for stiff systems of ODEs}


\author[inst1]{Jeremy Chouchoulis\orcidlink{0000-0003-3451-2568}}

\author[inst1]{Jochen Sch\"utz\orcidlink{0000-0002-6355-9130}} 

\affiliation[inst1]{organization={Faculty of Sciences \& Data Science Institute, Hasselt~University},%Department and Organization
            addressline={Agoralaan~Gebouw~D},
            city={Diepenbeek},
            postcode={3590},
            country={Belgium}}

\begin{abstract}
In this work, an approximate family of implicit multiderivative Runge-Kutta (MDRK) time integrators for stiff initial value problems is presented. The approximation procedure is based on the recent Approximate Implicit Taylor method (Baeza et al. in Comput. Appl. Math. 39:304, 2020). As a Taylor method can be written in MDRK format, the novel family constitutes a multistage generalization. Two different alternatives are investigated for the computation of the higher order derivatives: either directly as part of the stage equation, or either as a separate formula for each derivative added on top of the stage equation itself. From linearizing through Newton's method, it turns out that the conditioning of the Newton matrix behaves significantly different for both cases. We show that direct computation results in a matrix with a conditioning that is highly dependent on the stiffness, increasing exponentially in the stiffness parameter with the amount of derivatives. Adding separate formulas has a more favorable behavior, the matrix conditioning being linearly dependent on the stiffness, regardless of the amount of derivatives. Despite increasing the Newton system significantly in size, through several numerical results it is demonstrated that doing so can be considerably beneficial.
\end{abstract}


\end{frontmatter}


\section{Introduction}\label{sec:intro}
We are interested in developing stable and efficient \textit{implicit} multiderivative time integrators, see, e.g.,  \cite{KastlungerWanner1972,HaWa73,Butcher2005,TC10,Seal13,BAEZA201887,2020_Baeza_EtAl,SealSchuetzZeifang21,2022_Gottlieb_EtAl} and the references therein, for stiff ordinary differential equations (ODEs)
\begin{equation}\label{eq:ODEsystem}
y'(t) = \Phi(y),
\end{equation}
where ${\Phi\colon \mathbb{R}^M \rightarrow \mathbb{R}^M}$ is the flux and $y: \mathbb R^{+} \rightarrow \mathbb{R}^M$ the unknown solution variable.
In our case, stiffness is introduced through a variable $\varepsilon \ll 1$ into the flux, which is given by
\begin{equation}\label{eq:ODE-stiffComponent}
\Phi_i(y) = f_i(y_1,\dots, y_M) + \frac{g_i(y_1,\dots, y_M)}{\varepsilon}, \quad 1 \leq i \leq M,
\end{equation}
for smooth functions $f_i$ and $g_i$ that do not explicitly dependent on $\varepsilon$. Multiderivative methods not only take into account the first derivative $y'(t)$, but as well higher order time derivatives
\[
y^{(k)}(t) := \frac{\mathrm{d}^k}{{\mathrm{d}t}^k} y(t) \, .
\]
By repeatedly making use of the ODE system \eqref{eq:ODEsystem}, and ignoring the $t-$dependency of $y^{(k)}$ for the ease of presentation, this leads to the formulas
\begin{subequations} \label{eq:FaaDiBruno-timeDer}
\begin{align}
y^{(2)} &= \Phi'(y)y^{(1)} \, , \\
y^{(3)}  &= \Phi''(y) \bullet \left[y^{(1)} {\color{gray}|} y^{(1)}\right] +  \Phi'(y)y^{(2)} \, , \label{eq:FaaDiBruno-order2} \\ 
y^{(4)} &= \Phi'''(y)\bullet \left[y^{(1)} {\color{gray}|} y^{(1)} {\color{gray}|} y^{(1)}\right] + 3\Phi''(y) \bullet \left[y^{(1)} {\color{gray}|} y^{(2)}\right] + \Phi'(y)y^{(3)} \, ,
\end{align}
\end{subequations}
and so forth. The bullet operator is the tensor action, i.e.,
\begin{equation}
\Phi''' \bullet \left[u {\color{gray}|} v {\color{gray}|} w\right] := \sum\limits_{j,k,l=1}^M \frac{\partial^3 \Phi}{\partial y_j\partial y_k\partial y_l} u_j v_k w_l \, ,
\end{equation}
where $u,v,w \in \mathbb{R}^{M}$.
Already at this introductory level, it can be seen that it is quite cumbersome to explicitly put all the terms used in \eqref{eq:FaaDiBruno-timeDer} into an algorithm.
Furthermore, plugging $\Phi_i(y)$ into \eqref{eq:FaaDiBruno-timeDer} reveals that $y^{(k)} = \mathcal{O}(\varepsilon^{-k})$.
% yields
% \begin{subequations}\label{eq:FaaDiBruno-timeDer-stiffComponent}
% \begin{align}
% y^{(2)}_i &= \frac{1}{\varepsilon^2}\frac{\partial g_i}{\partial y_i}g_i + \mathcal{R}_2(y) \, , \\
% y^{(3)}_i  &= \frac{1}{\varepsilon^3} \left(\frac{\partial^2 g_i}{{\partial y_i}^2}g_i^2 + \bigg(\frac{\partial g_i}{\partial y_i} \bigg)^2 g_i \right) + \mathcal{R}_3(y) \, ,  \\
% y^{(4)}_i &= \frac{1}{\varepsilon^4} \left(\frac{\partial^3 g_i}{{\partial y_i}^3}g_i^3 + 4\frac{\partial^2 g_i}{{\partial y_i}^2}\frac{\partial g_i}{\partial y_i}g_i^2 + \bigg(\frac{\partial g_i}{\partial y_i} \bigg)^3 g_i \right) + \mathcal{R}_4(y) \, ,
% \end{align}
% \end{subequations}
% with $\mathcal{R}_k(y)$ the remaining terms pertaining to $f_i$ and the other components $\Phi_j(y)$ (${j \neq i}$).
As a result, the derivatives $y^{(k)}$ quickly tend to become extremely large with each added order of the derivative,  potentially leading to a huge disparity in values handled in a multiderivative solver. Therefore, one can expect the typical limitations associated to floating-point arithmetic. In particular, the algebraic system of equations that results from the nonlinear timescheme is strongly influenced. It is shown numerically in this work that the conditioning of the linearized equation system behaves as $\mathcal{O}(\varepsilon^{-k})$.

% Despite MD methods being recognized for its potential and computing power ubiquitously having increased over the decades, these methods have not been put much into practice. The complicated formulas \eqref{eq:FaaDiBruno-timeDer} for the time derivatives most likely are at cause of this, being the most significant hurdle to pass. In an effort to extract full potential of MD time integrators, several approximation techniques have been realized. \cj{if possible here a list with references of attempts, existing approaches}

In 2018, Baeza et al. \cite{BAEZA201887} have constructed a recursive algorithm on the basis of centered finite differences that approximates the derivatives $y^{(k)}$ for a Taylor expansion, accordingly named the Approximate Taylor (AT) method. This approach directly stems from a recursive finite difference scheme that was designed for the circumvention of the Cauchy-Kovalevskaya procedure in the context of hyperbolic conservation laws \cite{ZorioEtAl}. In order to deal with stiffness and strict timestepping restrictions, more recently Baeza et al. \cite{2020_Baeza_EtAl} extended the AT method with an implicit variant, named the Approximate Implicit Taylor method. To simplify the computation of the Newton Jacobian, \cite{2020_Baeza_EtAl} suggests including additional equations into the ODE system for the calculation of the derivatives $y^{(k)}$. We show that this as well can improve the conditioning of the Jacobian compared to the $\mathcal{O}(\varepsilon^{-k})$ behavior that is achieved by directly incorporating \eqref{eq:FaaDiBruno-timeDer}.

In this work, we generalize the approximate implicit Taylor method to more general multiderivative Runge-Kutta (MDRK) schemes. This improves the solution quality significantly. While a Taylor method has order of convergence $\mathcal{O}(\Delta t^{k})$, with $k$ denoting the maximally used derivative, MDRK schemes can achieve the same order through less derivatives by incorporating more stages. Furthermore, we thoroughly investigate multiple methods to solve the resulting algebraic system of equations.
 Although  all methods are equivalent with \textit{infinite} machine precision, we observe that numerically, the methods differ quite significantly.

First, in Sect.~\ref{sec:MDRK} traditional MDRK time integrators for ODEs are introduced, highlighting the variety of ways to compute the time derivatives $y^{(k)}$, among which a review of the AT procedure is given. Next, Sect.~\ref{sec:Newton_Stability} is devoted to understanding the stability of the linear system obtained from applying Newton's method. In settings with timesteps large compared to the stiffness parameter $\varepsilon$, we show that the Newton Jacobian has a condition number that grows exponentially with the amount of derivatives. As an alternative for the traditional MDRK approach, in Sect.~\ref{sec:DerSol}, along the lines of the approximate implicit Taylor method, we introduce the MDRK-DerSol approach, where the derivatives are computed as solution variables via new relations in a larger ODE system. We verify numerically that the Newton Jacobian of this bigger system has a more favorable conditioning asymptotically for $\varepsilon$ going to $0$. Finally, our conclusions are summarized and future endeavors are explored in Sect.~\ref{sec:conclusion}.

\section{Implicit multiderivative Runge-Kutta solvers}\label{sec:MDRK}
In order to apply a time-marching scheme to Eq.~\eqref{eq:ODEsystem}, we discretize the temporal domain with a constant\footnote{
A fixed $\Delta t$ is used for solving any ODE system described within this work. Nevertheless, all presented methods can readily be applied with a variable timestep ${\Delta t}^n$ if needed.
} timestep $\Delta t$
and iterate $N$ steps such that $\Delta t = T_\text{end} / N$. Consequently, we define the time levels by
\begin{equation*}\label{eq:time-discr}
t^n := n\Delta t \, \qquad 0 \leq n \leq N.
\end{equation*}
The central class of time integrators in this work are \emph{implicit} MDRK methods. By adding extra temporal derivatives of $\Phi(y)$, these form a natural generalization of classical implicit Runge-Kutta methods.
To present our ideas, let us formally define the MDRK scheme as follows:
\begin{definition}[Kastlunger, Wanner {\cite[Section 1]{KastlungerWanner1972}}]\label{def:MDRK-scheme}
A $q$-th order implicit $\mathtt{r}$-derivative Runge-Kutta scheme using $\mathtt{s}$ stages ($\method{\mathtt{r}}{q}{\mathtt{s}}$) is any method which can, for given coefficients $a_{{\color{BlueGreen}l}{\color{VioletRed}\nu}}^{(k)}$, $b^{(k)}_{\color{BlueGreen}l}$ and $c_{\color{BlueGreen}l}$, be formalized as
\begin{subequations}
\begin{equation}\label{eq:MDRK-ODE-stages}
{y}^{n,{\color{BlueGreen}l}} := {y}^n + \sum\limits_{k=1}^{\mathtt{r}} {\Delta t}^{k} \sum\limits_{{\color{VioletRed}\nu}=1}^{\mathtt{s}} a_{{\color{BlueGreen}l}{\color{VioletRed}\nu}}^{(k)}\dotPhi{k-1}\left({y}^{n,{\color{VioletRed}\nu}} \right), \quad
{\color{BlueGreen}l} = 1,\dots,\mathtt{s},
\end{equation}
where ${y}^{n,{\color{BlueGreen}l}}$ is a stage approximation of $y$ at time $t^{n,{\color{BlueGreen}l}} := t^n + c_{\color{BlueGreen}l} \Delta t$. The update is given by
\begin{equation}\label{eq:MDRK-ODE-update}
{y}^{n+1} := {y}^n + \sum\limits_{k=1}^{\mathtt{r}} {\Delta t}^{k} \sum\limits_{{\color{BlueGreen}l}=1}^{\mathtt{s}} b^{(k)}_{\color{BlueGreen}l}\dotPhi{k-1}({y}^{n,{\color{BlueGreen}l}}) \, . \quad \phantom{\quad
{\color{BlueGreen}l} = 1,\dots,\mathtt{s}}
\end{equation}
\end{subequations}
Typically, the values of $a_{{\color{BlueGreen}l}{\color{VioletRed}\nu}}^{(k)}$, $b^{(k)}_{\color{BlueGreen}l}$ and $c_{\color{BlueGreen}l}$ are summarized in an extended Butcher tableau, see \ref{app:ButherTableaux} for some examples.
\end{definition}
As can be seen from Eq.~\eqref{eq:FaaDiBruno-timeDer}, at least the $k$-th order Jacobian tensor
\begin{equation}
\Phi^{k}(y) = \frac{\partial^k \Phi}{\partial y^k}(y) \, 
\end{equation}
is needed for the derivation of $\dotPhi{k}$. For systems of ODEs, $\Phi^{k}$ is an $M \times \hdots \times M$ ($k$-times) tensor.
The generalization of \eqref{eq:FaaDiBruno-order2}, named Fa\'a Di Bruno's formula (see \cite[Prop.~1]{2020_Baeza_EtAl}), can therefore be very expensive.
A more sensible way to obtain the time derivatives of $\Phi$ 
is from the recursive relation
\begin{equation}\label{eq:timeDerRecursive}
\dotPhi{k}(y) = {\left[\frac{\mathrm{d}^{k-1}\Phi(y)}{{\mathrm{d}t}^{k-1}}\right]}'y^{(1)}\, .
\end{equation}
The prime symbol here denotes the Jacobian derivative with respect to $y$. Here, the quantities ${\left[\frac{\mathrm{d}^{k-1}\Phi}{{\mathrm{d}t}^{k-1}}\right]}'$ are $M \times M$ matrices, regardless of $k$.

\begin{remark}\label{rem:mathEquivFaaDiBrunoRecursive}
Although Fa\'a Di Bruno's formula is mathematically equivalent to the recursive relation \eqref{eq:timeDerRecursive}, numerical results do in actuality differ. Due to the many tensor actions with $\Phi^k$ in Fa\'a Di Bruno's formula, numerical computations are much more prone to round-off errors. To illustrate this, we will apply both Fa\'a Di Bruno's formula, as in \eqref{eq:FaaDiBruno-timeDer}, and the recursive relation \eqref{eq:timeDerRecursive}. We refer to the tensors $\Phi^k$ with the wording ``Exact Jacobians'' (EJ) from hereon.
\end{remark}

%\begin{remark}
% Note that standard Taylor methods can be cast in the framework of Def.~\ref{def:MDRK-scheme} through setting $\mathtt{s}=1$, $a_{11}^{(k)}=0$ and $b_1^{(k)} = 1/k!$ $(k=1,\dots,\mathtt{r})$.
% The multiderivative Runge-Kutta schemes used in this work can be found through their extended Butcher tableaux in Appendix~\ref{app:ButherTableaux}. 
%\end{remark}	

\subsection{Approximating the time derivatives}

Despite the availibility of the recursive relation \eqref{eq:timeDerRecursive}, the Jacobian derivatives w.r.t. $y$ within this relation can nevertheless be quite intricate to deal with. And as such, avoiding Jacobian derivation by hand often leads to the use of symbolic computing software to allow the user to focus directly on the numerical procedure. There are two major downsides here, first, it being that symbolic software is computationally expensive, and secondly, not always feasible to apply. For large numerical packages for example, generally it is not desirable to significantly alter vital portions of code. To overcome symbolic procedures completely, a high-order centered differences approximation strategy has recently been developed by Baeza et al. \cite{BAEZA201887, 2020_Baeza_EtAl} to obtain values
\begin{align}\label{eq:approximateDerivatives}
\begin{split}
\widetilde{y}^{(k)} &= y^{(k)}  + \mathcal{O}({\Delta t}^{\mathtt{r} -k + 1})
= \frac{\mathrm{d}^{k-1}}{{\mathrm{d}t}^{k-1}}\Phi(y) + \mathcal{O}({\Delta t}^{\mathtt{r} -k + 1}) \, ,
\end{split}
\end{align}
for $k = 2, \dots, \mathtt{r}$. An overview of the method is given here; first, necessary notation is introduced.
%\cj{MAYBE \\
%We are interested in implicit schemes, and therefore require a nonlinear solver such as Newton's method. Computing the associated Jacobian to Newton then involves obtaining
%\begin{equation}\label{eq:timeDerRecursive_der}
%\frac{\partial}{\partial y}\left(\dotPhi{k}(y)\right) = {\left[\frac{\mathrm{d}^{k-1}\Phi}{{\mathrm{d}t}^{k-1}}\right]}''\Phi(y) + {\left[\frac{\mathrm{d}^{k-1}\Phi}{{\mathrm{d}t}^{k-1}}\right]}'\Phi'(y) \, .
%\end{equation}
%}
%%
%
\begin{definition}
For any number $p\in \mathbb{N}$, define the locally centered stencil function having $2p+1$ nodes by means of angled brackets
\begin{equation}\label{eq:stencilFunction}
\langle\cdot\rangle \colon \mathbb{Z} \to \mathbb{Z}^{2p+1} \colon z \mapsto \begin{pmatrix} z-p, & \dots, & z+p \end{pmatrix} ^T \, .
\end{equation}
\end{definition}
In this manner it is possible to write the vectors
\begin{equation}\label{eq:stencilVector}
\mathbf{y}^{\langle n \rangle} := \begin{pmatrix}
y^{n-p} \\
\vdots \\
y^{n+p}
\end{pmatrix} \quad \text{and} \quad
y(\mathbf{t}^{\langle n \rangle}) := \begin{pmatrix}
y(t^{n-p}) \\
\vdots \\
y(t^{n+p})
\end{pmatrix} \, .
\end{equation}
Such representation allows us to concisely write down approximations $\widetilde{y}^{(k)}$ to $y^{(k)}$. Let $k=1,\dots,\mathtt{r}$ be the derivative order of interest which we would like to approximate.
\begin{lemma}[Carrillo, Par\'{e}s {\cite[Proposition 4]{CarrilloPares2019}}, Zor\'{\i}o et al. {\cite[Proposition 2]{ZorioEtAl}}]
For $k\geq1$ and $p \geq \lfloor\frac{k+1}{2}\rfloor$ ($p\in \mathbb{N}$), there exist $2p+1$ quantities $\delta^k_{p,j}\in \mathbb{R}$ for $j = -p,\dots,p$, such that the linear operator
\begin{equation}\label{eq:IntPolyCentDerLinear}
P^{(k)} \colon  \mathbb{R}^{2p+1} \rightarrow \mathbb{R}, \qquad \mathbf{v} \mapsto \frac{1}{{\Delta t}^{k}} \sum\limits_{j=-p}^p \delta^k_{p,j}v_j 
\end{equation}
approximates the $k$-th derivative up to order $\omega := 2p - 2\lfloor\frac{k-1}{2}\rfloor$, \emph{i.e.}
\begin{equation}
P^{(k)} y(\mathbf{t}^{\langle n \rangle})  = y^{(k)}(t^{n})  + \mathcal{O}({\Delta t}^{\omega}) \, .
\end{equation}
\end{lemma}
%\begin{corollary}
%For $q$-th order scheme, such that $y^{n} = y(t^n) + \mathcal{O}({\Delta t}^{q})$, there holds
%\begin{equation}
%P^{(k)} \mathbf{y}^{\langle n+1 \rangle}  = y^{(k)}  + \mathcal{O}({\Delta t}^{\omega}) + \mathcal{O}({\Delta t}^{q-k}) \, .
%\end{equation}
%\end{corollary}

The linear operator $P^{(k)}$, however, is difficult to apply in practice, since it introduces additional unknown values $y(t^{n+1}), \dots, y(t^{n+p})$ into the stencil. In order to bypass the issue of creating more unkowns, in \cite{BAEZA201887, 2020_Baeza_EtAl, ZorioEtAl} a recursive strategy that includes Taylor approximations into the centered difference operator is incorporated. This gives the following computations of the values $\widetilde{y}^{(k)}$:
\begin{align}\label{eq:timeDerApprox}
\begin{split}
 \widetilde{y}^{(1)} &:= \Phi(y^{n}),\\
 \widetilde{y}^{(k)} &:= P^{(k-1)}\mathbf{\Phi}^{k-1, \langle n \rangle}_T, \quad 2 \leq k \leq \mathtt{r} \, ,
\end{split}
\end{align}
in which
\begin{equation}\label{eq:approxPhiValues}
\Phi_T^{k-1,n+j} := \Phi\left(y^{n} + \sum\limits_{m=1}^{k-1}\frac{(j\Delta t)^m}{m!} \widetilde{y}^{(m)}\right)
\end{equation}
is an approximation to $\Phi\!\left(y(t^{n+j})\right)$. By adopting the recursive Taylor approach \eqref{eq:timeDerApprox}-\eqref{eq:approxPhiValues} into Def. \ref{def:MDRK-scheme}, we acquire a novel family of time-marching schemes:
%\begin{definition}[AMDRK method]\label{def:AMDRK-scheme}
%An implicit $q$-th order accurate \emph{approximate} $\mathtt{r}$-derivative Runge-Kutta scheme using $\mathtt{s}$ stages ($\amethod{\mathtt{r}}{q}{\mathtt{s}}$) is any method which can be formalized as
%\begin{subequations}
%\begin{equation*}\label{eq:AMDRK-ODE-stages}
%{y}^{n,{\color{BlueGreen}l}} := {y}^n + \sum\limits_{k=1}^{\mathtt{r}} {\Delta t}^{k} \sum\limits_{{\color{VioletRed}\nu}=1}^{\mathtt{s}} a_{{\color{BlueGreen}l}{\color{VioletRed}\nu}}^{(k)}P^{(k-1)}(\mathbf{\Phi}_T^{{\color{VioletRed}\nu}})^{k-1, \langle n+1 \rangle}
%\end{equation*}
%for ${\color{BlueGreen}l} = 1,\dots,\mathtt{s},$ where ${y}^{n,{\color{BlueGreen}l}}$ is a stage approximation at time $t^{n,{\color{BlueGreen}l}} := t^n + c_{\color{BlueGreen}l} \Delta t$. The update is given by
%\begin{equation*}\label{eq:AMDRK-ODE-update}
%{y}^{n+1} := {y}^n + \sum\limits_{k=1}^{\mathtt{r}} {\Delta t}^{k} \sum\limits_{{\color{BlueGreen}l}=1}^{\mathtt{s}} b^{(k)}_{\color{BlueGreen}l}P^{(k-1)}(\mathbf{\Phi}_T^{{\color{BlueGreen}l}})^{k-1, \langle n+1 \rangle} \, .
%\end{equation*}
%\end{subequations}
%\end{definition}
\begin{definition}[AMDRK method]\label{def:AMDRK-scheme}
The $\method{\mathtt{r}}{q}{\mathtt{s}}$ scheme (Def. \ref{def:MDRK-scheme}) in which the time derivatives $\dotPhi{k-1}(y)$ are approximated by using the formulas  \eqref{eq:timeDerApprox}-\eqref{eq:approxPhiValues} is called the \emph{Approximate} MDRK method, denoted by the short-hand notation $\amethod{\mathtt{r}}{q}{\mathtt{s}}$.
\end{definition}
Without proof -- it is very similar to related cases, see for example \cite{BAEZA201887} in the context of explicit Taylor schemes for ODEs and \cite{2021_Chouchoulis_EtAl} for explicit MDRK schemes applied to hyperbolic PDEs -- we state the order of convergence:
\begin{theorem}\label{thm:AMDRK-consistency}
The consistency order of an $\amethod{\mathtt{r}}{q}{\mathtt{s}}$ method is ${\min(2p+1, q)}$, the variable $q$ being the consistency order of the underlying MDRK method, and $p$ denoting the use of the $2p+1$ points $\left\{t^{n-p}, \dots, t^{n+p} \right\}$ in Eq.~\eqref{eq:timeDerApprox}.
\end{theorem}
\begin{remark}
Note that the variable $p$ is not defined in the terminology $``\amethod{\mathtt{r}}{q}{\mathtt{s}}"$. Since the consistency order is $\min(2p+1, q)$, the optimal choice w.r.t. computational efficiency is to set $p = \lfloor q/2 \rfloor$. Throughout this paper $p$ is chosen along this line of reasoning for all numerical results.
\end{remark}

\begin{figure}
\centering
\resizebox{0.85\textwidth}{!}{%
	\ifthenelse{\boolean{compilefromscratch}}{
        \tikzsetnextfilename{AMDRK-scheme}  
		\input{figs_tikz/AMDRK-scheme.tikz}
	  }
	  {
	  \includegraphics{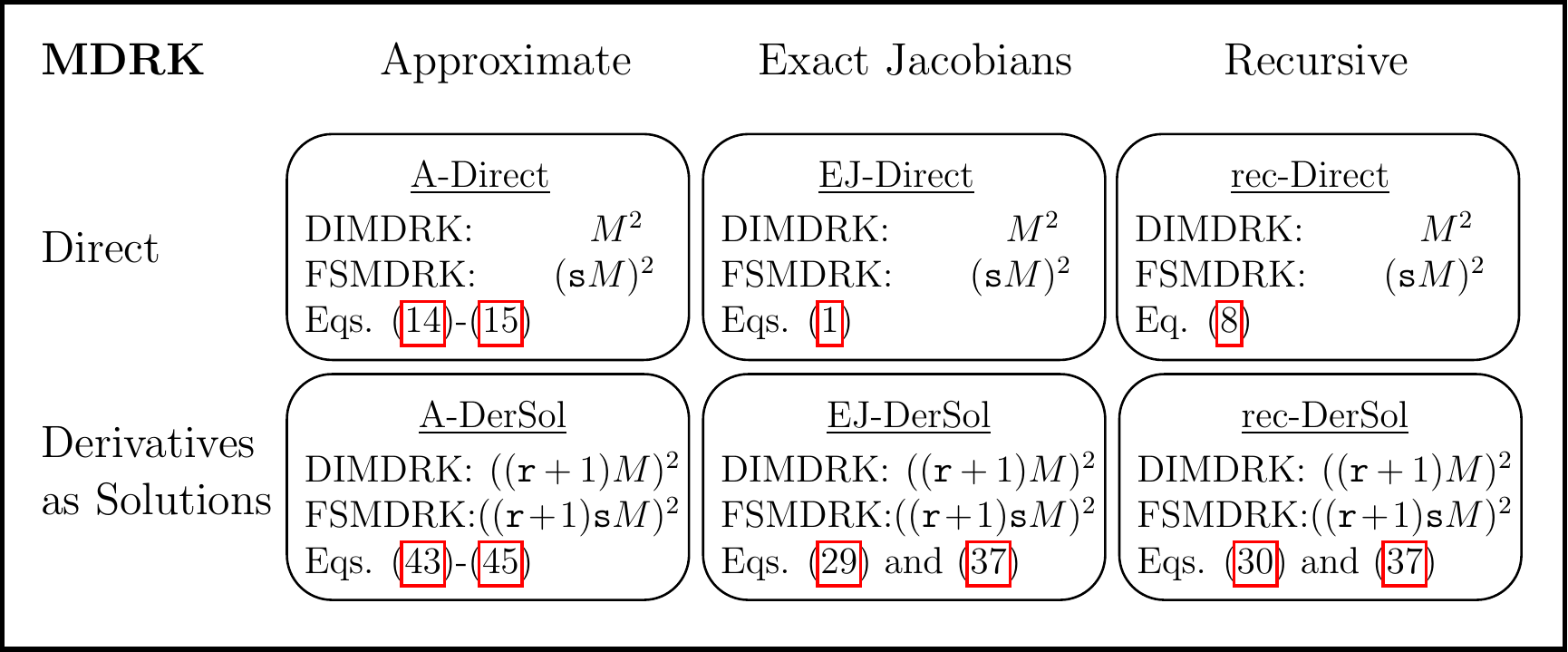}
	  }
}
	\caption{An overview of the different MDRK approaches applied in this work. In here $M$ represents the size of the ODE system being solved, $\mathtt{r}$ and $\mathtt{s}$ represent the amount of derivatives and the amount of stages, respectively, of the MDRK scheme (Def. \ref{def:MDRK-scheme}). For both DIMDRK and FSMDRK schemes (Subsect. \ref{subseq:specific-schemes}) the size of the linear system resulting from applying Newton's method is displayed.
	In comparison to the other Direct approaches, EJ-Direct neccessitates the most computations through tensor calculations, making it less fit for efficient time integration. Moreover, we found that using rec-Dersol leads to algebraic systems with extremely large condition numbers in relation to the other DerSol procedures; rec-Dersol hence turned out to be less suited for stiff equations than the others.
%	We have found that using EJ-Direct or rec-DerSol leads to algebraic systems with extremely large condition numbers in relation to the other procedures; they hence turned out to be less suited for stiff equations than the others.
	%The red boxes highlight the approaches which we highly disadvise to be used for practical purposes. \cj{-- TODO Jochen}
	}
	\label{fig:AMDRK-scheme}
\end{figure}

In the first row of Figure \ref{fig:AMDRK-scheme} an overview of the thus far presented MDRK methods is given, with differences focused around the computations of the derivatives $y^{(k)}$. As all of these MDRK methods exclusively solve the Eqs. \eqref{eq:MDRK-ODE-stages}-\eqref{eq:MDRK-ODE-update}, we refer to them as ``Direct'' in this work.

\subsection{Specific MDRK schemes}\label{subseq:specific-schemes}
As is the case for standard Runge-Kutta methods, the \mbox{(A)$\mathtt{r}$DRK$q$-$\mathtt{s}$} method has a lot of flexibility in choosing the coefficients $a_{{\color{BlueGreen}l}{\color{VioletRed}\nu}}^{(k)}$ and $b^{(k)}_{\color{BlueGreen}l}$. We put spotlight on two varying implementations of the (A)$\mathtt{r}$DRK$q$-$\mathtt{s}$ method:
\begin{itemize}
\item \textit{Full Storage MDRK (FSMDRK)}: \\
Under the assumption that the Butcher tableau consists of dense matrices, all $\mathtt{s}$ stages should be solved for simultaneously. This approach can be applied for any existing MDRK scheme, but might not be efficient as it often leads to large systems of equations.

\item \textit{Diagonally Implicit MDRK (DIMDRK)}:\\ If each stage ${\color{BlueGreen}l}$ only depends on previous stages ${\color{VioletRed}\nu} = 1, \dots, {\color{BlueGreen}l}-1$, and is only implicit in itself, it can be more efficient to solve for each stage one at a time.
\end{itemize}
In \ref{app:ButherTableaux} the extended Butcher tableaux used in this paper are displayed, three families are considered:
\begin{itemize}
\item \textit{Taylor schemes (Tables \ref{tab:explTaylor}-\ref{tab:implTaylor})}: \\
The explicit and implicit Taylor method can be reformulated as a single-stage MDRK scheme. This respectively leads to the Approximate Explicit Taylor methods in \cite{BAEZA201887} and the Approximate Implicit Taylor methods in \cite{2020_Baeza_EtAl}. Having only one stage, the FSMDRK and DIMDRK approaches are equivalent.

\item \textit{Hermite-Birkhoff (HB) schemes (Tables \ref{tab:HB-I2DRK4-2s}-\ref{tab:HB-I4DRK8-2s})} :\\
The coefficients are obtained from the Hermite-Birkhoff quadrature rule \cite{SealSchuetzZeifang21, QuarteroniNumA} which integrates a Hermite polynomial that also takes derivative data into account, with possibly a varying amount of derivative data per point. By taking equispaced abscissa $c_{{\color{BlueGreen}l}}$, the resulting tableau is fully implicit whilst having a fully explicit first stage. Hence, for $\mathtt{s} > 2$, stages $2$ to $\mathtt{s}$ should be solved with an FSMDRK method.

\item \textit{Strong-Stability Preserving (SSP) schemes (Tables \ref{tab:SSP-I2DRK3-2s}-\ref{tab:SSP-I2DRK4-5s})}: \\
In \cite{2022_Gottlieb_EtAl}, Gottlieb et al. have constructed implicit multiderivative SSP schemes.
%\[
%\|y^n + \Delta t \Phi(y^n) \| \leq \|y^n \| \quad \text{for} \quad \Delta t \leq {\Delta t}_{\text{FE}} \, ,
%\]
%where $\| \cdot \|$ is a convex norm and ${\Delta t}_{\text{FE}}$ is dictated by a CFL condition. In \cite{2022_Gottlieb_EtAl}, Gottlieb et al. have constructed implicit multiderivative SSP schemes that additionally satisfy some backward derivative condition
%\[
%\|y^n - {\Delta t}^2 \dotPhi{}(y^n) \| \leq \|y^n \| \quad \text{for} \quad {\Delta t}^2 \leq \dot{k}{\Delta t}^2_{\text{FE}} 
%\]
%for some $\dot{k} > 0$. Under certain requirements (which determine the coefficients), the obtained schemes can be shown to have an unconditional timestep restriction independent of a stiff component. 
The tableaux are diagonally implicit, and therefore each stage can be solved for one after another. Both the DIMDRK and FSMDRK approach are thus valid, with the DIMDRK approach likely being more efficient.
\end{itemize}

\subsection{Nonlinear solver}
The implicit (A)MDRK scheme (Defs. \ref{def:MDRK-scheme} and \ref{def:AMDRK-scheme}) requests a nonlinear solver, irrespective of whether the derivatives are either calculated exactly through \eqref{eq:FaaDiBruno-timeDer}, recursively obtained with \eqref{eq:timeDerRecursive} or approximated by means of \eqref{eq:timeDerApprox}-\eqref{eq:approxPhiValues}. In case that a single stage $Y={y}^{n,{\color{BlueGreen}l}}$ (with ${\color{BlueGreen}l} = 1, \dots, \mathtt{s}$) is considered, as for DIMDRK schemes, or all the unknown stages are combined into a single vector $Y = \left({y}^{n,1}, \dots, {y}^{n,\mathtt{s}}\right)$ as in the FSMDRK approach, it is possible to write the stage equation(s) \eqref{eq:MDRK-ODE-stages} as
\begin{equation}\label{eq:nonlinear_system}
F(Y) = 0 \, ,
\end{equation}
and then choose any nonlinear solver of preference. Computationally, solving Eq. \eqref{eq:nonlinear_system} is the most expensive portion of the numerical method. Hence, it is vital for the efficiency of the overall method to well understand the behavior of the selected solver. In this paper we use Newton's method, and thus require the Jacobian matrix $F'(Y)$. Given an initial value $Y^{[0]}$, the linearized system
\begin{equation}\label{eq:linearized_newton}
F'(Y^{[i]}) {\Delta Y}^{[i]} = - F(Y^{[i]}), \quad Y^{[i+1]} = Y^{[i]} + {\Delta Y}^{[i]}
\end{equation}
is solved for $i = 0, \dots, N_{\text{iter}}-1$ or until some convergence criteria are satisfied. In this work, criteria are invoked on the residuals,
\begin{equation}\label{eq:criteria_newton}
{\| F(Y^{[i]}) \|}_2 < 10^{-n_{\text{tol}}} \quad \text{or} \quad \frac{{\| F(Y^{[i]}) \|}_2}{{\| F(Y^{[0]}) \|}_2} < 10^{-n_{\text{tol0}}} \, ,
\end{equation}
where $n_{\text{tol}}, n_{\text{tol0}} \in \mathbb{N}$. Under the assumptions that $Y^{[0]}$ is in a neighborhood close enough to the exact solution $Y$, and the Jacobian matrix is nonsingular, Newton's method converges quadratically \mbox{\cite[Theorem 7.1]{QuarteroniNumA}}.
%\begin{equation}
%\|Y^{[i+1]} - Y^* \| \leq C \|Y^{[i]} - Y^* \|^2 \, ,
%\end{equation}
%for some constant $C > 0$.
%In practice, however, $F'(Y^{[i]})$ turns out to be ill-conditioned when the (A)MDRK schemes are applied to stiff problems. As a result, Newton's method is less stable, making it tough to reach an accurate solution.
%
%Obtaining $F'_{\varepsilon}(Y)$ is the most intricate piece of the scheme. And, given the linearized Newton system, preferably it is easy to solve with fast convergence.
%The actual next solution $y^{n+1}$ is then obtained from the stages with \eqref{eq:MDRK-ODE-update}.

\begin{remark}
In what follows, we avoid the superscript index $i$ whenever possible, and instead write $F'(Y)$ or even $F'$.
\end{remark}

%\section{Nonlinear solver algorithm}\label{sec:nonlinearSolver}
%\input{sec_nonlinear}

\section{Newton stability of direct (A)MDRK methods}\label{sec:Newton_Stability}
In order to investigate the conditioning of the Newton Jacobian $F'(Y)$ in the linearized Newton system \eqref{eq:linearized_newton},  we consider the Pareschi-Russo (PR) problem \cite{pareschi2000implicit},  given by
\begin{equation}\label{eq:PR}
 y'_1(t) = -y_2, \qquad y'_2(t) = y_1 + \frac{\sin(y_1) - y_2}{\varepsilon}, \qquad y(0) = \left(\frac{\pi}{2}, 1 \right).
\end{equation}
Let us first verify that the consistency order given in Theorem \ref{thm:AMDRK-consistency} is achieved and compare it with the exact MDRK method as in Def. \ref{def:MDRK-scheme} (using relation \eqref{eq:timeDerRecursive}). In Figure \ref{fig:Pareschi-Russo_eps_1_convergence_AT_direct}, convergence plots are shown for five different MDRK schemes (three Hermite-Birkhoff and two SSP, see \ref{app:ButherTableaux}). Final time is set to $T_{\text{end}} = 5$; the coarsest computation uses $N = 4$ timesteps. To separate convergence order from stiffness, $\varepsilon$ is set to 1.
\begin{figure}[h!]
\pgfplotsset{
	legend style={font=\footnotesize}
}
\centering
\hspace{1cm}\fbox{PR with $\varepsilon=1$}
\begin{subfigure}{.41\textwidth}
	\centering
	\ifthenelse{\boolean{compilefromscratch}}{
        \tikzsetnextfilename{Convergence_DIMDRK_AT_PR_Tend=5_eps=1}  
		\begin{tikzpicture}
		\begin{loglogaxis}[
			title={AMDRK methods},
			height=6cm,
			width=8cm,
			xlabel={$\Delta t$},
			ylabel={$l_2$-error},
		    ymax = 0.5, ymin = 5e-17,
			ytick={1e-1,1e-4,1e-7,1e-10,1e-13,1e-16},
			grid=major,
	%		legend style={nodes={scale=0.95, transform shape}},
	%		legend entries={HB-I2DRK4-2s, HB-I3DRK6-2s, HB-I4DRK8-2s, SSP-I2DRK3-2s, SSP-I2DRK4-5s},
	%		legend cell align=left,
	%		legend pos=outer north east,
			cycle list name = rainbow
			]
			
			\addplot table[
					x expr = {\thisrowno{0}},
					y expr = {\thisrowno{1}}
				] {data/Results_DIMDRK_HB-I2DRK4-2s_AT_p=2_PR_Tend=5_eps=1_splitting=0_damped_newton_backslash_norm=l2.csv};
	
			\addplot table[
					x expr = {\thisrowno{0}},
					y expr = {\thisrowno{1}}
				] {data/Results_DIMDRK_HB-I3DRK6-2s_AT_p=3_PR_Tend=5_eps=1_splitting=0_damped_newton_backslash_norm=l2.csv};
	
			\addplot table[
					x expr = {\thisrowno{0}},
					y expr = {\thisrowno{1}}
				] {data/Results_DIMDRK_HB-I4DRK8-2s_AT_p=4_PR_Tend=5_eps=1_splitting=0_damped_newton_backslash_norm=l2.csv};
	
			\addplot table[
					x expr = {\thisrowno{0}},
					y expr = {\thisrowno{1}}
				] {data/Results_DIMDRK_SSP-I2DRK3-2s_AT_p=1_PR_Tend=5_eps=1_splitting=0_damped_newton_backslash_norm=l2.csv};
	
			\addplot table[
					x expr = {\thisrowno{0}},
					y expr = {\thisrowno{1}}
				] {data/Results_DIMDRK_SSP-I2DRK4-5s_AT_p=2_PR_Tend=5_eps=1_splitting=0_damped_newton_backslash_norm=l2.csv};
	
			\logLogSlopeTriangle{0.18}{0.08}{0.29}{4}{1}{1}{red};
			\logLogSlopeTriangle{0.295}{0.08}{0.23}{6}{-1}{1}{orange};
			\logLogSlopeTriangle{0.605}{0.08}{0.23}{8}{1}{1}{grassGreen};
			\logLogSlopeTriangle{0.09}{0.08}{0.58}{3}{-1}{1}{mintGreen};
	
		\end{loglogaxis}
		\end{tikzpicture}
	  }
	  {
	  \includegraphics{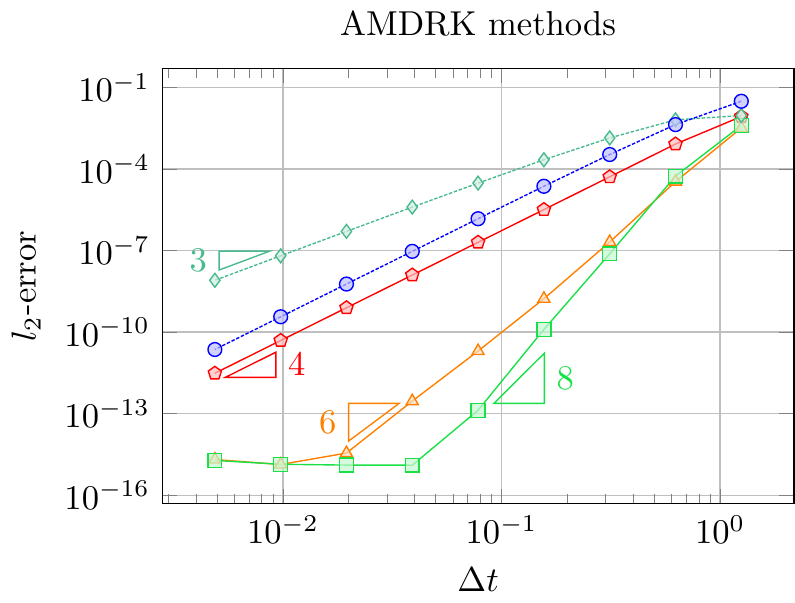}
	  }%
\end{subfigure}%
\begin{subfigure}{0.51\textwidth}
	\centering
	\ifthenelse{\boolean{compilefromscratch}}{
        \tikzsetnextfilename{Convergence_DIMDRK_PR_Tend=5_eps=1}  
		\begin{tikzpicture}
		\begin{loglogaxis}[
			title={MDRK methods},
			height=6cm,
			width=8cm,
			xlabel={$\Delta t$},
			ylabel={\phantom{$l_2$-error}},
			ylabel shift = 2em,
		    ymax = 0.5, ymin = 5e-17,
			ytick={1e-1,1e-4,1e-7,1e-10,1e-13,1e-16},
					ymajorticks=false,
			grid=major,
			legend style={nodes={scale=0.95, transform shape}},
			legend entries={HB-I2DRK4-2s, HB-I3DRK6-2s, HB-I4DRK8-2s, SSP-I2DRK3-2s, SSP-I2DRK4-5s},
			legend columns=3,
			legend to name=convergence-PR-eps1,
			cycle list name = rainbow
			]
			
			\addplot table[
					x expr = {\thisrowno{0}},
					y expr = {\thisrowno{1}}
				] {data/Results_DIMDRK_HB-I2DRK4-2s_PR_Tend=5_eps=1_splitting=0_damped_newton_backslash_norm=l2.csv};
	
			\addplot table[
					x expr = {\thisrowno{0}},
					y expr = {\thisrowno{1}}
				] {data/Results_DIMDRK_HB-I3DRK6-2s_PR_Tend=5_eps=1_splitting=0_damped_newton_backslash_norm=l2.csv};
	
			\addplot table[
					x expr = {\thisrowno{0}},
					y expr = {\thisrowno{1}}
				] {data/Results_DIMDRK_HB-I4DRK8-2s_PR_Tend=5_eps=1_splitting=0_damped_newton_backslash_norm=l2.csv};
	
			\addplot table[
					x expr = {\thisrowno{0}},
					y expr = {\thisrowno{1}}
				] {data/Results_DIMDRK_SSP-I2DRK3-2s_PR_Tend=5_eps=1_splitting=0_damped_newton_backslash_norm=l2.csv};
	
			\addplot table[
					x expr = {\thisrowno{0}},
					y expr = {\thisrowno{1}}
				] {data/Results_DIMDRK_SSP-I2DRK4-5s_PR_Tend=5_eps=1_splitting=0_damped_newton_backslash_norm=l2.csv};
	
			\logLogSlopeTriangle{0.18}{0.08}{0.29}{4}{1}{1}{red};
			\logLogSlopeTriangle{0.295}{0.08}{0.23}{6}{-1}{1}{orange};
			\logLogSlopeTriangle{0.71}{0.08}{0.25}{8}{1}{1}{grassGreen};
			\logLogSlopeTriangle{0.09}{0.08}{0.57}{3}{-1}{1}{mintGreen};
	
		\end{loglogaxis}
		\end{tikzpicture}
	  }
	  {
	  \includegraphics{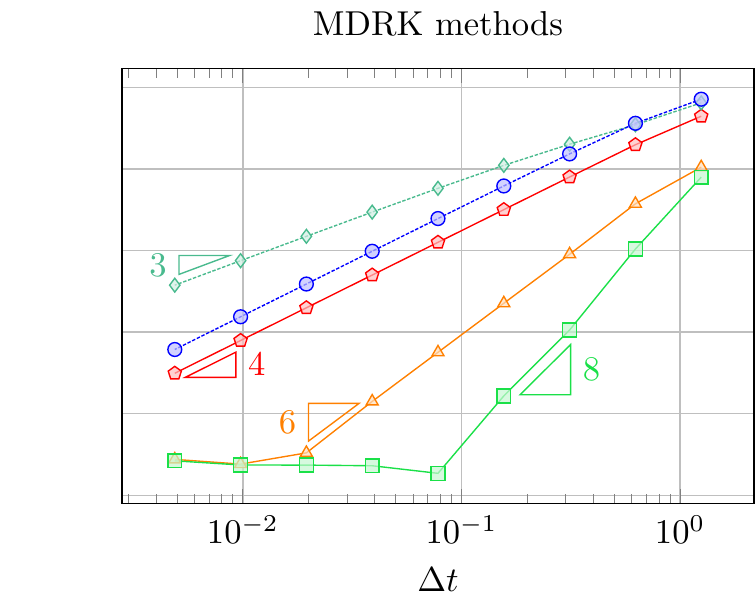}
	  }
\end{subfigure}

\hspace{1.3cm}\ifthenelse{\boolean{compilefromscratch}}{
				\ref{convergence-PR-eps1}
			  }
	  		  {
			  \includegraphics{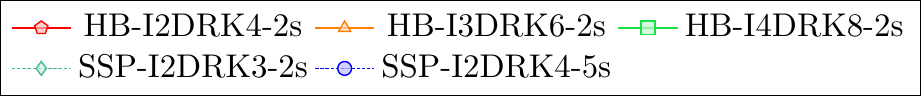}
			  }

	\caption{Pareschi-Russo problem: Convergence of the AMDRK scheme (Def. \ref{def:AMDRK-scheme}) and the MDRK scheme (Def. \ref{def:MDRK-scheme}). The final time is set to $T_{\text{end}}=5$, timesteps start from $N=4$; $\varepsilon=1$. Three Hermite-Birkhoff schemes and two SSP schemes are considered, see \ref{app:ButherTableaux}.}
\label{fig:Pareschi-Russo_eps_1_convergence_AT_direct}	
\end{figure}
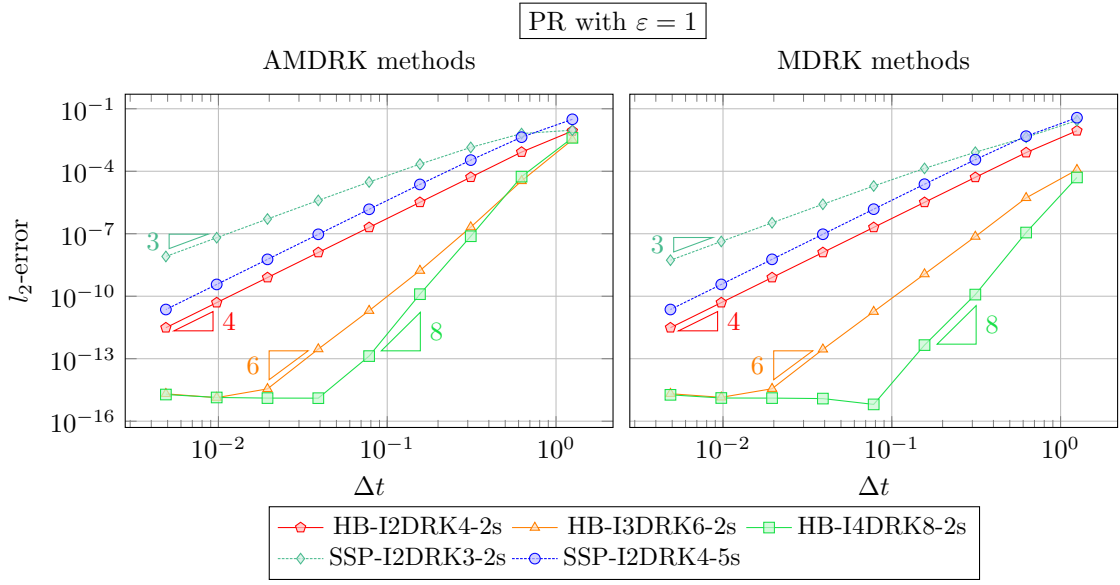
We can clearly see that the AMDRK method achieves the appropriate convergence orders for all the considered schemes.
Also, when compared to their exact MDRK counterpart, differences are barely visible. Only for high values of $\Delta t$ and large orders of consistency, differences are visible.

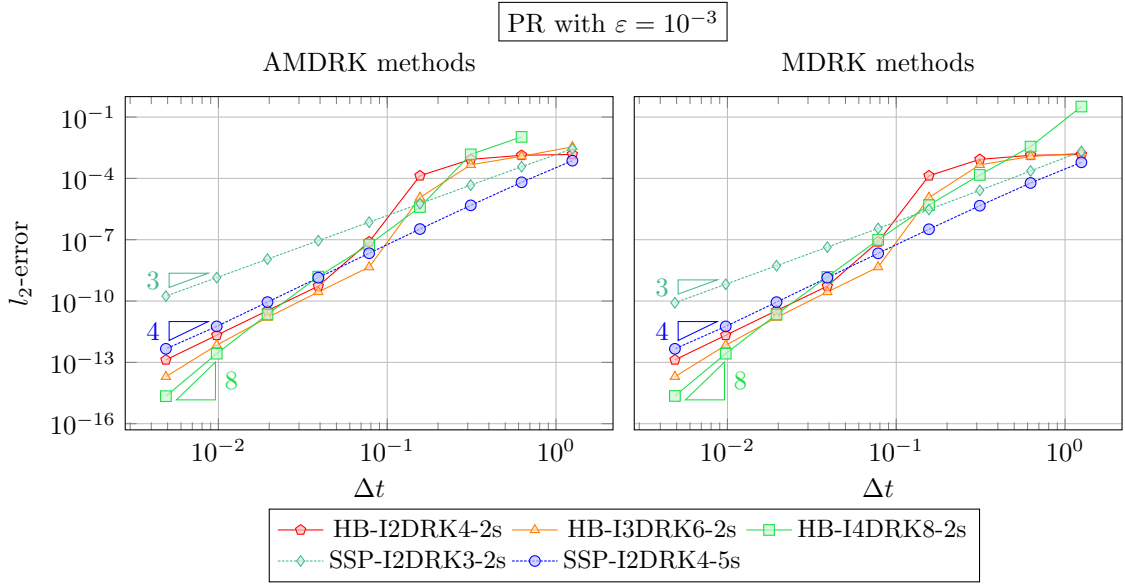
\begin{figure}[h!]
\pgfplotsset{
	legend style={font=\footnotesize}
}
\centering
\hspace{1cm}\fbox{PR with $\varepsilon=10^{-3}$}
\begin{subfigure}{.41\textwidth}
	\ifthenelse{\boolean{compilefromscratch}}{
        \tikzsetnextfilename{Convergence_DIMDRK_AT_PR_Tend=5_eps=0.001} 
		\begin{tikzpicture}
		\begin{loglogaxis}[
			title={AMDRK methods},
			height=6cm,
			width=8cm,
			xlabel={$\Delta t$},
			ylabel={$l_2$-error},
		    ymax = 1, ymin = 5e-17,
			ytick={1e-1,1e-4,1e-7,1e-10,1e-13,1e-16},
			grid=major,
	%		legend style={nodes={scale=0.95, transform shape}},
	%		legend entries={HB-I2DRK4-2s, HB-I3DRK6-2s, HB-I4DRK8-2s, SSP-I2DRK3-2s, SSP-I2DRK4-5s},
	%		legend cell align=left,
	%		legend pos=outer north east,
			cycle list name = rainbow
			]
			
			\addplot table[
					x expr = {\thisrowno{0}},
					y expr = {\thisrowno{1}}
				] {data/Results_DIMDRK_HB-I2DRK4-2s_AT_p=2_PR_Tend=5_eps=0.001_splitting=0_damped_newton_backslash_norm=l2.csv};
	
			\addplot table[
					x expr = {\thisrowno{0}},
					y expr = {\thisrowno{1}}
				] {data/Results_DIMDRK_HB-I3DRK6-2s_AT_p=3_PR_Tend=5_eps=0.001_splitting=0_damped_newton_backslash_norm=l2.csv};
	
			\addplot table[
					x expr = {\thisrowno{0}},
					y expr = {\thisrowno{1}}
				] {data/Results_DIMDRK_HB-I4DRK8-2s_AT_p=4_PR_Tend=5_eps=0.001_splitting=0_damped_newton_backslash_norm=l2.csv};
	
			\addplot table[
					x expr = {\thisrowno{0}},
					y expr = {\thisrowno{1}}
				] {data/Results_DIMDRK_SSP-I2DRK3-2s_AT_p=1_PR_Tend=5_eps=0.001_splitting=0_damped_newton_backslash_norm=l2.csv};
	
			\addplot table[
					x expr = {\thisrowno{0}},
					y expr = {\thisrowno{1}}
				] {data/Results_DIMDRK_SSP-I2DRK4-5s_AT_p=2_PR_Tend=5_eps=0.001_splitting=0_damped_newton_backslash_norm=l2.csv};
	
			\logLogSlopeTriangle{0.09}{0.08}{0.325}{4}{-1}{1}{blue};
	%		\logLogSlopeTriangle{0.295}{0.08}{0.23}{6}{-1}{1}{orange};
			\logLogSlopeTriangle{0.185}{0.08}{0.09}{8}{1}{1}{grassGreen};
			\logLogSlopeTriangle{0.09}{0.08}{0.47}{3}{-1}{1}{mintGreen};
	
		\end{loglogaxis}
		\end{tikzpicture}
	  }
	  {
	  \includegraphics{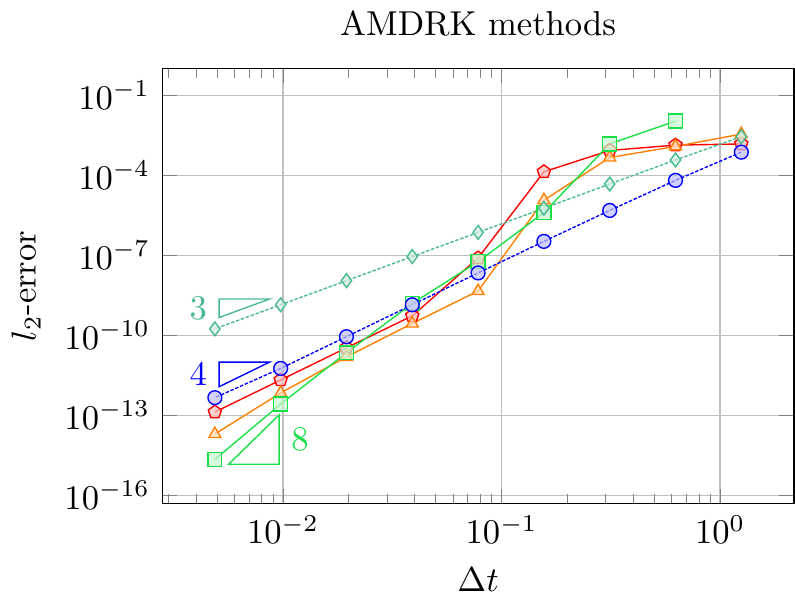}
	  }%
\end{subfigure}%
\begin{subfigure}{0.51\textwidth}
	\centering
	\ifthenelse{\boolean{compilefromscratch}}{
        \tikzsetnextfilename{Convergence_DIMDRK_PR_Tend=5_eps=0.001} 
		\begin{tikzpicture}
		\begin{loglogaxis}[
			title={MDRK methods},
			height=6cm,
			width=8cm,
			xlabel={$\Delta t$},
			ylabel={\phantom{$l_2$-error}},
			ylabel shift = 2em,
		    ymax = 1, ymin = 5e-17,
			ytick={1e-1,1e-4,1e-7,1e-10,1e-13,1e-16},
					ymajorticks=false,
			grid=major,
			legend style={nodes={scale=0.95, transform shape}},
			legend entries={HB-I2DRK4-2s, HB-I3DRK6-2s, HB-I4DRK8-2s, SSP-I2DRK3-2s, SSP-I2DRK4-5s},
			legend columns=3,
			legend to name=convergence-PR-eps1e-3,
			cycle list name = rainbow
			]
			
			\addplot table[
					x expr = {\thisrowno{0}},
					y expr = {\thisrowno{1}}
				] {data/Results_DIMDRK_HB-I2DRK4-2s_PR_Tend=5_eps=0.001_splitting=0_damped_newton_backslash_norm=l2.csv};
	
			\addplot table[
					x expr = {\thisrowno{0}},
					y expr = {\thisrowno{1}}
				] {data/Results_DIMDRK_HB-I3DRK6-2s_PR_Tend=5_eps=0.001_splitting=0_damped_newton_backslash_norm=l2.csv};
	
			\addplot table[
					x expr = {\thisrowno{0}},
					y expr = {\thisrowno{1}}
				] {data/Results_DIMDRK_HB-I4DRK8-2s_PR_Tend=5_eps=0.001_splitting=0_damped_newton_backslash_norm=l2.csv};
	
			\addplot table[
					x expr = {\thisrowno{0}},
					y expr = {\thisrowno{1}}
				] {data/Results_DIMDRK_SSP-I2DRK3-2s_PR_Tend=5_eps=0.001_splitting=0_damped_newton_backslash_norm=l2.csv};
	
			\addplot table[
					x expr = {\thisrowno{0}},
					y expr = {\thisrowno{1}}
				] {data/Results_DIMDRK_SSP-I2DRK4-5s_PR_Tend=5_eps=0.001_splitting=0_damped_newton_backslash_norm=l2.csv};
				
			\logLogSlopeTriangle{0.09}{0.08}{0.325}{4}{-1}{1}{blue};
	%		\logLogSlopeTriangle{0.295}{0.08}{0.23}{6}{-1}{1}{orange};
			\logLogSlopeTriangle{0.185}{0.08}{0.09}{8}{1}{1}{grassGreen};
			\logLogSlopeTriangle{0.09}{0.08}{0.45}{3}{-1}{1}{mintGreen};
	
		\end{loglogaxis}
		\end{tikzpicture}
	  }
	  {
	  \includegraphics{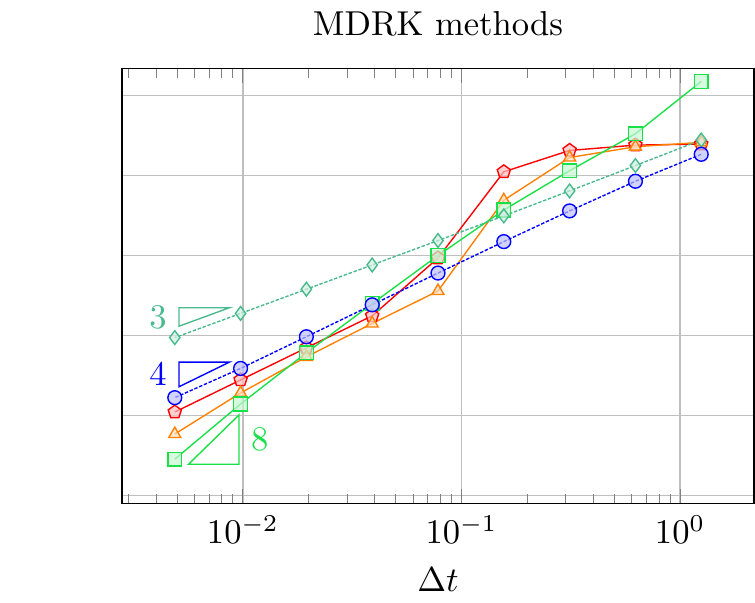}
	  }%
\end{subfigure}

\hspace{1.3cm}\ifthenelse{\boolean{compilefromscratch}}{
				\ref{convergence-PR-eps1e-3}
			  }
	  		  {
			  \includegraphics{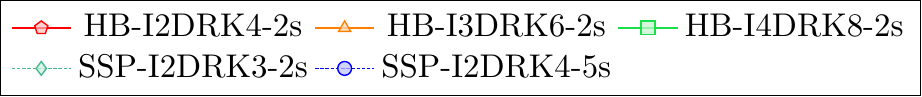}
			  }
	
	\caption{Pareschi-Russo problem: Convergence of the AMDRK scheme (Def. \ref{def:AMDRK-scheme}) and the MDRK scheme (Def. \ref{def:MDRK-scheme}). The final time $T_{\text{end}}=5$, timesteps start from $N=4$; $\varepsilon=10^{-3}$. Three Hermite-Birkhoff schemes and two SSP schemes are considered, see \ref{app:ButherTableaux}. For $N=4$ the HB-I4DRK8-2s scheme diverges, hence no node is shown.}
\label{fig:Pareschi-Russo_eps_1e-3_convergence_AT_direct}	
\end{figure}

When stiffness is increased by decreasing $\varepsilon$ to $\varepsilon = 10^{-3}$, see Figure \ref{fig:Pareschi-Russo_eps_1e-3_convergence_AT_direct}, the same methods as well show convergence for $\Delta t \rightarrow 0$. However, due to order reduction phenomena, it is more difficult to observe the appropriate order here. Above all, for values $\varepsilon \ll {\Delta t}$, the scheme \mbox{HB-I4DRK8-2s} (Table \ref{tab:HB-I4DRK8-2s}) has not properly converged in the Newton iterations; for $N=4$ the AMDRK method diverges immediately, hence there being no node in the left plot of Figure \ref{fig:Pareschi-Russo_eps_1e-3_convergence_AT_direct}, whereas the exact MDRK method shows a large error. 

%Similar results are obtained when the MDRK method is used, without approximating the derivatives $y^{(k)}$.

%We can clearly see that appropriate convergence orders are obtained, especially in case of little stiffness when $\varepsilon =1$. For smaller values of $\varepsilon$, the Hermite-Birkhoff schemes suffer from order reduction phenomena. Nevertheless, aside from order reduction, we can see that the (A)MDRK method behaves properly in the error propagation.

\subsection{Numerical observations of the Newton conditioning}
So, even though all three approaches \eqref{eq:FaaDiBruno-timeDer}, \eqref{eq:timeDerRecursive} and \eqref{eq:timeDerApprox}-\eqref{eq:approxPhiValues} for calculating the derivatives $y^{(k)}$ yield valid high-order algorithms, numerically we observe stability issues for stiff problems $\varepsilon \ll {\Delta t}$. More specifically, when $\varepsilon \ll {\Delta t}$, the Newton Jacobian $F'(Y)$ is badly conditioned. In Table \ref{tab:newton_cond_DIMDRK_ImplTaylor_order=3_AT_p=1_PR_Tend=1} we display the arithmetic mean of the condition numbers in the 1-norm w.r.t. the Newton iterations,
\[
\mu\!\left(\text{cond}(F')\right) := \frac{\sum\limits_{i=1}^{N_{\text{iter}}}{{\text{cond}(F'(Y^{[i]}))}}}{N_{\text{iter}}} \, ,
\]
which we have obtained from solving Eq.~\eqref{eq:PR} with the approximate implicit Taylor method of order $\mathtt{r}=3$ for different values of $\varepsilon$. To account for large timesteps, only a single step $N=1$ of size $\Delta t = 1$ was applied. Newton tolerances, Eq.~\eqref{eq:criteria_newton}, were set to $10^{-12}$ under a maximum of 10000 iterations.

In order to put the obtained condition numbers into perspective, the empirical orders w.r.t. $\varepsilon$
\begin{equation}\label{eq:EO}
	\text{EO}_{\varepsilon} := \frac{\log\!\left(\frac{\mu\!\left(\text{cond}(F'_{\varepsilon})\right)}{\mu\!\left(\text{cond}(F'_{10\varepsilon})\right)}\right)}{\log(10)}
	%\bigg\slash \log(10) \, ,
\end{equation}
are additionally computed (where $F'_{\varepsilon}$ denotes $F'$ for a particular value of $\varepsilon$). In this case, the experimental order seems to equal the order ($\mathtt{r}=3$ here) of the implicit Taylor method. And in fact, the same behavior is observed for any amount of derivatives $\mathtt{r}$ used. That means, we numerically observe the asymptotic behavior
\begin{equation}
\text{cond}(F') = \mathcal{O}(\varepsilon^{-\mathtt{r}})
\end{equation}
to hold true for any order of the implicit Taylor scheme. As a result of bad conditioning, in Table \ref{tab:newton_cond_DIMDRK_ImplTaylor_order=3_AT_p=1_PR_Tend=1} we can therefore observe that the (A)MDRK methods do not converge for Pareschi-Russo's equation \eqref{eq:PR} when $\varepsilon = 10^{-5}$.
In general, when considering any DIMDRK scheme, the same derivative-dependent behavior holds true for any implicit stage. In Figure \ref{fig:CondNr-DIMDRK-direct-Pareschi-Russo} we plot the condition number for the final stage of different DIMDRK schemes.
%Explicit stages are either solved explicitly, or we have $\text{cond}(F') = 1$ owing to $F' = I_M$ (the identity matrix).

The FSMDRK implementation shows a behavior similar to the DIMDRK implementation, see Figure~\ref{fig:CondNr-FSMDRK_direct-Pareschi-Russo}. An intuitive reasoning can be found in the block-matrix structure of the Newton Jacobian. Due to the stages all being solved for at once, the Jacobian reads as
\begin{align}\label{eq:jacobianFSMDRK-Direct}
F'(Y) = \frac{\partial}{\partial Y}\left[\def\arraystretch{1.3}
\begin{array}{c}
F_{1} \\
\vdots \\
F_{{\mathtt{s}}}
\end{array}\right]
 = \left[\def\arraystretch{1.3}
\begin{array}{c|c|c}
\partial_{Y_1}F_{1} & \hdots & \partial_{Y_{\mathtt{s}}}F_{1} \\ \hline 
\vdots & \ddots & \vdots \\ \hline
\partial_{Y_1}F_{{\mathtt{s}}} & \hdots & \partial_{Y_{\mathtt{s}}}F_{{\mathtt{s}}} 
\end{array}\right] \, ,
\end{align}
with $\partial_{Y_{{\color{VioletRed}\nu}}}F_{{\color{BlueGreen}l}}$ the partial derivative of the ${\color{BlueGreen}l}$-th stage equation w.r.t. the ${\color{VioletRed}\nu}$-th stage variable $Y_{{\color{VioletRed}\nu}}$. This block-structure assures a dependency of $\cond(F')$ on the conditioning $\cond(\partial_{Y_{{\color{VioletRed}\nu}}}F_{{\color{BlueGreen}l}})$ of the separate blocks. Hence, if $\cond(\partial_{Y_{{\color{VioletRed}\nu}}}F_{{\color{BlueGreen}l}}) =  \mathcal{O}(\varepsilon^{-\mathtt{r}})$, as is often the case from what is observed in the DIMDRK implementation for an $\mathtt{r}$-derivative scheme, the complete Jacobian likely also behaves at least as $\mathcal{O}(\varepsilon^{-\mathtt{r}})$.

%The FSMDRK implementation can yet show worse behavior than the DIMDRK implementation, see Figure \ref{fig:CondNr-FSMDRK_direct-Pareschi-Russo}. Comparing the HB-I4DRK8-2s scheme (Table \ref{tab:HB-I4DRK8-2s}) in both cases, the condition number of the FSMDRK implementation behaves at least as $\mathcal{O}(\varepsilon^{-5})$, which is at minimum one order higher than $\mathcal{O}(\varepsilon^{-4})$ in the DIMDRK implementation. 

\begin{remark} 
If the first stage is explicit, s.t. ${y}^{n,1} = y^n$, we make the assumption that the FSMDRK approach instead solves for $Y = \left({y}^{n,2}, \dots, {y}^{n,\mathtt{s}}\right)$. The Hermite-Birkhoff schemes (Tables \ref{tab:HB-I2DRK4-2s}-\ref{tab:HB-I4DRK8-2s}) are good examples of RK-schemes with an explicit first stage.
\end{remark}

%Explicit stages (in the sense that $F_{{\color{BlueGreen}l}}$ is independent of $Y_{{\color{BlueGreen}l}}$) now manifest as block-diagonal elements $\partial_{Y_{{\color{BlueGreen}l}}}F_{{\color{BlueGreen}l}} = I_M$. In contrast, however, such $I_M$ as a block does play a role on $\text{cond}(F')$, given that the inverse matrix ${F'}^{-1}(Y)$ will link the $\varepsilon$-dependency of different stages to one another.

Moreover, when we consider other problems with a similar dependency on a small non-dimensional value $\varepsilon$ as for the PR problem \eqref{eq:PR}, then as well $\mathcal{O}(\varepsilon^{-\mathtt{r}})$ behavior is observed. Similar condition number plots alike the ones in Figs. \ref{fig:CondNr-DIMDRK-direct-Pareschi-Russo} and \ref{fig:CondNr-FSMDRK_direct-Pareschi-Russo} have been obtained for van der Pol and Kaps~problems described in \cite{SealSchuetz19}.
%(for more details on the van der Pol problem we refer to \cite{HaiWan}, for Kaps' problem \cite{Kaps1981} and \cite[Eqs. (6.5)]{Kaps1985} provide more clarification).
%In \ref{app:vdP} (van der Pol problem) and \ref{app:kaps} (Kaps' problem), similar condition number plots alike the ones constructed for the PR problem are shown.

\begin{remark}
Albeit mathematically equivalent, in Table \ref{tab:newton_cond_DIMDRK_ImplTaylor_order=3_AT_p=1_PR_Tend=1} we can numerically see different results between using exact Jacobians (in the sense that we apply Fa\'a die Bruno's formula) and using recursive formulas for calculating the derivatives $y^{(k)}$.
%Before delving deeper into the conditioning of other two-variable ODE systems like the PR problem, we mention other noteworthy features of the data in Table  \ref{tab:newton_cond_DIMDRK_ImplTaylor_order=3_AT_p=1_PR_Tend=1}.
%\begin{itemize}
%\item On the same note, for $\varepsilon = 10^{-1}$, Newton did not converge when using EJ. Rerunning the same test case with a larger maximum of total iterations reveiled that 1655 iterations were needed for convergence. Likely, the initial condition $y(0) = \left(\frac{\pi}{2}, 1 \right)$ is suboptimal here as starting values for Newtons method. In \ref{app:NewtonData} a visualization (Figure \ref{fig:newton_cond_DIMDRK_ImplTaylor_order=3_EJ_PR_Tend=1.25}) of the condition numbers for the first 170 iterations is given with $\varepsilon \in \left\{10^{-1},10^{-2},10^{-3},10^{-4}\right\}$. Only when ${\varepsilon = 10^{-1}}$ the condition number shows heavy zig-zag motion with increasing iterations. If instead the approximate starting values $\left(1.57, 1 \right)$ are used, only 7 iterations are needed for convergence (see also Table \ref{tab:NewtonStartPR}).
%\item For $\varepsilon = 1$  we obtain zero residuals. The final residuals here are smaller than the machine precision, and hence zero values are returned.
%\end{itemize}
\end{remark}

\begin{table}[h!]
\centering
	\caption{Newton statistics of the implicit Taylor method of order 3 applied for a single timestep ($N=1$) of size $T_{\text{end}} = 1$ ($\Delta t = 1$) to the PR problem \eqref{eq:PR}. Tolerances, Eq. \eqref{eq:criteria_newton}, were set to $10^{-12}$ under a maximum of 10000 iterations. \underline{Left}: The amount of iterations $N_{\text{iter}}$ and the average condition number in the 1-norm of the Newton Jacobian $\mu\!\left(\text{cond}(F')\right)$. $\text{EO}_{\varepsilon}$ is the experimental order of the average w.r.t. $\varepsilon$ according to Eq. \eqref{eq:EO}. For $\varepsilon = 10^{-5}$, none of the methods converged, with A-Direct diverging at 702 iterations. \underline{Right}: The first 330 iterations of A-Direct. We can observe that for $\varepsilon = 10^{-5}$ the scheme becomes unstable and diverges eventually at iteration 702.}
\label{tab:newton_cond_DIMDRK_ImplTaylor_order=3_AT_p=1_PR_Tend=1}
\begin{minipage}{0.45\textwidth}
\begin{tabular}{c|c|c|l|l}
 Method & $\varepsilon$ & $N_{\text{iter}}$ & $\mu\!\left(\text{cond}(F')\right)$ & $\text{EO}_{\varepsilon}$ \\ \hline
 \multirow{7}{*}{A-Direct} & $1$ & $5$ & ~$4.45 \cdot 10^0$ & \\
		             & $10^{-1}$ & $6$ & ~$2.89 \cdot 10^2$ & $1.81$ \\
		             & $10^{-2}$ & $34$ & ~$2.69 \cdot 10^5$ & $2.97$ \\
		             & $10^{-3}$ & $75$ & ~$2.71 \cdot 10^8$ & $3.00$ \\
		             & $10^{-4}$ & $226$ & ~$2.51 \cdot 10^{11}$ & $2.97$ \\
		         	 & {\color{red}$10^{-5}$} & {\color{red}$702$} & ~$8.20 \cdot 10^{14}$ & $3.51$ \\ \hline
    \multirow{7}{*}{EJ-Direct} & $1$ & $5$ & ~$3.44 \cdot 10^0$ & \\
    		         & $10^{-1}$ & $7$ & ~$2.78 \cdot 10^2$ & $1.91$ \\
		             & $10^{-2}$ & $124$ & ~$3.00 \cdot 10^5$ & $3.03$ \\
		             & $10^{-3}$ & $45$ & ~$5.57 \cdot 10^8$ & $3.27$ \\
		             & {\color{red}$10^{-4}$} & {\color{red}$10000$} & ~$6.58 \cdot 10^{10}$ & $2.07$ \\
		         	 & {\color{red}$10^{-5}$} & {\color{red}$10000$} & ~$2.70 \cdot 10^{13}$ & $2.61$ \\ \hline
    \multirow{7}{*}{rec-Direct} & $1$ & $5$ & ~$3.44 \cdot 10^0$ & \\
   		             & $10^{-1}$ & $7$ & ~$2.78 \cdot 10^2$ & $1.91$ \\
		             & $10^{-2}$ & $124$ & ~$3.00 \cdot 10^5$ & $3.03$ \\
		             & $10^{-3}$ & $45$ & ~$5.57 \cdot 10^8$ & $3.27$ \\
		             & {\color{red}$10^{-4}$} & {\color{red}$10000$} & ~$6.32 \cdot 10^{10}$ & $2.05$ \\
		         	 & {\color{red}$10^{-5}$} & {\color{red}$10000$} & ~$3.31 \cdot 10^{13}$ & $2.72$
\end{tabular}%
\end{minipage}%
\hspace{1em}
\begin{minipage}{.45\textwidth}
\pgfplotsset{
	legend style={font=\normalsize}
}
\centering
	\ifthenelse{\boolean{compilefromscratch}}{
        \tikzsetnextfilename{newton_cond_DIMDRK_ImplTaylor_order=3_AT_p=1_PR_Tend=1} 
		\begin{tikzpicture}
		\begin{semilogyaxis}[
			title={A-Direct: $\text{cond}(F')$},
			width=7cm,
			height=8.1cm,
			xlabel={iteration},
			xmax = 330,
			ymajorgrids,
			legend style={nodes={scale=0.85, transform shape}},
			legend entries={$\varepsilon=10^{-2}$,$\varepsilon=10^{-3}$,$\varepsilon=10^{-4}$,$\varepsilon=10^{-5}$},
			legend cell align=left,
			legend pos=south east,
			cycle list name=color list
			]
			
	%		\addplot table[
	%				x expr = {\thisrowno{0}},
	%				y expr = {\thisrowno{1}}
	%			] {data/newton_cond_DIMDRK_ImplTaylor_order=3_AT_p=1_PR_Tend=1_eps=1_splitting=0_damped_newton_backslash_Nt=1_timestep1_stage1.csv};		
	%		
	%		\addplot table[
	%				x expr = {\thisrowno{0}},
	%				y expr = {\thisrowno{1}}
	%			] {data/newton_cond_DIMDRK_ImplTaylor_order=3_AT_p=1_PR_Tend=1_eps=0.1_splitting=0_damped_newton_backslash_Nt=1_timestep1_stage1.csv};
	
			\addplot+[line width=1.1pt] table[
					x expr = {\coordindex+1},
					y expr = {\thisrow{sym_cond1}}
				] {data/newton_cond_DIMDRK_ImplTaylor_order=3_AT_p=1_PR_Tend=1_eps=0.01_splitting=0_damped_newton_backslash_Nt=1_timestep1_stage1.csv};
		
			\addplot+[line width=0.8pt, mintGreen] table[
					x expr = {\coordindex+1},
					y expr = {\thisrow{sym_cond1}}
				] {data/newton_cond_DIMDRK_ImplTaylor_order=3_AT_p=1_PR_Tend=1_eps=0.001_splitting=0_damped_newton_backslash_Nt=1_timestep1_stage1.csv};
				
			\addplot+[line width=0.5pt] table[
					x expr = {\coordindex+1},
					y expr = {\thisrow{sym_cond1}}
				] {data/newton_cond_DIMDRK_ImplTaylor_order=3_AT_p=1_PR_Tend=1_eps=0.0001_splitting=0_damped_newton_backslash_Nt=1_timestep1_stage1.csv};
				
			\pgfplotsset{cycle list shift=1}	
			\addplot+[line width=0.2pt] table[
					x expr = {\coordindex+1},
					y expr = {\thisrow{sym_cond1}}
				] {data/newton_cond_DIMDRK_ImplTaylor_order=3_AT_p=1_PR_Tend=1_eps=1e-05_splitting=0_damped_newton_backslash_Nt=1_timestep1_stage1.csv};						
		\end{semilogyaxis}
		\end{tikzpicture}
	  }
	  {
	  \includegraphics{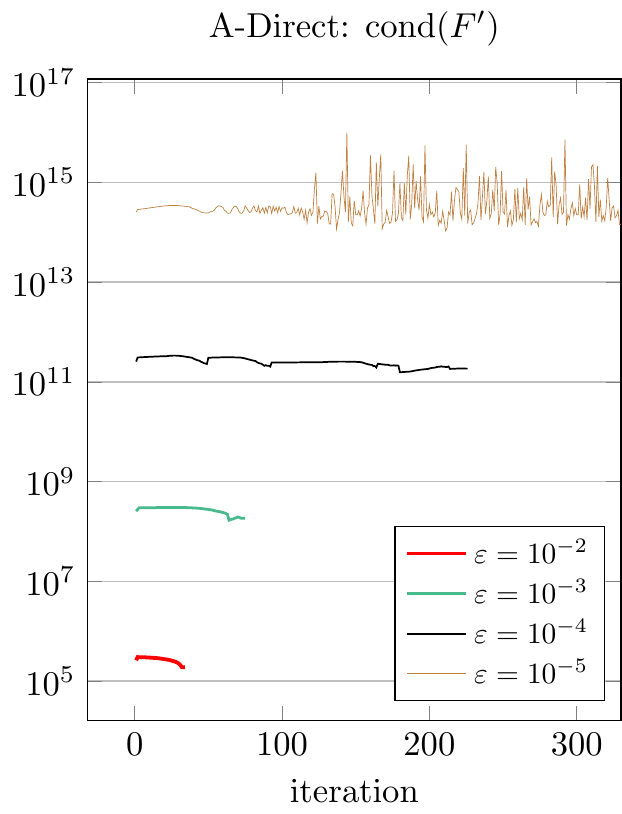}
	  }%
	\end{minipage}
\end{table}

%%%%%%%%%%%%%%%%%%%%%%%%%%%%%%%%%%%%%%%
%% DIMDRK (direct) -- Pareschi-Russo %%
%%%%%%%%%%%%%%%%%%%%%%%%%%%%%%%%%%%%%%%
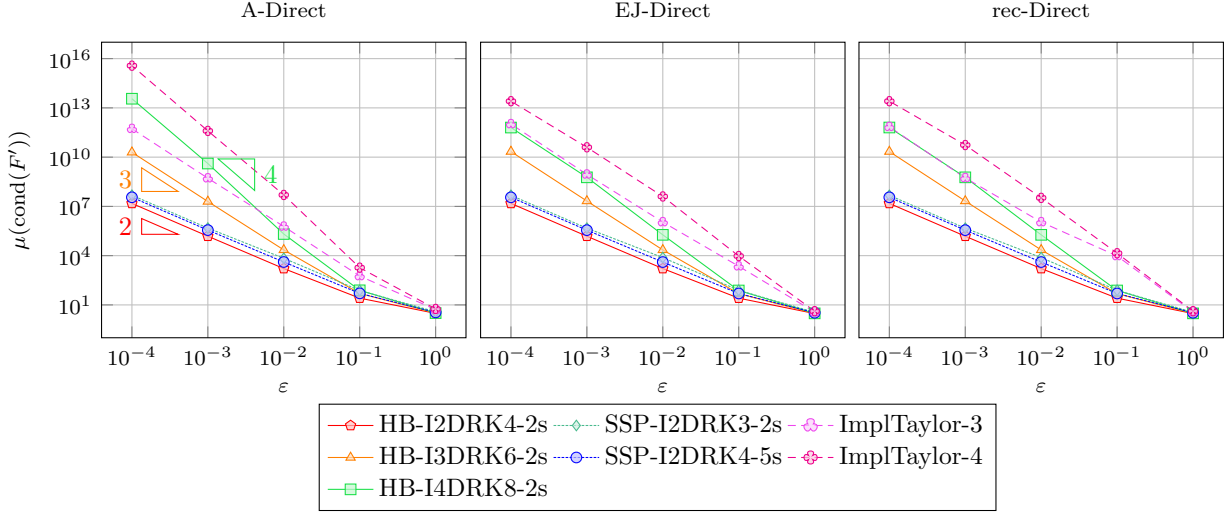
\begin{figure}[t]
\pgfplotsset{
	title style={font=\footnotesize},
	tick label style={font=\footnotesize},
	label style={font=\footnotesize},
	legend style={font=\footnotesize}
}
\centering
\begin{subfigure}{.325\textwidth}
  \centering
  \ifthenelse{\boolean{compilefromscratch}}{
        \tikzsetnextfilename{DataCondNr_DIMDRK-Direct_AT_PR_Tend=1.25} 
		\begin{tikzpicture}
		\begin{loglogaxis}[
			title={A-Direct},
			height=5.5cm,
			xlabel={$\varepsilon$},
			ylabel={$\mu\!\left(\text{cond}(F')\right)$},
			xtick={1e0,1e-1,1e-2,1e-3,1e-4},
	        ymax = 1e17, ymin = 1e-1,
			ytick={1e1,1e4,1e7,1e10,1e13,1e16},
			grid=major,
			legend style={nodes={scale=0.95, transform shape}},
	%		legend entries={HB-I2DRK4-2s, HB-I3DRK6-2s, SSP-I2DRK3-2s, SSP-I2DRK4-5s},
	%		legend pos=north east,
			cycle list name = rainbow2
			]
			\addplot table[
					x expr = {\thisrow{epsilon}},
					y expr = {\thisrow{mean_symcond1}}
				]{data/DataCondNr_DIMDRK_HB-I2DRK4-2s_AT_p=2_PR_Tend=1.25_splitting=0_damped_newton_backslash_Nt=1_timestep1_stage2.csv};
		    \addplot table[
					x expr = {\thisrow{epsilon}},
					y expr = {\thisrow{mean_symcond1}}
				]{data/DataCondNr_DIMDRK_HB-I3DRK6-2s_AT_p=3_PR_Tend=1.25_splitting=0_damped_newton_backslash_Nt=1_timestep1_stage2.csv};
		    \addplot table[
					x expr = {\thisrow{epsilon}},
					y expr = {\thisrow{mean_symcond1}}
				]{data/DataCondNr_DIMDRK_HB-I4DRK8-2s_AT_p=4_PR_Tend=1.25_splitting=0_damped_newton_backslash_Nt=1_timestep1_stage2.csv};
		    \addplot table[
					x expr = {\thisrow{epsilon}},
					y expr = {\thisrow{mean_symcond1}}
				]{data/DataCondNr_DIMDRK_SSP-I2DRK3-2s_AT_p=1_PR_Tend=1.25_splitting=0_damped_newton_backslash_Nt=1_timestep1_stage2.csv};
			\addplot table[
					x expr = {\thisrow{epsilon}},
					y expr = {\thisrow{mean_symcond1}}
				]{data/DataCondNr_DIMDRK_SSP-I2DRK4-5s_AT_p=2_PR_Tend=1.25_splitting=0_damped_newton_backslash_Nt=1_timestep1_stage5.csv};
			\addplot table[
					x expr = {\thisrow{epsilon}},
					y expr = {\thisrow{mean_symcond1}}
				]{data/DataCondNr_DIMDRK_ImplTaylor_order=3_AT_p=1_PR_Tend=1.25_splitting=0_damped_newton_backslash_Nt=1_timestep1_stage1.csv};
		    \addplot table[
					x expr = {\thisrow{epsilon}},
					y expr = {\thisrow{mean_symcond1}}
				]{data/DataCondNr_DIMDRK_ImplTaylor_order=4_AT_p=2_PR_Tend=1.25_splitting=0_damped_newton_backslash_Nt=1_timestep1_stage1.csv};
		
			\logLogSlopeTriangle{0.21}{0.1}{0.35}{2}{1}{-1}{red};
			\logLogSlopeTriangle{0.21}{0.1}{0.495}{3}{1}{-1}{orange};
			\logLogSlopeTriangle{0.32}{0.1}{0.605}{4}{-1}{-1}{grassGreen};	
		
		\end{loglogaxis}
		\end{tikzpicture}
	  }
	  {
	  \includegraphics{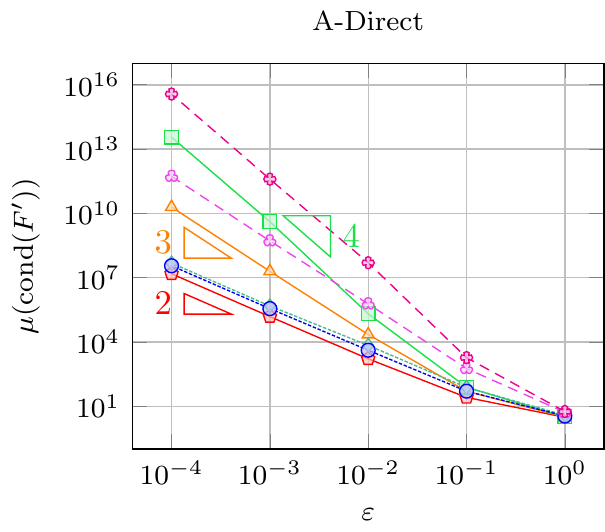}
	  }%	
\end{subfigure}%
\begin{subfigure}{0.33\textwidth}
  \centering
  \ifthenelse{\boolean{compilefromscratch}}{
        \tikzsetnextfilename{DataCondNr_DIMDRK-Direct_EJ_PR_Tend=1.25} 
		\begin{tikzpicture}
		\begin{loglogaxis}[
			title={EJ-Direct},
			height=5.5cm,
			xlabel={$\varepsilon$},
			ylabel={\phantom{$\text{cond}(F')$}},
			ylabel shift = 1.3em,
		    xtick={1e0,1e-1,1e-2,1e-3,1e-4},
	        ymax = 1e17, ymin = 1e-1,
			ytick={1e1,1e4,1e7,1e10,1e13,1e16},
			ymajorticks=false,
			grid=major,
			legend style={nodes={scale=0.95, transform shape}},
			transpose legend,
			legend columns=3,
			legend to name=DIMDRKschemes-PR,
			cycle list name = rainbow2
			]
			\addplot table[
					x expr = {\thisrow{epsilon}},
					y expr = {\thisrow{mean_symcond1}}
				]{data/DataCondNr_DIMDRK_HB-I2DRK4-2s_EJ_PR_Tend=1.25_splitting=0_damped_newton_backslash_Nt=1_timestep1_stage2.csv};
		    \addlegendentry{HB-I2DRK4-2s};
		    
		    \addplot table[
					x expr = {\thisrow{epsilon}},
					y expr = {\thisrow{mean_symcond1}}
				]{data/DataCondNr_DIMDRK_HB-I3DRK6-2s_EJ_PR_Tend=1.25_splitting=0_damped_newton_backslash_Nt=1_timestep1_stage2.csv};
			\addlegendentry{HB-I3DRK6-2s};
			
			\addplot table[
					x expr = {\thisrow{epsilon}},
					y expr = {\thisrow{mean_symcond1}}
				]{data/DataCondNr_DIMDRK_HB-I4DRK8-2s_EJ_PR_Tend=1.25_splitting=0_damped_newton_backslash_Nt=1_timestep1_stage2.csv};
			\addlegendentry{HB-I4DRK8-2s};
			
			\addplot table[
					x expr = {\thisrow{epsilon}},
					y expr = {\thisrow{mean_symcond1}}
				]{data/DataCondNr_DIMDRK_SSP-I2DRK3-2s_EJ_PR_Tend=1.25_splitting=0_damped_newton_backslash_Nt=1_timestep1_stage2.csv};
		    \addlegendentry{SSP-I2DRK3-2s};
		   
		    \addplot table[
					x expr = {\thisrow{epsilon}},
					y expr = {\thisrow{mean_symcond1}}
				]{data/DataCondNr_DIMDRK_SSP-I2DRK4-5s_EJ_PR_Tend=1.25_splitting=0_damped_newton_backslash_Nt=1_timestep1_stage5.csv};
			\addlegendentry{SSP-I2DRK4-5s};
			
			% Add empty legend to fix the legend layout.
	        \addlegendimage{empty legend}
	        \addlegendentry{}		
			
			\addplot table[
					x expr = {\thisrow{epsilon}},
					y expr = {\thisrow{mean_symcond1}}
				]{data/DataCondNr_DIMDRK_ImplTaylor_order=3_EJ_PR_Tend=1.25_splitting=0_damped_newton_backslash_Nt=1_timestep1_stage1.csv};
			\addlegendentry{ImplTaylor-3};
			
			\addplot table[
					x expr = {\thisrow{epsilon}},
					y expr = {\thisrow{mean_symcond1}}
				]{data/DataCondNr_DIMDRK_ImplTaylor_order=4_EJ_PR_Tend=1.25_splitting=0_damped_newton_backslash_Nt=1_timestep1_stage1.csv};
			\addlegendentry{ImplTaylor-4};
		
		\end{loglogaxis}
		\end{tikzpicture}
	  }
	  {
	  \includegraphics{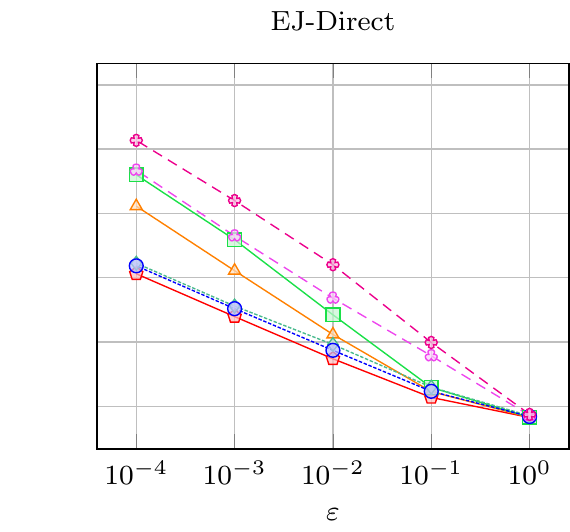}
	  }%	
\end{subfigure}%
\begin{subfigure}{0.325\textwidth}
  \centering
  \ifthenelse{\boolean{compilefromscratch}}{
        \tikzsetnextfilename{DataCondNr_DIMDRK-Direct_rec_PR_Tend=1.25} 
		\begin{tikzpicture}
		\begin{loglogaxis}[
			title={rec-Direct},
			height=5.5cm,
			xlabel={$\varepsilon$},
			ylabel={\phantom{$\text{cond}(F')$}},
			xtick={1e0,1e-1,1e-2,1e-3,1e-4},
	        ymax = 1e17, ymin = 1e-1,
			ytick={1e1,1e4,1e7,1e10,1e13,1e16},
			ymajorticks=false,
			grid=major,
			legend style={nodes={scale=0.95, transform shape}},
			cycle list name = rainbow2
			]
			\addplot table[
					x expr = {\thisrow{epsilon}},
					y expr = {\thisrow{mean_symcond1}}
				]{data/DataCondNr_DIMDRK_HB-I2DRK4-2s_PR_Tend=1.25_splitting=0_damped_newton_backslash_Nt=1_timestep1_stage2.csv};
		    \addplot table[
					x expr = {\thisrow{epsilon}},
					y expr = {\thisrow{mean_symcond1}}
				]{data/DataCondNr_DIMDRK_HB-I3DRK6-2s_PR_Tend=1.25_splitting=0_damped_newton_backslash_Nt=1_timestep1_stage2.csv};
			\addplot table[
					x expr = {\thisrow{epsilon}},
					y expr = {\thisrow{mean_symcond1}}
				]{data/DataCondNr_DIMDRK_HB-I4DRK8-2s_PR_Tend=1.25_splitting=0_damped_newton_backslash_Nt=1_timestep1_stage2.csv};
			\addplot table[
					x expr = {\thisrow{epsilon}},
					y expr = {\thisrow{mean_symcond1}}
				]{data/DataCondNr_DIMDRK_SSP-I2DRK3-2s_PR_Tend=1.25_splitting=0_damped_newton_backslash_Nt=1_timestep1_stage2.csv};
		    \addplot table[
					x expr = {\thisrow{epsilon}},
					y expr = {\thisrow{mean_symcond1}}
				]{data/DataCondNr_DIMDRK_SSP-I2DRK4-5s_PR_Tend=1.25_splitting=0_damped_newton_backslash_Nt=1_timestep1_stage5.csv};
			\addplot table[
					x expr = {\thisrow{epsilon}},
					y expr = {\thisrow{mean_symcond1}}
				]{data/DataCondNr_DIMDRK_ImplTaylor_order=3_PR_Tend=1.25_splitting=0_damped_newton_backslash_Nt=1_timestep1_stage1.csv};
			\addplot table[
					x expr = {\thisrow{epsilon}},
					y expr = {\thisrow{mean_symcond1}}
				]{data/DataCondNr_DIMDRK_ImplTaylor_order=4_PR_Tend=1.25_splitting=0_damped_newton_backslash_Nt=1_timestep1_stage1.csv};
		\end{loglogaxis}
		\end{tikzpicture}
	  }
	  {
	  \includegraphics{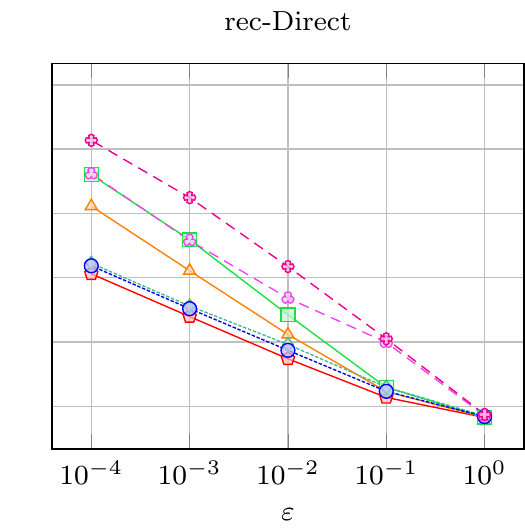}
	  }%
\end{subfigure}
\vspace{-2mm}
\hspace{1.2cm}\ifthenelse{\boolean{compilefromscratch}}{
				\ref{DIMDRKschemes-PR}
			  }
	  		  {
			  \includegraphics{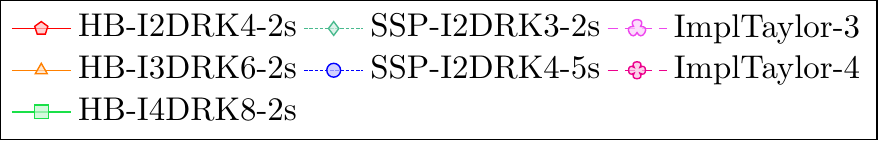}
			  }

\caption{DIMDRK schemes applied as a Direct method (Figure~\ref{fig:AMDRK-scheme}) to the PR problem~\eqref{eq:PR}. The average condition number in the 1-norm of the Newton Jacobian obtained from the last RK-stage is shown for different values of $\varepsilon$.  The behavior $\text{cond}(F') = \mathcal{O}(\varepsilon^{-\mathtt{r}})$ is observed, $\mathtt{r}$ being the amount of derivatives. A single timestep ($N=1$) of size $T_{\text{end}} = 1.25$ ($\Delta t = 1.25$) has been considered with tolerances, Eq.~\eqref{eq:criteria_newton}, set to $10^{-12}$ under a maximum of 1000 iterations.}
\label{fig:CondNr-DIMDRK-direct-Pareschi-Russo}
\end{figure}

\begin{figure}[h!]
\pgfplotsset{
	title style={font=\footnotesize},
	tick label style={font=\footnotesize},
	label style={font=\footnotesize},
	legend style={font=\footnotesize}
}
\centering
\begin{subfigure}{.325\textwidth}
  \centering
  \ifthenelse{\boolean{compilefromscratch}}{
        \tikzsetnextfilename{DataCondNr_FSMDRK-Direct_AT_PR_Tend=1.25} 
		\begin{tikzpicture}
		\begin{loglogaxis}[
			title={A-Direct},
			height=5.5cm,
			xlabel={$\varepsilon$},
			ylabel={$\mu\!\left(\text{cond}(F')\right)$},
			xtick={1e0,1e-1,1e-2,1e-3,1e-4},
	        ymax = 1e17, ymin = 1e-1,
			ytick={1e1,1e4,1e7,1e10,1e13,1e16},
			grid=major,
			legend style={nodes={scale=0.95, transform shape}},
	%		legend entries={HB-I2DRK4-2s, HB-I3DRK6-2s, SSP-I2DRK3-2s, SSP-I2DRK4-5s},
	%		legend pos=north east,
			cycle list name = rainbow
			]
			\addplot table[
					x expr = {\thisrow{epsilon}},
					y expr = {\thisrow{mean_symcond1}}
				]{data/DataCondNr_FSMDRK_HB-I2DRK4-2s_AT_p=2_PR_Tend=1.25_splitting=0_damped_newton_backslash_Nt=1_timestep1.csv};
		    \addplot table[
					x expr = {\thisrow{epsilon}},
					y expr = {\thisrow{mean_symcond1}}
				]{data/DataCondNr_FSMDRK_HB-I3DRK6-2s_AT_p=3_PR_Tend=1.25_splitting=0_damped_newton_backslash_Nt=1_timestep1.csv};
		    \addplot table[
					x expr = {\thisrow{epsilon}},
					y expr = {\thisrow{mean_symcond1}}
				]{data/DataCondNr_FSMDRK_HB-I4DRK8-2s_AT_p=4_PR_Tend=1.25_splitting=0_damped_newton_backslash_Nt=1_timestep1.csv};
		    \addplot table[
					x expr = {\thisrow{epsilon}},
					y expr = {\thisrow{mean_symcond1}}
				]{data/DataCondNr_FSMDRK_SSP-I2DRK3-2s_AT_p=1_PR_Tend=1.25_splitting=0_damped_newton_backslash_Nt=1_timestep1.csv};
			\addplot table[
					x expr = {\thisrow{epsilon}},
					y expr = {\thisrow{mean_symcond1}}
				]{data/DataCondNr_FSMDRK_SSP-I2DRK4-5s_AT_p=2_PR_Tend=1.25_splitting=0_damped_newton_backslash_Nt=1_timestep1.csv};
			\addplot table[
					x expr = {\thisrow{epsilon}},
					y expr = {\thisrow{mean_symcond1}}
				]{data/DataCondNr_FSMDRK_HB-I2DRK6-3s_AT_p=3_PR_Tend=1.25_splitting=0_damped_newton_backslash_Nt=1_timestep1.csv};
		    \addplot table[
					x expr = {\thisrow{epsilon}},
					y expr = {\thisrow{mean_symcond1}}
				]{data/DataCondNr_FSMDRK_HB-I3DRK9-3s_AT_p=4_PR_Tend=1.25_splitting=0_damped_newton_backslash_Nt=1_timestep1.csv};
		\end{loglogaxis}
		\end{tikzpicture}
	  }
	  {
	  \includegraphics{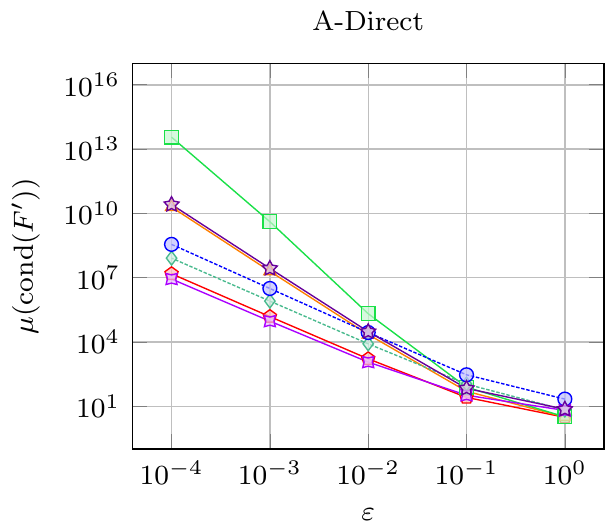}
	  }%	
\end{subfigure}%
\begin{subfigure}{0.33\textwidth}
  \centering
  \ifthenelse{\boolean{compilefromscratch}}{
        \tikzsetnextfilename{DataCondNr_FSMDRK-Direct_EJ_PR_Tend=1.25} 
		\begin{tikzpicture}
		\begin{loglogaxis}[
			title={EJ-Direct},
			height=5.5cm,
			xlabel={$\varepsilon$},
			ylabel={\phantom{$\text{cond}(F')$}},
			ylabel shift = 1.3em,
		    xtick={1e0,1e-1,1e-2,1e-3,1e-4},
	        ymax = 1e17, ymin = 1e-1,
			ytick={1e1,1e4,1e7,1e10,1e13,1e16},
			ymajorticks=false,
			grid=major,
			legend style={nodes={scale=0.95, transform shape}},
			transpose legend,
			legend columns=3,
			legend to name=FSMDRK_schemes-direct-PR,
			cycle list name = rainbow
			]
			\addplot table[
					x expr = {\thisrow{epsilon}},
					y expr = {\thisrow{mean_symcond1}}
				]{data/DataCondNr_FSMDRK_HB-I2DRK4-2s_EJ_PR_Tend=1.25_splitting=0_damped_newton_backslash_Nt=1_timestep1.csv};
			\addlegendentry{HB-I2DRK4-2s};
			
		    \addplot table[
					x expr = {\thisrow{epsilon}},
					y expr = {\thisrow{mean_symcond1}}
				]{data/DataCondNr_FSMDRK_HB-I3DRK6-2s_EJ_PR_Tend=1.25_splitting=0_damped_newton_backslash_Nt=1_timestep1.csv};
			\addlegendentry{HB-I3DRK6-2s};	
				
			\addplot table[
					x expr = {\thisrow{epsilon}},
					y expr = {\thisrow{mean_symcond1}}
				]{data/DataCondNr_FSMDRK_HB-I4DRK8-2s_EJ_PR_Tend=1.25_splitting=0_damped_newton_backslash_Nt=1_timestep1.csv};
			\addlegendentry{HB-I4DRK8-2s};	
				
			\addplot table[
					x expr = {\thisrow{epsilon}},
					y expr = {\thisrow{mean_symcond1}}
				]{data/DataCondNr_FSMDRK_SSP-I2DRK3-2s_EJ_PR_Tend=1.25_splitting=0_damped_newton_backslash_Nt=1_timestep1.csv};
			\addlegendentry{SSP-I2DRK3-2s};			
				
		    \addplot table[
					x expr = {\thisrow{epsilon}},
					y expr = {\thisrow{mean_symcond1}}
				]{data/DataCondNr_FSMDRK_SSP-I2DRK4-5s_EJ_PR_Tend=1.25_splitting=0_damped_newton_backslash_Nt=1_timestep1.csv};
			\addlegendentry{SSP-I2DRK4-5s};			
			
			% Add empty legend to fix the legend layout.
	        	\addlegendimage{empty legend}
	        	\addlegendentry{}
				
			\addplot table[
					x expr = {\thisrow{epsilon}},
					y expr = {\thisrow{mean_symcond1}}
				]{data/DataCondNr_FSMDRK_HB-I2DRK6-3s_EJ_PR_Tend=1.25_splitting=0_damped_newton_backslash_Nt=1_timestep1.csv};
			\addlegendentry{HB-I2DRK6-3s};	
				
			\addplot table[
					x expr = {\thisrow{epsilon}},
					y expr = {\thisrow{mean_symcond1}}
				]{data/DataCondNr_FSMDRK_HB-I3DRK9-3s_EJ_PR_Tend=1.25_splitting=0_damped_newton_backslash_Nt=1_timestep1.csv};
			\addlegendentry{HB-I3DRK9-3s};	
				
		\end{loglogaxis}
		\end{tikzpicture}
	  }
	  {
	  \includegraphics{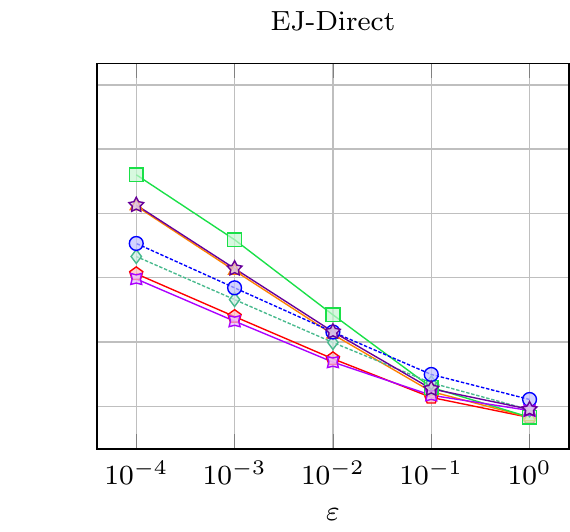}
	  }%
\end{subfigure}%
\begin{subfigure}{0.325\textwidth}
  \centering
  \ifthenelse{\boolean{compilefromscratch}}{
        \tikzsetnextfilename{DataCondNr_FSMDRK-Direct_rec_PR_Tend=1.25} 
		\begin{tikzpicture}
		\begin{loglogaxis}[
			title={rec-Direct},
			height=5.5cm,
			xlabel={$\varepsilon$},
			ylabel={\phantom{$\text{cond}(F')$}},
			xtick={1e0,1e-1,1e-2,1e-3,1e-4},
	        ymax = 1e17, ymin = 1e-1,
			ytick={1e1,1e4,1e7,1e10,1e13,1e16},
			ymajorticks=false,
			grid=major,
			legend style={nodes={scale=0.95, transform shape}},
			cycle list name = rainbow
			]
			\addplot table[
					x expr = {\thisrow{epsilon}},
					y expr = {\thisrow{mean_symcond1}}
				]{data/DataCondNr_FSMDRK_HB-I2DRK4-2s_PR_Tend=1.25_splitting=0_damped_newton_backslash_Nt=1_timestep1.csv};
		    \addplot table[
					x expr = {\thisrow{epsilon}},
					y expr = {\thisrow{mean_symcond1}}
				]{data/DataCondNr_FSMDRK_HB-I3DRK6-2s_PR_Tend=1.25_splitting=0_damped_newton_backslash_Nt=1_timestep1.csv};
			\addplot table[
					x expr = {\thisrow{epsilon}},
					y expr = {\thisrow{mean_symcond1}}
				]{data/DataCondNr_FSMDRK_HB-I4DRK8-2s_PR_Tend=1.25_splitting=0_damped_newton_backslash_Nt=1_timestep1.csv};
			\addplot table[
					x expr = {\thisrow{epsilon}},
					y expr = {\thisrow{mean_symcond1}}
				]{data/DataCondNr_FSMDRK_SSP-I2DRK3-2s_PR_Tend=1.25_splitting=0_damped_newton_backslash_Nt=1_timestep1.csv};
		    \addplot table[
					x expr = {\thisrow{epsilon}},
					y expr = {\thisrow{mean_symcond1}}
				]{data/DataCondNr_FSMDRK_SSP-I2DRK4-5s_PR_Tend=1.25_splitting=0_damped_newton_backslash_Nt=1_timestep1.csv};
			\addplot table[
					x expr = {\thisrow{epsilon}},
					y expr = {\thisrow{mean_symcond1}}
				]{data/DataCondNr_FSMDRK_HB-I2DRK6-3s_PR_Tend=1.25_splitting=0_damped_newton_backslash_Nt=1_timestep1.csv};
			\addplot table[
					x expr = {\thisrow{epsilon}},
					y expr = {\thisrow{mean_symcond1}}
				]{data/DataCondNr_FSMDRK_HB-I3DRK9-3s_PR_Tend=1.25_splitting=0_damped_newton_backslash_Nt=1_timestep1.csv};
		\end{loglogaxis}
		\end{tikzpicture}
	  }
	  {
	  \includegraphics{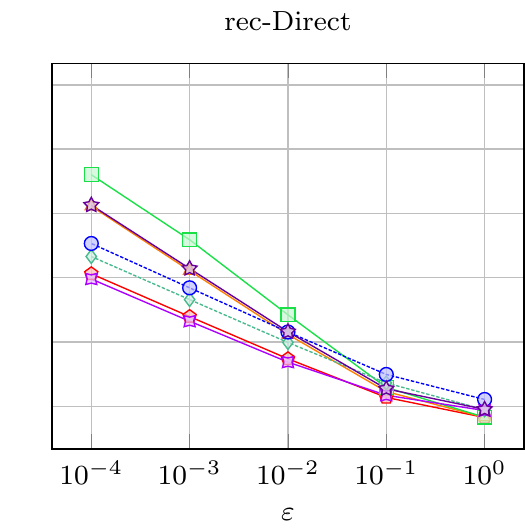}
	  }%
\end{subfigure}
\vspace{-2mm}
\hspace{1.2cm}\ifthenelse{\boolean{compilefromscratch}}{
				\ref{FSMDRK_schemes-direct-PR}
			  }
	  		  {
			  \includegraphics{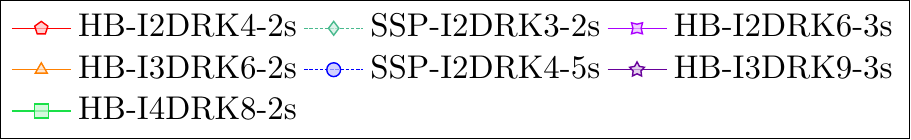}
			  }

\caption{FSMDRK schemes applied as a Direct method (Figure~\ref{fig:AMDRK-scheme}) to the PR problem~\eqref{eq:PR}. The average condition number in the 1-norm of the Newton Jacobian is shown for different values of $\varepsilon$. The behavior $\text{cond}(F') = \mathcal{O}(\varepsilon^{-\mathtt{r}})$ is observed, $\mathtt{r}$ being the amount of derivatives. A single timestep ($N=1$) of size $T_{\text{end}} = 1.25$ ($\Delta t = 1.25$) has been considered with tolerances, Eq.~\eqref{eq:criteria_newton}, set to $10^{-12}$ under a maximum of 1000 iterations.}

\label{fig:CondNr-FSMDRK_direct-Pareschi-Russo}
\end{figure}
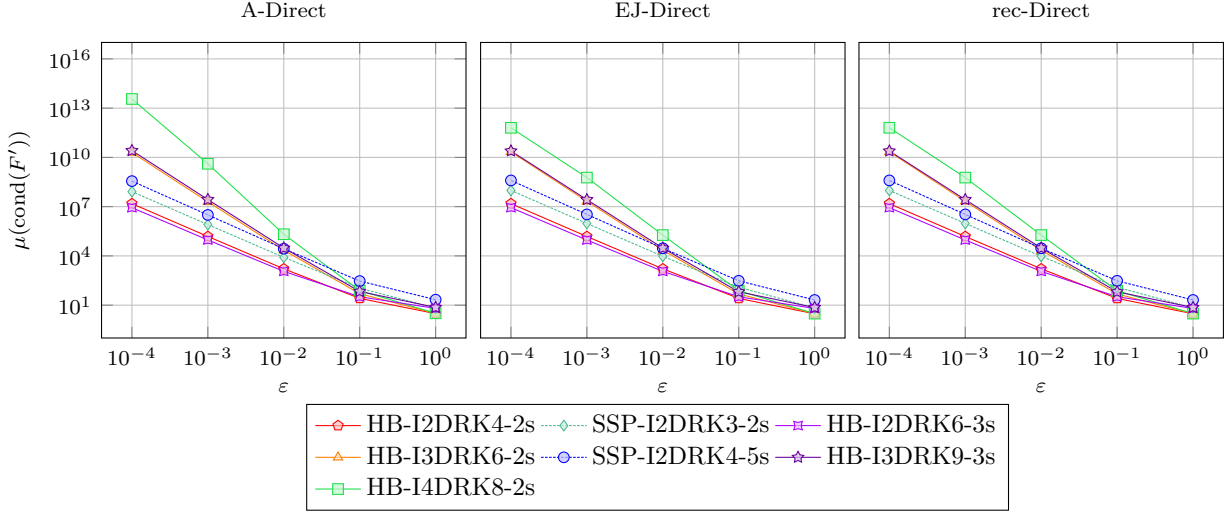

\subsection{Conditioning of a two-variable ODE system}
Eq.~\eqref{eq:PR}, but also van der Pol and Kaps equation, can be put into the form
\begin{align}
y_1'(t) &= f_1(y_1, y_2) \\
y_2'(t) &= f_2(y_1, y_2) + \frac{g(y_1, y_2)}{\varepsilon} \, ,
\end{align}
in which $f_1, f_2$ and $g$ are smooth functions. In order to get a basic understanding of how the condition number of the Newton Jacobian behaves in terms of $\varepsilon$, we consider the simplified system
\begin{equation}\label{eq:2var_systemODEs}
y_1'(t) = y_2, \quad y_2'(t) = \alpha y_1 + \frac{g(y_1, y_2)}{\varepsilon}, \quad 0 \leq t \leq T \, ,
\end{equation}
where $\alpha \in \mathbb{R}$ and $g: \mathbb{R}^2 \rightarrow \mathbb{R}$ is smooth. We are interested in the analytical form of the Newton Jacobian obtained from the (A)MDRK method in the case that $\varepsilon \ll \Delta t$.

\begin{example}\label{ex:taylor2_analytical}
Applying implicit Taylor order $\mathtt{r}=2$ to the system of ODEs \eqref{eq:2var_systemODEs} yields a system $F = (y_1^n, y_2^n)^T$ with
\begin{equation}\label{eq:taylor2_analytical_system}
F = \left[\begin{array}{l}
y_1^{n+1} - \Delta t y_2^{n+1} + \frac{{\Delta t }^2}{2}\left(\alpha y_1^{n+1} + \frac{g^{n+1}}{\varepsilon}\right)\\[0.75em]
y_2^{n+1} - \Delta t\left(\alpha y_1^{n+1} + \frac{g^{n+1}}{\varepsilon}\right) + \frac{{\Delta t }^2}{2} \left(\alpha y_2^{n+1} + \frac{\partial_{y^{}_{1}}g^{n+1}}{\varepsilon}y_2^{n+1} + \frac{\partial_{y^{}_{2}}g^{n+1}}{\varepsilon}(\alpha y_1^{n+1}+\frac{g^{n+1}}{\varepsilon})\right)
\end{array}\right].
\end{equation}
Note that $g^{n+1}$ has been defined as $g(y_1^{n+1}, y_2^{n+1})$.
\begin{proposition}\label{prop:taylor2_analyticalCond}
Assume that $g$ and all its partial derivatives are $\mathcal{O}(1)$, and assume that $\varepsilon \ll \Delta t$. Then, the  Newton Jacobian $F'$ obtained from solving the system of ODEs \eqref{eq:2var_systemODEs} with the implicit Taylor method of order $\mathtt{r}=2$ behaves in the 1-norm as
\[
{\| F'\|}_1 = \mathcal{O}\!\left(\frac{{\Delta t}^2}{\varepsilon^2}\right) \, , \quad \text{and} \quad \cond(F') = \mathcal{O}\left(\varepsilon^{-1}\right) \, .
\]
\end{proposition}
\end{example}

\begin{customremark}{\arabic{remark} (part 1)}
The behavior shown in Prop.~\ref{prop:taylor2_analyticalCond} is not what we observe from the numerical experiments in Figs. \ref{fig:CondNr-DIMDRK-direct-Pareschi-Russo} and \ref{fig:CondNr-FSMDRK_direct-Pareschi-Russo}, where we obtained $ \cond(F') = \mathcal{O}(\varepsilon^{-2})$ for two-derivative schemes. There is no contradiction here though. We reason in part 2 of this remark that, often, an order of $\varepsilon$ is gained through the determinant $\det(F')$. 
%In the following proof we will arrange $\det(F')$ by terms $\mathcal{O}\left(\frac{{\Delta t}^m}{\varepsilon^n}\right)$ of different orders $m, n \in \mathbb{N}$. We reason in part 2 of this remark that, often, instead of having $\det(F') = \mathcal{O}\left(\frac{{\Delta t}^{m_{\text{max}}}}{\varepsilon^{n_{\text{max}}}}\right)$, an order of $\varepsilon$ is gained, so that $\det(F') = \mathcal{O}\left(\frac{{\Delta t}^{m_{\text{max}}}}{\varepsilon^{n_{\text{max}}-1}}\right)$.
\end{customremark}

\begin{proof}[Proof of Proposition \ref{prop:taylor2_analyticalCond}]
For simplicity, the notation $(u,v) = (y_1, y_2)$ will be used in what follows. 
From the construction of $F$ as given in Eq. \eqref{eq:taylor2_analytical_system}, it is apparent that the Newton Jacobian satisfies
\[
{\| F'\|}_1 = \mathcal{O}\!\left(\frac{{\Delta t}^2}{\varepsilon^2}\right) \, ,
\] 
under the assumption that $\varepsilon \ll \Delta t$. For the behavior of the inverse matrix ${F'}^{-1}$ we make use of the identity $ A^{-1} = \frac{1}{\det(A)}\adj(A)$. As $F'$ is a $2\times 2$ matrix, its adjugate is obtained from simply shuffling terms and possibly adding a minus sign. Consequently, the behavior of its norm remains unaffected w.r.t $\varepsilon$ and $\Delta t$. The determinant can be explicitly computed as
\begin{align}\label{eq:taylor2_analyticalDer}
\det(F') = 1 + \frac{1}{4}\frac{{\Delta t}^4}{\varepsilon^3}\underbrace{\left( \partial_{u}g\partial_{vv}g - \partial_{v}g\partial_{uv}g\right)}_{Dg}g + \mathcal{O}(\varepsilon^{-2}) \, .
\end{align}
%\begin{align}\label{eq:taylor2_analyticalDer}
%\begin{split}
%\det(F') = 1 &+ \frac{{\color{magenta}\alpha}^2}{4}{\color{blue}{\Delta t }^4} - {\color{blue}\frac{\Delta t}{\varepsilon}}\partial_{v}g + \frac{1}{2}{\color{blue}\frac{{\Delta t}^2}{\varepsilon}}\left( \partial_{uv}g \cdot v + {\color{magenta}\alpha}\partial_{vv}g \cdot u \right) \\
%&+ \frac{1}{2}{\color{blue}\frac{{\Delta t}^3}{\varepsilon}}\left( \partial_{uu}g \cdot v + {\color{magenta}\alpha}({\partial_{u v}g}\cdot u + {\partial_{v}g})\right)
%\\
%&+ \frac{{\color{magenta}\alpha}}{4}{\color{blue}\frac{{\Delta t}^4}{\varepsilon}}\left( \partial_{uv}g \cdot v + 2\partial_{u}g + {\color{magenta}\alpha} \partial_{vv}g \cdot u \right)  \\
%&+ \frac{1}{2}{\color{blue}\frac{{\Delta t}^2}{\varepsilon^2}}\left( \partial_{vv}g \cdot g + (\partial_{v}g)^2 \right) \\
%&+ \frac{1}{2}{\color{blue}\frac{{\Delta t}^3}{\varepsilon^2}}\left( \partial_{uv}g \cdot g + \partial_{u}g\partial_{v}g \right)  \\
%&+ \frac{1}{4}{\color{blue}\frac{{\Delta t}^4}{\varepsilon^2}}\bigg[ \partial_{u}g\partial_{uv}g \cdot v + (\partial_{u}g)^2 - \partial_{v}g\partial_{uu}g \cdot v \\
%&\phantom{+ \frac{1}{4}{\color{blue}\frac{{\Delta t}^4}{\varepsilon^2}}\bigg[}+ {\color{magenta}\alpha}\big(\partial_{vv}g \cdot g + \underbrace{\left( \partial_{u}g\partial_{vv}g - \partial_{v}g\partial_{uv}g\right)}_{Dg}u \big)\bigg]  \\[-1em]
%&+ \frac{1}{4}{\color{blue}\frac{{\Delta t}^4}{\varepsilon^3}}\underbrace{\left( \partial_{u}g\partial_{vv}g - \partial_{v}g\partial_{uv}g\right)}_{Dg}g  \\
%\end{split}
%\end{align}
So in total:
\begin{align*}
\cond(F') &= {\|F' \|}_1 \cdot {\|{F'}^{-1}\|}_1
= \frac{{\|F' \|}_1{\|\adj(F') \|}_1}{\left|\det(F')\right|}
= \mathcal{O}\!\left(\frac{{\Delta t}^2}{\varepsilon^2}\right)\mathcal{O}\!\left(\frac{\varepsilon^3}{{\Delta t}^4}\right)\mathcal{O}\!\left(\frac{{\Delta t}^2}{\varepsilon^2}\right)
= \mathcal{O}(\varepsilon^{-1}) \, ,
\end{align*}
under the assumption that $\varepsilon \ll \Delta t$.
\end{proof}

\addtocounter{remark}{-1}
\begin{customremark}{\arabic{remark} (part 2)}
In equation \eqref{eq:taylor2_analyticalDer} we observe that $\det(F') = \mathcal{O}\!\left(\frac{{\Delta t}^4}{\varepsilon^3}\right)$ under the assumption that $\varepsilon \ll \Delta t$. In many cases we nonetheless observe ${\det(F') = \mathcal{O}(\varepsilon^{-2})}$:
\begin{enumerate}
\item The values are mainly decided by $g(y_1,y_2)$ and a function of partial derivatives which we have denoted $Dg(y_1, y_2)$. In case of the PR-problem \eqref{eq:PR}, $\alpha = 1$ and $g(y_1, y_2) = \sin(y_1) - y_2$. Therefore, any mixed partial derivatives of $g$, or second partial derivative of $g$ w.r.t $y_2$ equals 0. So for the PR-problem $Dg = 0$.

\item In general, it does not need to hold true that $Dg = 0$. The van der Pol problem (as in \cite{SealSchuetz19}) for instance has ${g(y_1,y_2) = (1-y_1^2)y_2 - y_1}$, and therefore yields $Dg = 2y_1( 1 - y_1^2)$. Here, a clarification can be given by the (very) harsh restriction set in Prop.~\ref{prop:taylor2_analyticalCond} that $g$ and all its partial derivatives are $\mathcal{O}(1)$, which typically is not true. For well-prepared initial conditions and an asymptotically consistent algorithm, $g = \mathcal{O}{(\varepsilon)}$ \cite{RuBosc09}.
%In Figure \ref{fig:vanderPol-solution-derivatives_g} we show plots of the solution $y_1(t)$, $g(y_1, y_2)$ and $Dg(y_1, y_2)$ for several values of $\varepsilon$. From $y_1(t)$, it is clear that up to $t = 0.7$ large timesteps should be feasible, so that possibly $\varepsilon \ll \Delta t$. Up to this time there holds $g = \mathcal{O}(\varepsilon)$ and $Dg = \mathcal{O}(1)$, which thus gives $ \det(F') = \mathcal{O}(\varepsilon^{-2})$.
\end{enumerate}
\end{customremark}

A similar type of effect takes place for a higher amount of derivatives $\mathtt{r}$;  the resulting conditioning is $\mathcal{O}(\varepsilon^{-\mathtt{r}})$.

\section{Derivatives as members of the solution}\label{sec:DerSol}
One of the main issues for the $\mathcal{O}(\varepsilon^{-\mathtt{r}})$ conditioning of the direct (A)MDRK method is the fact that with each higher derivative $y^{(k)}$, the order of $\varepsilon$ increases simultaneously. Such behavior is to be expected due to a built-in dependency on the lower order derivatives, \emph{i.e.}
\begin{align}\label{eq:timeDerDepencies}
\begin{split}
y^{(1)} &= \Phi(y), \\
y^{(k)} &= \Psi_k(y,y^{(1)},\dots,y^{(k-1)}), \quad 2 \leq k \leq \mathtt{r} \, .
\end{split}
\end{align}
The operator $\Psi_k$ is then either the relation that uses the Exact Jacobians (EJ) as in \eqref{eq:FaaDiBruno-timeDer}, so that 
\begin{subequations}\label{eq:FaaDiBruno-timeDer-operator}
\begin{align}
\Psi_2 &= \Phi'(y)y^{(1)}  \, , \\
\Psi_3 &= \Phi''(y) \bullet \left[y^{(1)} {\color{gray}|} y^{(1)}\right] +  \Phi'(y)y^{(2)} \, , \\ 
\Psi_4 &= \Phi'''(y)\bullet \left[y^{(1)} {\color{gray}|} y^{(1)} {\color{gray}|} y^{(1)}\right] + 3\Phi''(y) \bullet \left[y^{(1)} {\color{gray}|} y^{(2)}\right] + \Phi'(y)y^{(3)}  \, ,
\end{align}
\end{subequations}
and so forth, or either is given recursively from \eqref{eq:timeDerRecursive}, so that
\begin{equation}\label{eq:timeDerRecursive-operator}
\Psi_{k+1} = {\left[\frac{\mathrm{d}^{k-1}\Phi(y)}{{\mathrm{d}t}^{k-1}}\right]}'y^{(1)} \, ,
\end{equation}
for $k = 1,\dots, \mathtt{r}-1$.

\begin{customexample}{\arabic{example} (part 1)}
Consider the implicit Taylor scheme of order $\mathtt{r}=3$, then there is only a single stage $Y = y$ to solve for. In terms of the relations \eqref{eq:timeDerDepencies}, the Newton system $F(y) = 0$ simply writes as
\begin{equation}\label{eq:implicitTaylor3Scalar-simple}
y - \Delta t \Phi(y) + \frac{{\Delta t}^2}{2} \Psi_2- \frac{{\Delta t}^3}{6} \Psi_3 - y^n = 0 \, .
\end{equation} 
\end{customexample}
From the above example it is clear that computing $F'(Y)$ necessitates deriving the formulas $\Psi_k$ with respect to $y$,
\begin{equation}\label{eq:jacobianTimeDer}
\frac{\partial y^{(k)}}{\partial y} = \partial_{y}\Psi_k + \sum\limits_{m=1}^{k-1}\partial_{y^{(m)}}\Psi_k \cdot \frac{\partial y^{(m)}}{\partial y} \, .
\end{equation}
It is exactly because of this recursive dependency on lower order derivatives that the order of $\varepsilon$ increases in $\cond(F')$. A similar recursion holds true when calculating the approximate values $\widetilde{y}^{(k)}$ with the recursive formulas \eqref{eq:timeDerApprox}-\eqref{eq:approxPhiValues}.

In order to better understand the $\varepsilon$-behavior, we investigate a linear problem in the sequel. To reduce the complexity of involved formulas, we only consider scalar problems ($m=1$) in this section.

\subsection{$\varepsilon$-scaled Dahlquist test equation}
We consider an $\varepsilon$-scaled Dahlquist test problem
\begin{equation}\label{eq:Dahlquist}
	y' = \frac{\lambda}{\varepsilon}y, \qquad y(0) = 1,
\end{equation}
with the exact solution $y(t) = \mathrm{e}^{(\lambda/\varepsilon)t}$. As the equation is linear, the AMDRK method (A-Direct) coincides with the MDRK method that uses EJ (EJ-Direct), see for example \cite[Proposition~1]{2018_Baeza_EtAl_Arxiv}\footnote{The AMDRK method approximates the derivatives $y^{(k)}$ on the basis of finite differences. For linear problems finite differences are exact.}. The rationale behind the observed behavior follows immediately from the next lemma.
\begin{lemma}
The derivatives $y^{(k)}$ and their Jacobians $\partial_y y^{(k)}$ of the $\varepsilon$-scaled Dahlquist test are $\mathcal{O}(\varepsilon^{-k})$, {i.e.}
\begin{alignat*}{2}
y^{(k)} &= \dotPhi{k-1} =  \left(\frac{\lambda}{\varepsilon} \right)^k y = \mathcal{O}(\varepsilon^{-k}), \qquad
\frac{\partial y^{(k)}}{\partial y} &=& {\left[\frac{\mathrm{d}^{k-1}\Phi}{{\mathrm{d}t}^{k-1}}\right]}' = \left(\frac{\lambda}{\varepsilon} \right)^k = \mathcal{O}(\varepsilon^{-k})\, .
\end{alignat*}
\end{lemma}
It would be more optimal for the conditioning of the Jacobian to unfold the $\varepsilon$-dependency through its recursion given by $\Psi_k$ in Eqs. \eqref{eq:timeDerDepencies}. When applying EJ (and thus also for AMDRK) there holds,
\begin{equation}\label{eq:EJ_Dahlquist}
\Psi_k = \frac{\lambda}{\varepsilon} y^{(k-1)} \, ,
\end{equation}
whereas recursion (rec-Direct) gives the relation
\begin{equation}\label{eq:rec_Dahlquist}
\Psi_k = \left(\frac{\lambda}{\varepsilon} \right)^{k-1} y^{(1)} \, ,
\end{equation}
for $k = 1, \dots, \mathtt{r}$. Already here we can notice that the first out of these two is more favorable, as it unfolds the $\varepsilon$-dependency more thoroughly.

\subsection{Recursive dependencies as additional system equations}
In order to achieve such an unfolding of the $\varepsilon$-dependency, Baeza et al. \cite{2020_Baeza_EtAl} suggest to take the derivatives as members of the solution. Instead of directly solving for $Y$, additionally, the independent unknowns
\begin{equation}
~~ z_k \approx y^{(k)} \, , \qquad 1 \leq k \leq \mathtt{r},
\end{equation}
are sought for using the same recursive dependencies
\begin{align}\label{eq:DerSolTimeDerDepencies}
\begin{split}
z_1 &= \Phi(z_0), \\
z_k &= \Psi_k(z_0,z_1,\dots,z_{k-1}), \quad 2 \leq k \leq \mathtt{r},
\end{split}
\end{align}
where we have defined $z_0:=Y$. In constrast to the single relation $F(Y) = 0$, we now solve the $\mathtt{r} + 1$ relations as a bigger system $\mathcal{F}(z) = 0$, with ${z:=(z_0, z_1, \dots, z_{\mathtt{r}})}$. In summary, the recursive dependency in one single formula is traded off for a larger system containing the $\mathtt{r}$ additional relations given by \eqref{eq:DerSolTimeDerDepencies}.

\addtocounter{example}{-1}
\begin{customexample}{\arabic{example} (part 2)}\label{ex:DerSol-Taylor3}
For the third order Taylor scheme \eqref{eq:implicitTaylor3Scalar-simple},
\begin{equation}
z_0 - \Delta t z_1 + \frac{{\Delta t}^2}{2} z_2 - \frac{{\Delta t}^3}{6} z_3 - y^n = 0 \, ,
\end{equation}
 and
\begin{equation}
\mathcal{F}(z) = \begin{bmatrix}
z_0 - \Delta t z_1 + \frac{{\Delta t}^2}{2} z_2 - \frac{{\Delta t}^3}{6} z_3 - y^n \\
\Phi(z_0) - z_1 \\
\Psi_2(z_0,z_1) - z_2 \\
\Psi_3(z_0,z_1,z_{2}) - z_3
\end{bmatrix} \, .
\end{equation}
The Jacobian is now less clustered, in our example
\begin{equation}\label{eq:JacobianTaylor3DerSol}
\mathcal{F}'(z) = \begin{bmatrix}
1 & -\Delta t & \frac{{\Delta t}^2}{2} & - \frac{{\Delta t}^3}{6} \\
\Phi'(z_0) & -1 & 0 & 0 \\
\partial_{z_0}\Psi_2 & \partial_{z_1}\Psi_2 & -1 & 0 \\
\partial_{z_0}\Psi_3 & \partial_{z_1}\Psi_3 & \partial_{z_2}\Psi_3 & -1 
\end{bmatrix} \, .
\end{equation}
In the case of the $\varepsilon$-scaled Dahlquist test \eqref{eq:Dahlquist}, the relations \eqref{eq:EJ_Dahlquist} and \eqref{eq:rec_Dahlquist} respectively yield
\begin{equation}
\mathcal{F}_{\normalfont\text{EJ}}'(z) = \begin{bmatrix}
1 & -\Delta t & \frac{{\Delta t}^2}{2} & -\frac{{\Delta t}^3}{6}\\
-\frac{1}{\varepsilon} & -1 & 0 & 0  \\
0 & -\frac{1}{\varepsilon}& -1 & 0 \\
0 & 0 & -\frac{1}{\varepsilon}& -1 \\
\end{bmatrix} \quad \text{and} \quad \mathcal{F}_{\normalfont\text{rec}}'(z) = \begin{bmatrix}
1 & -\Delta t & \frac{{\Delta t}^2}{2} & -\frac{{\Delta t}^3}{6}\\
-\frac{1}{\varepsilon} & -1 & 0 & 0  \\
0 & -\frac{1}{\varepsilon}& -1 & 0 \\
0 & \frac{1}{\varepsilon^2} & 0 & -1 \\
\end{bmatrix}.
\end{equation}
\end{customexample}

Regarding AMDRK schemes, Baeza et al \cite{2020_Baeza_EtAl} introduce the scaled unknowns $z_k \approx {\Delta t}^{k-1}\widetilde{y}^{(k)}$, $1 \leq k \leq \mathtt{r}$. With this choice, analogous relations
\begin{align}\label{eq:A-DerSolTimeDerDepencies}
\begin{split}
z_1 &= \Phi(z_0), \\
z_k &= \widetilde{\Psi}_k(z_0,z_1,\dots,z_{k-1}), \quad 2 \leq k\leq \mathtt{r},
\end{split}
\end{align}
are found on the basis of the formulas \eqref{eq:timeDerApprox}-\eqref{eq:approxPhiValues}, namely
\begin{equation}\label{eq:A-DerSol_Psi_k}
\widetilde{\Psi}_k := {\Delta t}^{k-1}P^{(k-1)}\mathbf{\Phi}^{k-1, \langle n \rangle}_T \, .
\end{equation}
In here, 
\begin{equation}\label{eq:A-DerSol_PhiValues}
\Phi_T^{k-1,n+j} := \Phi\left(z_0 + \Delta t \sum\limits_{m=1}^{k-1}\frac{j^m}{m!} z_m\right) \, ,
\end{equation}
for $j = -p, \dots, p$. Note the slight redefinition of $\Phi_T^{k-1,n+j}$ in contrast to Eq.~\eqref{eq:approxPhiValues} to account for the $\Delta t$ dependency of the $\widetilde{\Psi}_k$.
For a specific example of the AMDRK method, and its Jacobian $\widetilde {\mathcal{F}}'(z)$, we refer the reader to {\cite[Subsection 4.2]{2020_Baeza_EtAl}}.

%\begin{remark}
%The scaling is favored since the aproximations $\widetilde{y}^{(k)}$ (consider Eq. \eqref{eq:timeDerApprox}) replace the derivatives $y^{(k)}$ in Def. \ref{def:MDRK-scheme}. Therefore terms ${\Delta t}^{k} P^{(k-1)}\mathbf{\Phi}^{k-1, \langle n+1 \rangle}_T$ have to be computed, making the scaling $\frac{1}{{\Delta t}^{k-1}}$ in the linear operator $P^{(k-1)}$ redundant. 
%\end{remark}

% \addtocounter{example}{-1}
% \begin{customexample}{\arabic{example} (part 3)}
% The same third order Taylor scheme then gives the approximate system
% \begin{equation}\label{eq:approxTaylor3DerSol}
% \widetilde{\mathcal{F}}(z) = \begin{bmatrix}
% z_0 - \Delta t z_1 + \frac{\Delta t}{2} z_2 - \frac{\Delta t}{6} z_3 - y^n \\
% \Phi(z_0) - z_1 \\
% \widetilde{\Psi}_2(z_0,z_1) - z_2 \\
% \widetilde{\Psi}_3(z_0,z_1,z_{2}) - z_3
% \end{bmatrix} \, .
% \end{equation}
% \end{customexample}

As a counterpart to the ``Direct'' MDRK methods in Section \ref{sec:MDRK}, we denote the MDRK approach in which the derivatives are taken as members of the solution by ``DerSol''. A summary of the six different MDRK approaches is presented in Figure \ref{fig:AMDRK-scheme}.
From the specific Taylor example that we have investigated in this section, there are two important observations to be made:
\begin{itemize}
\item Most importantly, compared to $\mathcal{F}'_{\text{EJ}}$ and $\mathcal{F}'_{\text{rec}}$, no second order Jacobian $\Phi''$ occurs for the approximate procedure. From \eqref{eq:A-DerSol_Psi_k}-\eqref{eq:A-DerSol_PhiValues} it can be observed that the AMDRK method solely relies on finite difference computations of $\Phi$. Hence, $\Phi'$ is sufficient for retrieving partial derivatives of $\widetilde{\Psi}_k$. If the problem is not scalar anymore  ($m > 1$) no tensor calculations are needed, whereas such calculations can not be avoided for an exact MDRK scheme.
\item Starting from three derivatives, the matrices $\mathcal{F}'_{\text{EJ}}$ and $\mathcal{F}'_{\text{rec}}$ are not the same anymore, i.e. $\mathcal{F}'_{\text{rec}}$ will only fill up the first two columns (and the diagonal), whereas $\mathcal{F}'_{\text{EJ}}$ has a full lower-triangular submatrix. In the numerical results below it will be demonstrated that there is significantly different behavior in the conditioning of these Jacobians.
\end{itemize}

%%%%%%%%%%%%%%%%%%%%%%%%%%%%%%%%%%%%%%%
%% DIMDRK (dersol) -- Pareschi-Russo %%
%%%%%%%%%%%%%%%%%%%%%%%%%%%%%%%%%%%%%%%
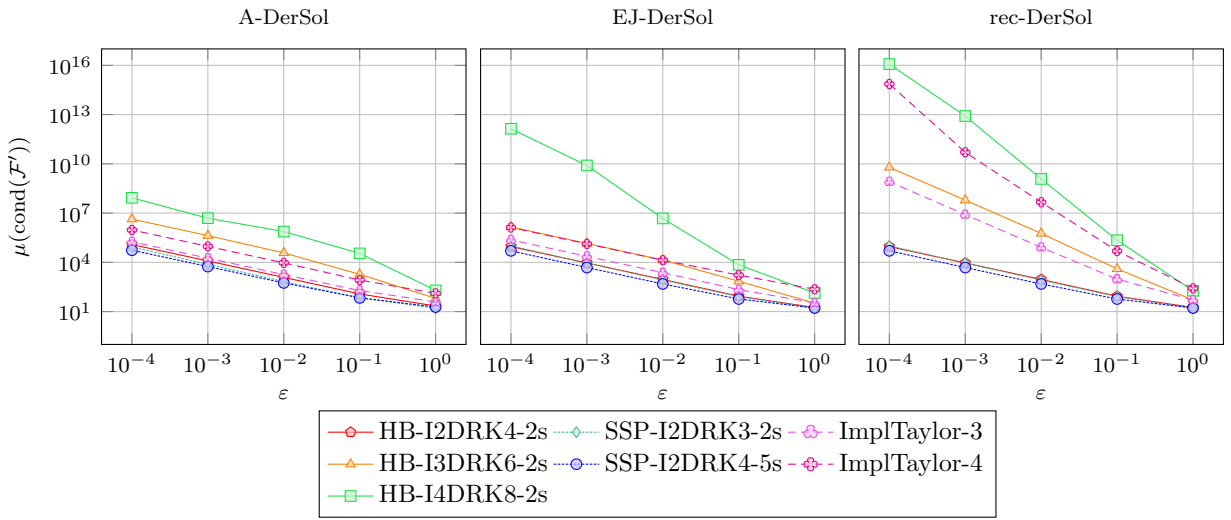
\begin{figure}[h!]
\pgfplotsset{
	title style={font=\footnotesize},
	tick label style={font=\footnotesize},
	label style={font=\footnotesize},
	legend style={font=\footnotesize}
}
\centering
\begin{subfigure}{.325\textwidth}
  \centering
  \ifthenelse{\boolean{compilefromscratch}}{
        \tikzsetnextfilename{DataCondNr_DIMDRK-DerSol_AT_PR_Tend=1.25}
		\begin{tikzpicture}
		\begin{loglogaxis}[
			title={A-DerSol},
			height=5.5cm,
			xlabel={$\varepsilon$},
			ylabel={$\mu\!\left(\text{cond}(\mathcal{F}')\right)$},
			xtick={1e0,1e-1,1e-2,1e-3,1e-4},
	        ymax = 1e17, ymin = 1e-1,
			ytick={1e1,1e4,1e7,1e10,1e13,1e16},
			grid=major,
			legend style={nodes={scale=0.95, transform shape}},
	%		legend entries={HB-I2DRK4-2s, HB-I3DRK6-2s, SSP-I2DRK3-2s, SSP-I2DRK4-5s},
	%		legend pos=north east,
			cycle list name = rainbow2
			]
			\addplot table[
					x expr = {\thisrow{epsilon}},
					y expr = {\thisrow{mean_symcond1}}
				]{data/DataCondNr_DIMDRK_HB-I2DRK4-2s_AT_p=2_dersol_PR_Tend=1.25_splitting=0_damped_newton_backslash_Nt=1_timestep1_stage2.csv};
		    \addplot table[
					x expr = {\thisrow{epsilon}},
					y expr = {\thisrow{mean_symcond1}}
				]{data/DataCondNr_DIMDRK_HB-I3DRK6-2s_AT_p=3_dersol_PR_Tend=1.25_splitting=0_damped_newton_backslash_Nt=1_timestep1_stage2.csv};
		    \addplot table[
					x expr = {\thisrow{epsilon}},
					y expr = {\thisrow{mean_symcond1}}
				]{data/DataCondNr_DIMDRK_HB-I4DRK8-2s_AT_p=4_dersol_PR_Tend=1.25_splitting=0_damped_newton_backslash_Nt=1_timestep1_stage2.csv};
			\addplot table[
					x expr = {\thisrow{epsilon}},
					y expr = {\thisrow{mean_symcond1}}
				]{data/DataCondNr_DIMDRK_SSP-I2DRK3-2s_AT_p=1_dersol_PR_Tend=1.25_splitting=0_damped_newton_backslash_Nt=1_timestep1_stage2.csv};
		    \addplot table[
					x expr = {\thisrow{epsilon}},
					y expr = {\thisrow{mean_symcond1}}
				]{data/DataCondNr_DIMDRK_SSP-I2DRK4-5s_AT_p=2_dersol_PR_Tend=1.25_splitting=0_damped_newton_backslash_Nt=1_timestep1_stage5.csv};
			\addplot table[
					x expr = {\thisrow{epsilon}},
					y expr = {\thisrow{mean_symcond1}}
				]{data/DataCondNr_DIMDRK_ImplTaylor_order=3_AT_p=1_dersol_PR_Tend=1.25_splitting=0_damped_newton_backslash_Nt=1_timestep1_stage1.csv};
			\addplot table[
					x expr = {\thisrow{epsilon}},
					y expr = {\thisrow{mean_symcond1}}
				]{data/DataCondNr_DIMDRK_ImplTaylor_order=4_AT_p=2_dersol_PR_Tend=1.25_splitting=0_damped_newton_backslash_Nt=1_timestep1_stage1.csv};
		\end{loglogaxis}
		\end{tikzpicture}	
	  }
	  {
	  \includegraphics{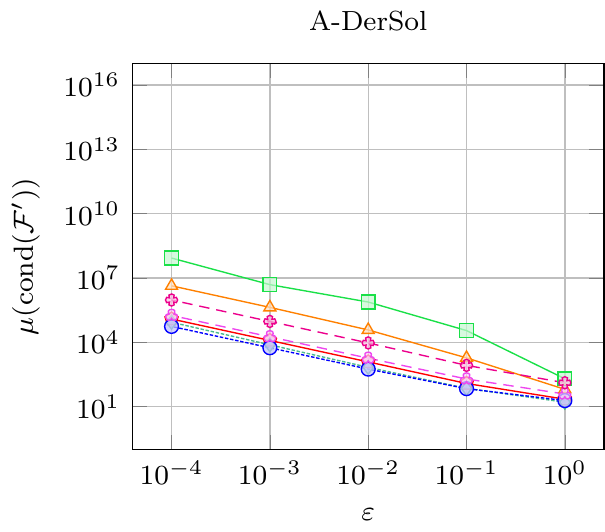}
	  }%
\end{subfigure}%
\begin{subfigure}{0.33\textwidth}
  \centering
  \ifthenelse{\boolean{compilefromscratch}}{
        \tikzsetnextfilename{DataCondNr_DIMDRK-DerSol_EJ_PR_Tend=1.25}
		\begin{tikzpicture}
		\begin{loglogaxis}[
			title={EJ-DerSol},
			height=5.5cm,
			xlabel={$\varepsilon$},
			ylabel={\phantom{$\text{cond}(\mathcal{F}')$}},
			xtick={1e0,1e-1,1e-2,1e-3,1e-4},
			ylabel shift = 1.3em,
	        ymax = 1e17, ymin = 1e-1,
			ytick={1e1,1e4,1e7,1e10,1e13,1e16},
			ymajorticks=false,
			grid=major,
			legend style={nodes={scale=0.95, transform shape}},
			transpose legend,
			legend columns=3,
			legend to name=DIMDRKschemes-dersol-PR,
			cycle list name = rainbow2
			]
			
			\addplot table[
					x expr = {\thisrow{epsilon}},
					y expr = {\thisrow{mean_symcond1}}
				]{data/DataCondNr_DIMDRK_HB-I2DRK4-2s_dersol_EJ_PR_Tend=1.25_splitting=0_damped_newton_backslash_Nt=1_timestep1_stage2.csv};
		    \addlegendentry{HB-I2DRK4-2s};
		    
		    \addplot table[
					x expr = {\thisrow{epsilon}},
					y expr = {\thisrow{mean_symcond1}}
				]{data/DataCondNr_DIMDRK_HB-I3DRK6-2s_dersol_EJ_PR_Tend=1.25_splitting=0_damped_newton_backslash_Nt=1_timestep1_stage2.csv};
			\addlegendentry{HB-I3DRK6-2s};		
			
			\addplot table[
					x expr = {\thisrow{epsilon}},
					y expr = {\thisrow{mean_symcond1}}
				]{data/DataCondNr_DIMDRK_HB-I4DRK8-2s_dersol_EJ_PR_Tend=1.25_splitting=0_damped_newton_backslash_Nt=1_timestep1_stage2.csv};
			\addlegendentry{HB-I4DRK8-2s};
			
			\addplot table[
					x expr = {\thisrow{epsilon}},
					y expr = {\thisrow{mean_symcond1}}
				]{data/DataCondNr_DIMDRK_SSP-I2DRK3-2s_dersol_EJ_PR_Tend=1.25_splitting=0_damped_newton_backslash_Nt=1_timestep1_stage2.csv};
		    \addlegendentry{SSP-I2DRK3-2s};
		    
		    \addplot table[
					x expr = {\thisrow{epsilon}},
					y expr = {\thisrow{mean_symcond1}}
				]{data/DataCondNr_DIMDRK_SSP-I2DRK4-5s_dersol_EJ_PR_Tend=1.25_splitting=0_damped_newton_backslash_Nt=1_timestep1_stage5.csv};
			\addlegendentry{SSP-I2DRK4-5s};
			
			% Add empty legend to fix the legend layout.
	        \addlegendimage{empty legend}
	        \addlegendentry{}
			
			\addplot table[
					x expr = {\thisrow{epsilon}},
					y expr = {\thisrow{mean_symcond1}}
				]{data/DataCondNr_DIMDRK_ImplTaylor_order=3_dersol_EJ_PR_Tend=1.25_splitting=0_damped_newton_backslash_Nt=1_timestep1_stage1.csv};
			\addlegendentry{ImplTaylor-3};
			
			\addplot table[
					x expr = {\thisrow{epsilon}},
					y expr = {\thisrow{mean_symcond1}}
				]{data/DataCondNr_DIMDRK_ImplTaylor_order=4_dersol_EJ_PR_Tend=1.25_splitting=0_damped_newton_backslash_Nt=1_timestep1_stage1.csv};
			\addlegendentry{ImplTaylor-4};
		
		\end{loglogaxis}
		\end{tikzpicture}
	  }
	  {
	  \includegraphics{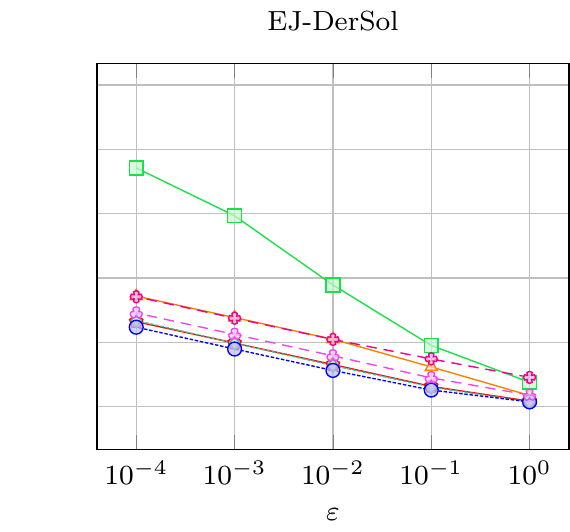}
	  }%
\end{subfigure}%
\begin{subfigure}{0.325\textwidth}
  \centering
  \ifthenelse{\boolean{compilefromscratch}}{
        \tikzsetnextfilename{DataCondNr_DIMDRK-DerSol_rec_PR_Tend=1.25}
		\begin{tikzpicture}
		\begin{loglogaxis}[
			title={rec-DerSol},
			height=5.5cm,
			xlabel={$\varepsilon$},
			ylabel={\phantom{$\text{cond}(\mathcal{F}')$}},
			xtick={1e0,1e-1,1e-2,1e-3,1e-4},
	        ymax = 1e17, ymin = 1e-1,
			ytick={1e1,1e4,1e7,1e10,1e13,1e16},
			ymajorticks=false,
			grid=major,
			legend style={nodes={scale=0.95, transform shape}},
			cycle list name = rainbow2
			]
			\addplot table[
					x expr = {\thisrow{epsilon}},
					y expr = {\thisrow{mean_symcond1}}
				]{data/DataCondNr_DIMDRK_HB-I2DRK4-2s_dersol_PR_Tend=1.25_splitting=0_damped_newton_backslash_Nt=1_timestep1_stage2.csv};
		    \addplot table[
					x expr = {\thisrow{epsilon}},
					y expr = {\thisrow{mean_symcond1}}
				]{data/DataCondNr_DIMDRK_HB-I3DRK6-2s_dersol_PR_Tend=1.25_splitting=0_damped_newton_backslash_Nt=1_timestep1_stage2.csv};
		    \addplot table[
					x expr = {\thisrow{epsilon}},
					y expr = {\thisrow{mean_symcond1}}
				]{data/DataCondNr_DIMDRK_HB-I4DRK8-2s_dersol_PR_Tend=1.25_splitting=0_damped_newton_backslash_Nt=1_timestep1_stage2.csv};
			\addplot table[
					x expr = {\thisrow{epsilon}},
					y expr = {\thisrow{mean_symcond1}}
				]{data/DataCondNr_DIMDRK_SSP-I2DRK3-2s_dersol_PR_Tend=1.25_splitting=0_damped_newton_backslash_Nt=1_timestep1_stage2.csv};
		    \addplot table[
					x expr = {\thisrow{epsilon}},
					y expr = {\thisrow{mean_symcond1}}
				]{data/DataCondNr_DIMDRK_SSP-I2DRK4-5s_dersol_PR_Tend=1.25_splitting=0_damped_newton_backslash_Nt=1_timestep1_stage5.csv};
			\addplot table[
					x expr = {\thisrow{epsilon}},
					y expr = {\thisrow{mean_symcond1}}
				]{data/DataCondNr_DIMDRK_ImplTaylor_order=3_dersol_PR_Tend=1.25_splitting=0_damped_newton_backslash_Nt=1_timestep1_stage1.csv};
			\addplot table[
					x expr = {\thisrow{epsilon}},
					y expr = {\thisrow{mean_symcond1}}
				]{data/DataCondNr_DIMDRK_ImplTaylor_order=4_dersol_PR_Tend=1.25_splitting=0_damped_newton_backslash_Nt=1_timestep1_stage1.csv};
		\end{loglogaxis}
		\end{tikzpicture}
	  }
	  {
	  \includegraphics{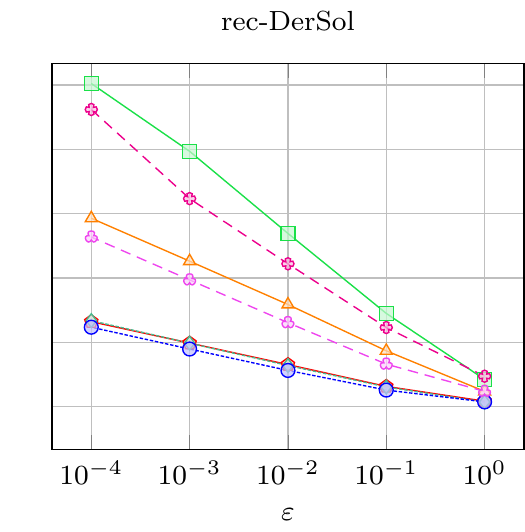}
	  }%
\end{subfigure}
\vspace{-2mm}
\hspace{1.2cm}\ifthenelse{\boolean{compilefromscratch}}{
				\ref{DIMDRKschemes-dersol-PR}
			  }
	  		  {
			  \includegraphics{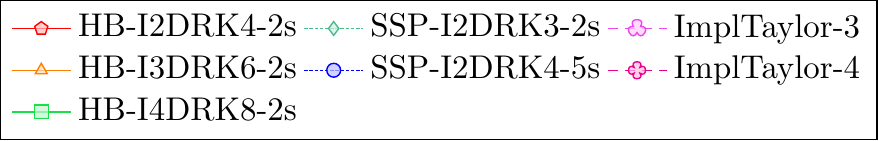}
			  }
			  
\caption{DIMDRK schemes applied as a DerSol method (Figure \ref{fig:AMDRK-scheme}) to the PR problem \eqref{eq:PR}. The average condition number in the 1-norm of the Newton Jacobian obtained from the last RK-stage is shown for different values of $\varepsilon$.  The behavior $\text{cond}(\mathcal{F}') = \mathcal{O}(\varepsilon^{-1})$ is observed for the A-DerSol and EJ-DerSol methods, the rec-DerSol methods seem to behave as $\mathcal{O}(\varepsilon^{-\mathtt{r}+1})$, $\mathtt{r}$ being the amount of derivatives. A single timestep ($N=1$) of size $T_{\text{end}} = 1.25$ ($\Delta t = 1.25$) has been considered with tolerances, Eq. \eqref{eq:criteria_newton}, set to $10^{-12}$ under a maximum of 1000 iterations.}
\label{fig:CondNr-DIMDRK-dersol-Pareschi-Russo}
\end{figure}

When effectively applying the DerSol approach to several DIMDRK schemes, different orders of $\varepsilon$ can be observed in the condition numbers, see Figure~\ref{fig:CondNr-DIMDRK-dersol-Pareschi-Russo}. In comparison to the Direct MDRK approach (see Figure~\ref{fig:CondNr-DIMDRK-direct-Pareschi-Russo}), many schemes behave as
\begin{equation}
\text{cond}(\mathcal{F}') = \mathcal{O}(\varepsilon^{-1}) \, ,
\end{equation}
confirming the succesful unfolding of the $\varepsilon$-dependency through the $\mathtt{r}$ additional equations in the A-DerSol and (partially) in the EJ-DerSol approach. The same can not be said for the rec-DerSol approach, where the order seems to behave as $\mathcal{O}(\varepsilon^{-\mathtt{r}+1})$. This behavior was foreshadowed in the relation \eqref{eq:rec_Dahlquist}; exactly one recursion order is resolved, therefore as well unfolding exactly one order in the $\varepsilon$-dependency. For that reason, it is highly disadvised to apply the rec-DerSol approach for practical purposes.

The EJ-DerSol approach as well does not seem to be flawless when we consider the scheme \mbox{HB-I4DRK8-2s} (Table \ref{tab:HB-I4DRK8-2s}). Instead, the order $\mathcal{O}(\varepsilon^{-3})$ seems to be achieved. In fact, numerically we observe that the scheme tend toward $\mathcal{O}(\varepsilon^{-2})$ for up to $\varepsilon = 10^{-8}$. From running all schemes in Figure \ref{fig:CondNr-DIMDRK-dersol-Pareschi-Russo} up to $\varepsilon = 10^{-8}$, this behavior appears to be unique among the applied DIMDRK schemes. Even more so, when considering different problems (van der Pol and Kaps, see \cite{SealSchuetz19}), all the same schemes show $\mathcal{O}(\varepsilon^{-1})$ up to $\varepsilon = 10^{-8}$, except for HB-I4DRK8-2s applied to van der Pol. For both the A-DerSol and the EJ-DerSol approach, around $\varepsilon \approx 10^{-6}$ there is a sudden change from $\mathcal{O}(\varepsilon^{-1})$ to $\mathcal{O}(\varepsilon^{-4})$ and worse.

This leads us to believe that the observed phenomena of the \mbox{HB-I4DRK8-2s} scheme come as a result of floating-point arithmetic. The double-precision format in \texttt{MATLAB} has a machine precision of ${2^{-52} \approx 2.22\cdot 10^{-16}}$. Given a value $\varepsilon = 10^{-4}$, a four-derivative Runge-Kutta method yields values $\varepsilon^4 = 10^{-16}$ in the denominator of $z_4 = \Psi_3(z_0,z_1,z_2,z_3)$. Albeit the implicit Taylor method of order $4$ giving the requested behavior for the condition number, the Butcher coefficients are larger compared with those of the \mbox{HB-I4DRK8-2s} scheme (see Tables \ref{tab:implTaylor} and \ref{tab:HB-I4DRK8-2s}). The application of many-derivative schemes to stiff problems having very small values $\varepsilon$ should therefore be regarded with sufficient awareness of the machine accuracy being used.

\subsection{The (A)MDRK scheme for a general amount of stages}
 In the most general case, it is not possible to solve for each stage one at a time, an FSMDRK approach is therefore a necessity. Thus, there is a need to solve for $Y = \left({y}^{n,1}, \dots, {y}^{n,\mathtt{s}}\right)$ at once. This entails that for each stage $\mathtt{r}+1$ separate equations have to be solved, leading to a Jacobian matrix $\mathcal{F}(z)$ of size $\left((\mathtt{r}+1)\mathtt{s}M\right)^2$.
   
When using the DerSol approach, one has two options in which one can order all the unknown variables. Either all the variables of the same stage are grouped together, or either the variables are collected by degree of the derivatives. In this work we have chosen to do the ordering in a \textit{stage-based} manner
\begin{equation}
z = \left(z^{n,1}, \dots, {z}^{n,\mathtt{s}} \right) \, ,
\end{equation}
with for each stage $z^{n,{\color{BlueGreen}l}}:=(z^{n,{\color{BlueGreen}l}}_0, z^{n,{\color{BlueGreen}l}}_1, \dots, z^{n,{\color{BlueGreen}l}}_{\mathtt{r}})$. This allows us to obtain an anologous block-structure \eqref{eq:jacobianFSMDRK-Direct} as in the Direct implementation:
\begin{equation}\label{eq:jacobianFSMDRK-DerSol}
\mathcal{F}'(z) = \left[\def\arraystretch{1.3}
\begin{array}{c|c|c}
\partial_{z^{n,1}}\mathcal{F}_{1} & \hdots & \partial_{z^{n,\mathtt{s}}}\mathcal{F}_{1} \\ \hline 
\vdots & \ddots & \vdots \\ \hline
\partial_{z^{n,1}}\mathcal{F}_{{\mathtt{s}}} & \hdots & \partial_{z^{n,\mathtt{s}}}\mathcal{F}_{{\mathtt{s}}}
\end{array}\right] \, ,
\end{equation}
where each block-matrix $\partial_{z^{n,{\color{VioletRed}\nu}}}\mathcal{F}_{{\color{BlueGreen}l}}$ inside is of size $\left((\mathtt{r}+1)M\right)^2$ with a similar construction as the matrix \eqref{eq:JacobianTaylor3DerSol} in Example \ref{ex:DerSol-Taylor3}.

\begin{figure}[h!]
\pgfplotsset{
	title style={font=\footnotesize},
	tick label style={font=\footnotesize},
	label style={font=\footnotesize},
	legend style={font=\footnotesize}
}
\centering
\begin{subfigure}{.325\textwidth}
  \centering
  \ifthenelse{\boolean{compilefromscratch}}{
        \tikzsetnextfilename{DataCondNr_FSMDRK-DerSol_AT_PR_Tend=1.25}
		\begin{tikzpicture}
		\begin{loglogaxis}[
			title={A-DerSol},
			height=5.5cm,
			xlabel={$\varepsilon$},
			ylabel={$\mu\!\left(\text{cond}(\mathcal{F}')\right)$},
			xtick={1e0,1e-1,1e-2,1e-3,1e-4},
	        ymax = 1e17, ymin = 1e-1,
			ytick={1e1,1e4,1e7,1e10,1e13,1e16},
			grid=major,
			legend style={nodes={scale=0.95, transform shape}},
	%		legend entries={HB-I2DRK4-2s, HB-I3DRK6-2s, SSP-I2DRK3-2s, SSP-I2DRK4-5s},
	%		legend pos=north east,
			cycle list name = rainbow
			]
			\addplot table[
					x expr = {\thisrow{epsilon}},
					y expr = {\thisrow{mean_symcond1}}
				]{data/DataCondNr_FSMDRK_HB-I2DRK4-2s_AT_p=2_dersol_PR_Tend=1.25_splitting=0_damped_newton_backslash_Nt=1_timestep1.csv};
		    \addplot table[
					x expr = {\thisrow{epsilon}},
					y expr = {\thisrow{mean_symcond1}}
				]{data/DataCondNr_FSMDRK_HB-I3DRK6-2s_AT_p=3_dersol_PR_Tend=1.25_splitting=0_damped_newton_backslash_Nt=1_timestep1.csv};
		    \addplot table[
					x expr = {\thisrow{epsilon}},
					y expr = {\thisrow{mean_symcond1}}
				]{data/DataCondNr_FSMDRK_HB-I4DRK8-2s_AT_p=4_dersol_PR_Tend=1.25_splitting=0_damped_newton_backslash_Nt=1_timestep1.csv};
			\addplot table[
					x expr = {\thisrow{epsilon}},
					y expr = {\thisrow{mean_symcond1}}
				]{data/DataCondNr_FSMDRK_SSP-I2DRK3-2s_AT_p=1_dersol_PR_Tend=1.25_splitting=0_damped_newton_backslash_Nt=1_timestep1.csv};
		    \addplot table[
					x expr = {\thisrow{epsilon}},
					y expr = {\thisrow{mean_symcond1}}
				]{data/DataCondNr_FSMDRK_SSP-I2DRK4-5s_AT_p=2_dersol_PR_Tend=1.25_splitting=0_damped_newton_backslash_Nt=1_timestep1.csv};
			\addplot table[
					x expr = {\thisrow{epsilon}},
					y expr = {\thisrow{mean_symcond1}}
				]{data/DataCondNr_FSMDRK_HB-I2DRK6-3s_AT_p=3_dersol_PR_Tend=1.25_splitting=0_damped_newton_backslash_Nt=1_timestep1.csv};
			\addplot table[
					x expr = {\thisrow{epsilon}},
					y expr = {\thisrow{mean_symcond1}}
				]{data/DataCondNr_FSMDRK_HB-I3DRK9-3s_AT_p=4_dersol_PR_Tend=1.25_splitting=0_damped_newton_backslash_Nt=1_timestep1.csv};
		\end{loglogaxis}
		\end{tikzpicture}
	  }
	  {
	  \includegraphics{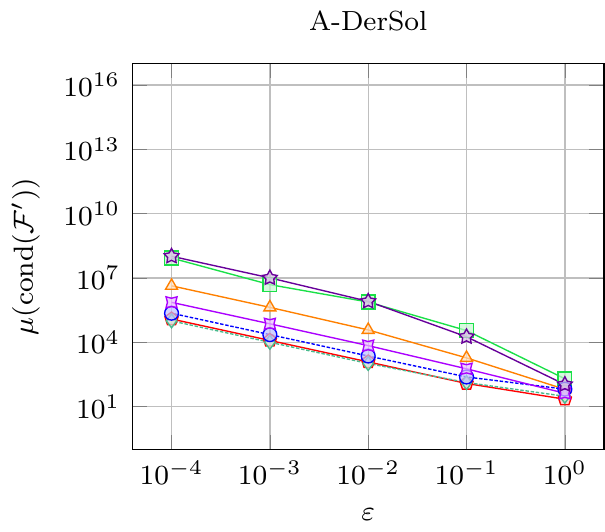}
	  }%	
\end{subfigure}%
\begin{subfigure}{0.33\textwidth}
  \centering
  \ifthenelse{\boolean{compilefromscratch}}{
        \tikzsetnextfilename{DataCondNr_FSMDRK-DerSol_EJ_PR_Tend=1.25}
		\begin{tikzpicture}
		\begin{loglogaxis}[
			title={EJ-DerSol},
			height=5.5cm,
			xlabel={$\varepsilon$},
			ylabel={\phantom{$\text{cond}(\mathcal{F}')$}},
			xtick={1e0,1e-1,1e-2,1e-3,1e-4},
			ylabel shift = 1.3em,
	        ymax = 1e17, ymin = 1e-1,
			ytick={1e1,1e4,1e7,1e10,1e13,1e16},
			ymajorticks=false,
			grid=major,
			legend style={nodes={scale=0.95, transform shape}},
			transpose legend,
			legend columns=3,
			legend to name=FSMDRK_schemes-dersol-PR,
			cycle list name = rainbow
			]
			\addplot table[
					x expr = {\thisrow{epsilon}},
					y expr = {\thisrow{mean_symcond1}}
				]{data/DataCondNr_FSMDRK_HB-I2DRK4-2s_dersol_EJ_PR_Tend=1.25_splitting=0_damped_newton_backslash_Nt=1_timestep1.csv};
		    \addlegendentry{HB-I2DRK4-2s};
		    
		    \addplot table[
					x expr = {\thisrow{epsilon}},
					y expr = {\thisrow{mean_symcond1}}
				]{data/DataCondNr_FSMDRK_HB-I3DRK6-2s_dersol_EJ_PR_Tend=1.25_splitting=0_damped_newton_backslash_Nt=1_timestep1.csv};
			\addlegendentry{HB-I3DRK6-2s};
			
			\addplot table[
					x expr = {\thisrow{epsilon}},
					y expr = {\thisrow{mean_symcond1}}
				]{data/DataCondNr_FSMDRK_HB-I4DRK8-2s_dersol_EJ_PR_Tend=1.25_splitting=0_damped_newton_backslash_Nt=1_timestep1.csv};
			\addlegendentry{HB-I4DRK8-2s};	
			
			\addplot table[
					x expr = {\thisrow{epsilon}},
					y expr = {\thisrow{mean_symcond1}}
				]{data/DataCondNr_FSMDRK_SSP-I2DRK3-2s_dersol_EJ_PR_Tend=1.25_splitting=0_damped_newton_backslash_Nt=1_timestep1.csv};
		    \addlegendentry{SSP-I2DRK3-2s};
		    
		    \addplot table[
					x expr = {\thisrow{epsilon}},
					y expr = {\thisrow{mean_symcond1}}
				]{data/DataCondNr_FSMDRK_SSP-I2DRK4-5s_dersol_EJ_PR_Tend=1.25_splitting=0_damped_newton_backslash_Nt=1_timestep1.csv};
			\addlegendentry{SSP-I2DRK4-5s};
			
			% Add empty legend to fix the legend layout.
	        	\addlegendimage{empty legend}
	        	\addlegendentry{}		
			
			\addplot table[
					x expr = {\thisrow{epsilon}},
					y expr = {\thisrow{mean_symcond1}}
				]{data/DataCondNr_FSMDRK_HB-I2DRK6-3s_dersol_EJ_PR_Tend=1.25_splitting=0_damped_newton_backslash_Nt=1_timestep1.csv};
			\addlegendentry{HB-I2DRK6-3s};
			
			\addplot table[
					x expr = {\thisrow{epsilon}},
					y expr = {\thisrow{mean_symcond1}}
				]{data/DataCondNr_FSMDRK_HB-I3DRK9-3s_dersol_EJ_PR_Tend=1.25_splitting=0_damped_newton_backslash_Nt=1_timestep1.csv};
			\addlegendentry{HB-I3DRK9-3s};	
		
		\end{loglogaxis}
		\end{tikzpicture}
	  }
	  {
	  \includegraphics{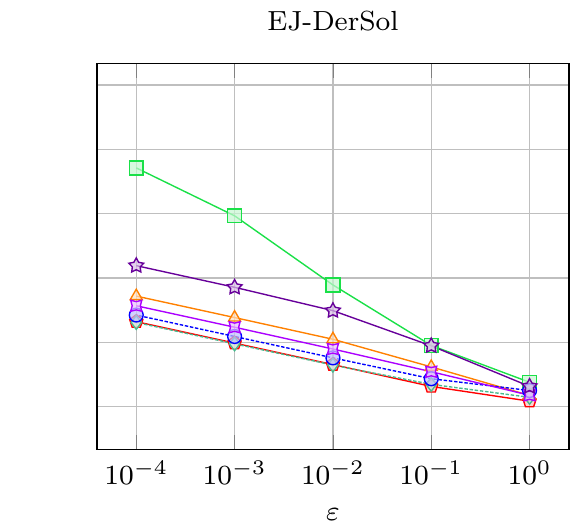}
	  }%
\end{subfigure}%
\begin{subfigure}{0.325\textwidth}
  \centering
  \ifthenelse{\boolean{compilefromscratch}}{
        \tikzsetnextfilename{DataCondNr_FSMDRK-DerSol_rec_PR_Tend=1.25}
		\begin{tikzpicture}
		\begin{loglogaxis}[
			title={rec-DerSol},
			height=5.5cm,
			xlabel={$\varepsilon$},
			ylabel={\phantom{$\text{cond}(\mathcal{F}')$}},
			xtick={1e0,1e-1,1e-2,1e-3,1e-4},
	        ymax = 1e17, ymin = 1e-1,
			ytick={1e1,1e4,1e7,1e10,1e13,1e16},
			ymajorticks=false,
			grid=major,
			legend style={nodes={scale=0.95, transform shape}},
			cycle list name = rainbow
			]
			\addplot table[
					x expr = {\thisrow{epsilon}},
					y expr = {\thisrow{mean_symcond1}}
				]{data/DataCondNr_FSMDRK_HB-I2DRK4-2s_dersol_PR_Tend=1.25_splitting=0_damped_newton_backslash_Nt=1_timestep1.csv};
		    \addplot table[
					x expr = {\thisrow{epsilon}},
					y expr = {\thisrow{mean_symcond1}}
				]{data/DataCondNr_FSMDRK_HB-I3DRK6-2s_dersol_PR_Tend=1.25_splitting=0_damped_newton_backslash_Nt=1_timestep1.csv};
		    \addplot table[
					x expr = {\thisrow{epsilon}},
					y expr = {\thisrow{mean_symcond1}}
				]{data/DataCondNr_FSMDRK_HB-I4DRK8-2s_dersol_PR_Tend=1.25_splitting=0_damped_newton_backslash_Nt=1_timestep1.csv};
			\addplot table[
					x expr = {\thisrow{epsilon}},
					y expr = {\thisrow{mean_symcond1}}
				]{data/DataCondNr_FSMDRK_SSP-I2DRK3-2s_dersol_PR_Tend=1.25_splitting=0_damped_newton_backslash_Nt=1_timestep1.csv};
		    \addplot table[
					x expr = {\thisrow{epsilon}},
					y expr = {\thisrow{mean_symcond1}}
				]{data/DataCondNr_FSMDRK_SSP-I2DRK4-5s_dersol_PR_Tend=1.25_splitting=0_damped_newton_backslash_Nt=1_timestep1.csv};
			\addplot table[
					x expr = {\thisrow{epsilon}},
					y expr = {\thisrow{mean_symcond1}}
				]{data/DataCondNr_FSMDRK_HB-I2DRK6-3s_dersol_PR_Tend=1.25_splitting=0_damped_newton_backslash_Nt=1_timestep1.csv};
			\addplot table[
					x expr = {\thisrow{epsilon}},
					y expr = {\thisrow{mean_symcond1}}
				]{data/DataCondNr_FSMDRK_HB-I3DRK9-3s_dersol_PR_Tend=1.25_splitting=0_damped_newton_backslash_Nt=1_timestep1.csv};
		\end{loglogaxis}
		\end{tikzpicture}
	  }
	  {
	  \includegraphics{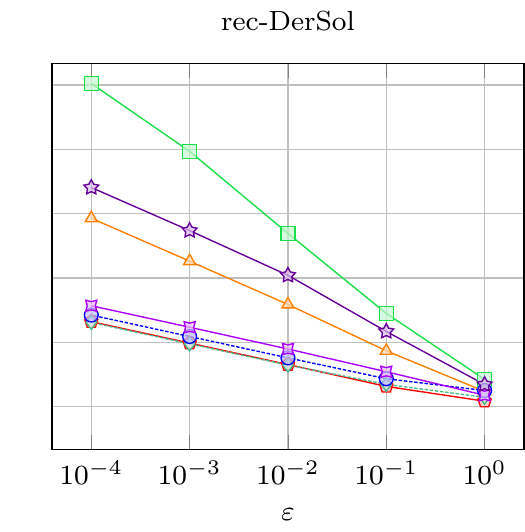}
	  }%
\end{subfigure}
\vspace{-2mm}
\hspace{1.2cm}\ifthenelse{\boolean{compilefromscratch}}{
				\ref{FSMDRK_schemes-dersol-PR}
			  }
	  		  {
			  \includegraphics{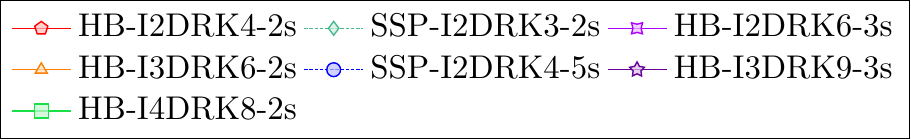}
			  }
			  
\caption{FSMDRK schemes applied as a DerSol method (Figure \ref{fig:AMDRK-scheme}) to the PR problem \eqref{eq:PR}. The average condition number in the 1-norm of the Newton Jacobian is shown for different values of $\varepsilon$. The behavior $\text{cond}(\mathcal{F}') = \mathcal{O}(\varepsilon^{-1})$ is observed for the A-DerSol and EJ-DerSol methods, the rec-DerSol methods seem to behave as $\mathcal{O}(\varepsilon^{-\mathtt{r}+1})$, $\mathtt{r}$ being the amount of derivatives. A single timestep ($N=1$) of size $T_{\text{end}} = 1.25$ ($\Delta t = 1.25$) has been considered with tolerances, Eq. \eqref{eq:criteria_newton}, set to $10^{-12}$ under a maximum of 1000 iterations.}
\label{fig:CondNr-FSMDRK_dersol-Pareschi-Russo}
\end{figure}
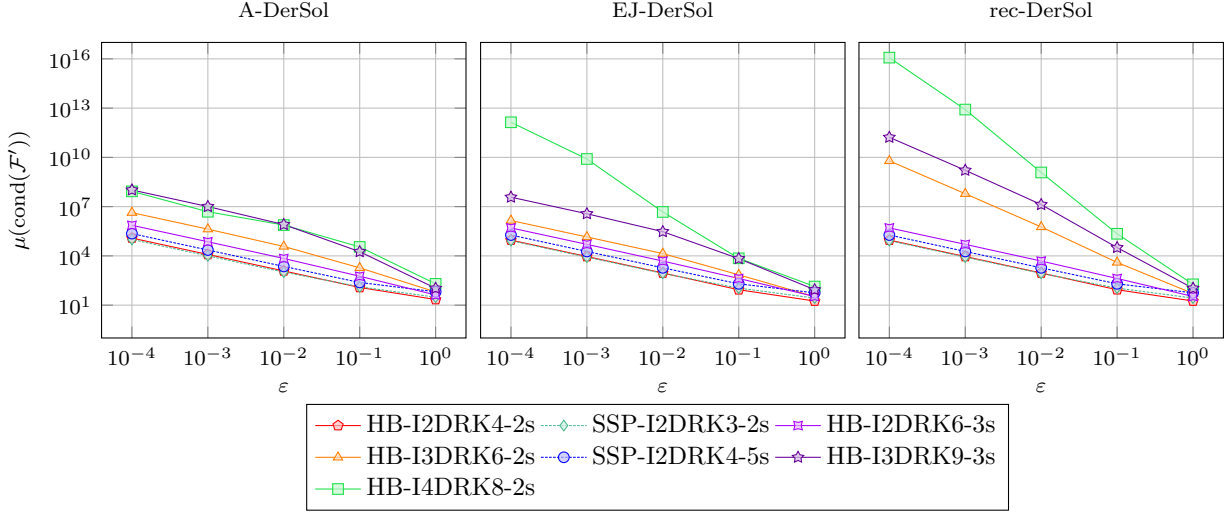

Figure \ref{fig:CondNr-FSMDRK_dersol-Pareschi-Russo} displays the average condition numbers $\mu\!\left(\text{cond}(\mathcal{F}')\right)$ for different MDRK schemes. The results are very similar to the ones of the DIMDRK implementation in Figure \ref{fig:CondNr-DIMDRK-dersol-Pareschi-Russo}, thus the previous remarks remaining valid pertaining to the FSMDRK implementation.
It is clear that $\mathcal{F}'(z)$ will quickly grow large for an increasing amount of derivatives $\mathtt{r}$ and stages $\mathtt{s}$, and that this consequently has an impact on the performance of the MDRK method. Still, it might be beneficial to  introduce the additional derivative relations for the overall efficiency of the method. As highlighted before w.r.t. the condition of the block-Jacobian \eqref{eq:jacobianFSMDRK-Direct}, here as well $\cond(\mathcal{F}'(z))$ is strongly dependent on the condition of the seperate blocks. If $\cond(\partial_{z^{n,{\color{VioletRed}\nu}}}\mathcal{F}_{{\color{BlueGreen}l}}) = \mathcal{O}(\varepsilon^{-1})$ can be guaranteed, there might be a significant difference in the total used amount of Newton iterations compared to the Direct counterpart. Furthermore, there is more certainty that the method itself will converge after all, which, for example, is not always the case for A-Direct methods (see Table \ref{tab:newton_cond_DIMDRK_ImplTaylor_order=3_AT_p=1_PR_Tend=1}).

\section{Conclusion and outlook}\label{sec:conclusion}
We have developed a family of implicit Jacobian-free multiderivative Runge-Kutta (MDRK) solvers for stiff systems of ODEs. These so-called AMDRK methods have been tailored to deal with the unwanted outcomes that come from the inclusion of a higher amount of derivatives: (1) each added ${k\text{-th}}$ derivative yields a power term $\varepsilon^k$ in the denominator, and (2) the complexity of the formulas for the derivatives increases rapidly with each derivative order.

When adopting Newton's method as a nonlinear solver, these two negatives become noticeable in the Jacobian: the condition number of the Jacobian grows exponentially with each added derivative, as well as the Jacobian having to be obtained from intricate formulas that request tensor calculations. In order to manage these negatives, the AMDRK methods have been established along the lines of the  Approximate Implicit Taylor method in \citep{2020_Baeza_EtAl}.

First, by adding an additional equation to the ODE system for each derivative, the derivatives become a part of the unknown solution, which we named ${\text{MDRK-DerSol}}$. In this manner, the $\varepsilon$-dependencies are distributed among the newly added relations. Numerically we have shown that this procedure alleviates the exponential growth in the condition number that is typical for direct MDRK methods (correspondingly named MDRK-Direct), for some cases resulting in much less Newton iterations per timestep. Second, by recursively approximating the derivatives using centered differences, no complicated formulas or tensor calculation are needed. The desired convergence order $\min(2p+1, q)$ is achieved, $2p+1$ denoting the amount of stencil points used for the centered differences and $q$ being the order of the MDRK scheme. 

Despite the (A)MDRK-DerSol methods for $\varepsilon \rightarrow 0$ having a more favorable behavior in the condition number in comparison to (A)MDRK-Direct methods, the total system grows in size, and therefore might be less effficient. In order to balance on the one hand the amount of Newton iterations per timestep, and on the other hand the computing time that is needed for solving the linear system, it might be beneficial in the future to establish a threshold value that switches between (A)MDRK-DerSol and (A)MDRK-Direct methods. Such threshold can play a siginificant role when transitioning to parabolic PDEs with viscous effects, where the size of linear systems depends on the spatial resolution. A careful consideration w.r.t. efficiency will be needed in the development of MDRK-DerSol approaches for PDEs with viscous effects.
%Furthermore, one can think of developing a similar method for general linear methods (GLMs), see for example \cj{our preprint in Arxiv??} where Jacobian-free explicit GLMs are constructed for hyperbolic conservation laws.

%This paper serves as a precursor for developing stable implicit MDRK methods for parabolic PDEs with viscous effects.

\section*{Declarations}
\textit{Conflicts of interest} The authors declare that they have no known competing financial interests or personal relationships that could have appeared to influence the work reported in this paper. \\
\\
\textit{Availability of data and material} The datasets generated and/or analyzed during the current study are available from the corresponding author on reasonable request via \url{jeremy.chouchoulis@uhasselt.be}.

%\textit{Code availability} The code can be downloaded from the personal webpage of Jochen Sch\"utz at \newline{\url{www.uhasselt.be/cmat}} or directly from ... \cj{todo}

\appendix

%\section{Newton data plots}\label{app:NewtonData}
%\input{sec_newtonData}

%\section{van der Pol problem}\label{app:vdP}
%\input{sec_vdP}
%
%\section{Kaps' problem}\label{app:kaps}
%\input{sec_kaps}

\section{Butcher tableaux}\label{app:ButherTableaux}
All the used multiderivative Runge-Kutta methods in this paper are displayed in this section. A typical multiderivative Runge-Kutta method can be summarized in an extended Butcher tableau of the form as in Table \ref{tab:MDRK_tableau}.

\begin{table}[h!]
\centering
\caption{A general extended Butcher tableau for a multiderivative Runge-Kutta scheme having $\mathtt{r}$ derivatives and $\mathtt{s}$ stages. The associated matrices and vectors are of size $A^{(k)} \in \mathbb{R}^{\mathtt{s}\times \mathtt{s}}$, $b^{(k)} \in \mathbb{R}^{1 \times \mathtt{s}}$ and $c \in \mathbb{R}^{\mathtt{s}\times 1}$ for $k = 1, \dots, \mathtt{r}$.}
\label{tab:MDRK_tableau}
\begin{tabular}{c|c|c|c}
$c$ & $A^{(1)}$ & $\hdots$  & $A^{(\mathtt{r})}$\\ \hline
    & $b^{(1)}$ & $\hdots$  & $b^{(\mathtt{r})}$
\end{tabular}
\end{table}

We use the explicit and implicit Taylor method reformulated as RK scheme, two-derivative Hermite-Birkhoff (HB) schemes taken from \cite{SealSchuetzZeifang21} together with new higher-derivative HB-schemes designed along the same line of reasoning, and Strong-Stability Preserving schemes taken from \cite{2022_Gottlieb_EtAl}. The corresponding Butcher tableaux of the HB schemes have been generated using a short \texttt{MATLAB} code which can be downloaded from the personal webpage of Jochen Sch\"utz at \url{www.uhasselt.be/cmat} or directly from \mbox{\url{http://www.uhasselt.be/Documents/CMAT/Code/generate_HBRK_tables.zip}}.

\begin{table}[h!]
\centering
\caption{$\mathtt{r}$-th order explicit Taylor.}
\label{tab:explTaylor}
\begin{tabular}{c|c|c|c|c}
$0$ & $0$ & $0$  & $\dots$ & $0$\\ \hline
    & $1$ & $1/2$  & $\dots$ & $1/\mathtt{r!}$
\end{tabular}
\end{table}

\begin{table}[h!]
\centering
\caption{$\mathtt{r}$-th order implicit Taylor.}
\label{tab:implTaylor}
\begin{tabular}{c|c|c|c|c}
$1$ & $1$ & $-1/2$  & $\dots$ & $(-1)^{\mathtt{r}+1}/\mathtt{r!}$\\ \hline
    & $1$ & $-1/2$  & $\dots$ & $(-1)^{\mathtt{r}+1}/\mathtt{r!}$
\end{tabular}
\end{table}

\begin{table}[h!]
\centering
\caption{HB-I2DRK4-2s: Fourth order implicit two-derivative Hermite-Birkhoff scheme using two stages \cite{SealSchuetzZeifang21}.}
\label{tab:HB-I2DRK4-2s}
\begin{tabular}{c|cc|cc}
$0$ & $0$   & $0$   & $0$   & $0$ \\
$1$ & $1/2$   & $1/2$   & $1/12$ & $-1/12$ \\ \hline
  & $1/2$   & $1/2$   & $1/12$ & $-1/12$
\end{tabular}
\end{table}

\begin{table}[h!]
\centering
\caption{HB-I2DRK6-3s: Sixth order implicit two-derivative Hermite-Birkhoff scheme using three stages \cite{SealSchuetzZeifang21}.}
\label{tab:HB-I2DRK6-3s}
\begin{tabular}{c|ccc|ccc}
$0$ & $0$   & $0$   & $0$   & $0$  & $0$   & $0$ \\
$1/2$ & $101/480$   & $8/30$ & $55/2400$ & $65/4800$ & $-25/600$ & $-25/8000$ \\
$1$ & $7/30$   & $16/30$   & $7/30$ & $5/300$ & $0$ & $-5/300$ \\ \hline
  & $7/30$   & $16/30$   & $7/30$ & $5/300$ & $0$ & $-5/300$ 
\end{tabular}
\end{table}

\begin{table}[h!]
\centering
\caption{HB-I2DRK8-4s: Eighth order implicit two-derivative Hermite-Birkhoff scheme using four stages \cite{SealSchuetzZeifang21}.}
\label{tab:HB-I2DRK8-4s}
\resizebox{\columnwidth}{!}{%
\begin{tabular}{c|cccc|cccc}
$0$ & $0$   & $0$   & $0$   & $0$ &  $0$   & $0$   & $0$   & $0$\\
$1/3$ & $6893/54432$   & $313/2016$   & $89/2016$ & $397/54432$ & $1283/272160$ & $-851/30240$ & $-269/30240$ & $-163/272160$ \\
$2/3$ & $223/1701$ & $20/63$ & $13/63$ & $20/1701$ & $43/8505$ & $-16/945$ & $-19/945$ & $-8/8505$ \\
$1$ & $31/224$ & $81/224$ & $81/224$ & $31/224$ & $19/3360$ & $-9/1120$ & $9/1120$ & $-19/3360$ \\ \hline
  & $31/224$ & $81/224$ & $81/224$ & $31/224$ & $19/3360$ & $-9/1120$ & $9/1120$ & $-19/3360$ 
\end{tabular}
}
\end{table}

\begin{table}[h!]
\centering
\caption{HB-I3DRK6-2s: Sixth order implicit three-derivative Hermite-Birkhoff scheme using two stages.}
\label{tab:HB-I3DRK6-2s}
\begin{tabular}{c|cc|cc|cc}
$0$ & $0$   & $0$   & $0$   & $0$ & $0$   & $0$\\
$1$ & $1/2$   & $1/2$   & $1/10$ & $-1/10$ & $1/120$ & $1/120$ \\ \hline
  & $1/2$   & $1/2$   & $1/10$ & $-1/10$ & $1/120$ & $1/120$
\end{tabular}
\end{table}

\begin{table}[h!]
\centering
\caption{HB-I3DRK9-3s: Ninth order implicit three-derivative Hermite-Birkhoff scheme using three stages.}
\label{tab:HB-I3DRK9-3s}
\resizebox{\columnwidth}{!}{%
\begin{tabular}{c|ccc|ccc|ccc}
$0$ & $0$   & $0$   & $0$   & $0$ & $0$   & $0$   & $0$ & $0$   & $0$ \\
$1/2$ & $5669/26880$   & $32/105$   & $-421/26880$ & $303/17920$ & $-1/32$ & $47/17920$ & $169/322560$ & $1/315$ & $-41/322560$ \\
$1$ & $41/210$ & $64/105$ & $41/210$ & $1/70$ & $0$ & $-1/70$ & $1/2520$ & $2/315$ & $1/2520$ \\ \hline
  & $41/210$ & $64/105$ & $41/210$ & $1/70$ & $0$ & $-1/70$ & $1/2520$ & $2/315$ & $1/2520$
\end{tabular}
}
\end{table}

\begin{table}[h!]
\centering
\caption{HB-I4DRK8-2s: Eighth order implicit four-derivative Hermite-Birkhoff scheme using two stages.}
\label{tab:HB-I4DRK8-2s}
\begin{tabular}{c|cc|cc|cc|cc}
$0$ & $0$   & $0$   & $0$   & $0$ & $0$   & $0$   & $0$   & $0$  \\
$1$ & $1/2$   & $1/2$   & $3/28$ & $-3/28$ & $1/84$ & $1/84$ & $1/1680$ & $-1/1680$ \\ \hline
  & $1/2$   & $1/2$   & $3/28$ & $-3/28$ & $1/84$ & $1/84$ & $1/1680$ & $-1/1680$ 
\end{tabular}
\end{table}

\begin{table}[h!]
\centering
\caption{SSP-I2DRK3-2s: Third order implicit two-derivative Strong-Stability Preserving scheme using two stages \cite{2022_Gottlieb_EtAl}.}
\label{tab:SSP-I2DRK3-2s}
\begin{tabular}{c|cc|cc}
$0$ & $0$   & $0$   & $-1/6$   & $0$ \\
$1$ & $0$   & $1$   & $-1/6$ & $-1/3$ \\ \hline
  &  $0$   & $1$   & $-1/6$ & $-1/3$ 
\end{tabular}
\end{table}

\begin{table}[h!]
\centering
\caption{SSP-I2DRK4-5s: Fourth order implicit two-derivative Strong-Stability Preserving scheme using five stages \cite{2022_Gottlieb_EtAl}.}
\label{tab:SSP-I2DRK4-5s}
\resizebox{\columnwidth}{!}{%
\begin{tabular}{c|ccccc}
$0.660949255604937$ & $0.660949255604937$ & $0$   & $0$   & $0$  & $0$ \\
$0.903150646005785$ & $0.660949255604937$ & $0.242201390400848$ & $0$ & $0$ & $0$ \\
$2.020339810245656$ & $0.660949255604937$ & $0.221847558352979$ & $1.137542996287740$ & $0$ & $0$ \\
$0.374733308278053$ & $0.060653001401867$ & $0.020022818960029$ & $0.102668776898047$ & $0.191388711018110$ & $0$ \\
$1.000000000000000$ & $0.060653001401867$ & $0.020022818960029$ & $0.102668776898047$ & $0.191388711018110$ & $0.625266691721946$ \\ \hline
  & $0.060653001401867$ & $0.020022818960029$ & $0.102668776898047$ & $0.191388711018110$ & $0.625266691721946$
\end{tabular}
}\vspace{1em}
\resizebox{\columnwidth}{!}{
\begin{tabular}{c|ccccc}
\phantom{$0.660949255604937$} & $-0.177750705279127$ & $0$   & $0$   & $0$  & $0$ \\
 & $-0.177750705279127$ & $-0.354733903778084$ & $0$ & $0$ & $0$ \\
 & $-0.177750705279127$ & $-0.324923198367868$ & $-0.403963513682271$ & $0$ & $0$ \\
 & $-0.016311560509453$ & $-0.029325895786881$ & $-0.036459667895230$ & $-0.161628266349058$ & $0$ \\
 & $-0.016311560509453$ & $-0.029325895786881$ & $-0.036459667895230$ & $-0.161628266349058$ & $-0.218859021269943$ \\ \hline
  & $-0.016311560509453$ & $-0.029325895786881$ & $-0.036459667895230$ & $-0.161628266349058$ & $-0.218859021269943$ 
\end{tabular}
}
\end{table}

\bibliographystyle{elsarticle-num-names} 
\bibliography{ListPaper}

\begin{thebibliography}{17}
\expandafter\ifx\csname natexlab\endcsname\relax\def\natexlab#1{#1}\fi
\providecommand{\url}[1]{\texttt{#1}}
\providecommand{\href}[2]{#2}
\providecommand{\path}[1]{#1}
\providecommand{\DOIprefix}{doi:}
\providecommand{\ArXivprefix}{arXiv:}
\providecommand{\URLprefix}{URL: }
\providecommand{\Pubmedprefix}{pmid:}
\providecommand{\doi}[1]{\href{http://dx.doi.org/#1}{\path{#1}}}
\providecommand{\Pubmed}[1]{\href{pmid:#1}{\path{#1}}}
\providecommand{\bibinfo}[2]{#2}
\ifx\xfnm\relax \def\xfnm[#1]{\unskip,\space#1}\fi
%Type = Article
\bibitem[{Kastlunger and Wanner(1972)}]{KastlungerWanner1972}
\bibinfo{author}{K.~Kastlunger}, \bibinfo{author}{G.~Wanner},
\newblock \bibinfo{title}{On {T}uran type implicit {R}unge-{K}utta methods},
\newblock \bibinfo{journal}{Computing} \bibinfo{volume}{9}
  (\bibinfo{year}{1972}) \bibinfo{pages}{317--325}.
%Type = Article
\bibitem[{Hairer and Wanner(1973)}]{HaWa73}
\bibinfo{author}{E.~Hairer}, \bibinfo{author}{G.~Wanner},
\newblock \bibinfo{title}{Multistep-multistage-multiderivative methods for
  ordinary differential equations},
\newblock \bibinfo{journal}{Computing (Arch. Elektron. Rechnen)}
  \bibinfo{volume}{11} (\bibinfo{year}{1973}) \bibinfo{pages}{287--303}.
%Type = Article
\bibitem[{Butcher and Hojjati(2005)}]{Butcher2005}
\bibinfo{author}{J.~C. Butcher}, \bibinfo{author}{G.~Hojjati},
\newblock \bibinfo{title}{Second derivative methods with {RK} stability},
\newblock \bibinfo{journal}{Numerical Algorithms} \bibinfo{volume}{40}
  (\bibinfo{year}{2005}) \bibinfo{pages}{415--429}.
%Type = Article
\bibitem[{Chan and Tsai(2010)}]{TC10}
\bibinfo{author}{R.~Chan}, \bibinfo{author}{A.~Tsai},
\newblock \bibinfo{title}{On explicit two-derivative {Runge-Kutta} methods},
\newblock \bibinfo{journal}{Numerical Algorithms} \bibinfo{volume}{53}
  (\bibinfo{year}{2010}) \bibinfo{pages}{171--194}.
%Type = Article
\bibitem[{Seal et~al.(2014)Seal, G\"u\c{c}l\"u, and Christlieb}]{Seal13}
\bibinfo{author}{D.~C. Seal}, \bibinfo{author}{Y.~G\"u\c{c}l\"u},
  \bibinfo{author}{A.~Christlieb},
\newblock \bibinfo{title}{High-order multiderivative time integrators for
  hyperbolic conservation laws},
\newblock \bibinfo{journal}{Journal of Scientific Computing}
  \bibinfo{volume}{60} (\bibinfo{year}{2014}) \bibinfo{pages}{101--140}.
%Type = Article
\bibitem[{Baeza et~al.(2018)Baeza, Boscarino, Mulet, Russo, and
  Zor\'{\i}o}]{BAEZA201887}
\bibinfo{author}{A.~Baeza}, \bibinfo{author}{S.~Boscarino},
  \bibinfo{author}{P.~Mulet}, \bibinfo{author}{G.~Russo},
  \bibinfo{author}{D.~Zor\'{\i}o},
\newblock \bibinfo{title}{Reprint of: Approximate {Taylor} methods for {ODE}s},
\newblock \bibinfo{journal}{Computers \& Fluids} \bibinfo{volume}{169}
  (\bibinfo{year}{2018}) \bibinfo{pages}{87 -- 97}. \bibinfo{note}{Recent
  progress in nonlinear numerical methods for time-dependent flow \& transport
  problems}.
%Type = Article
\bibitem[{Baeza et~al.(2020)Baeza, B{\"u}rger, Mart{\'i}, Mulet, and
  Zor{\'i}o}]{2020_Baeza_EtAl}
\bibinfo{author}{A.~Baeza}, \bibinfo{author}{R.~B{\"u}rger},
  \bibinfo{author}{M.~d.~C. Mart{\'i}}, \bibinfo{author}{P.~Mulet},
  \bibinfo{author}{D.~Zor{\'i}o},
\newblock \bibinfo{title}{On approximate implicit {Taylor} methods for ordinary
  differential equations},
\newblock \bibinfo{journal}{Computational and Applied Mathematics}
  \bibinfo{volume}{{39}} (\bibinfo{year}{2020}) \bibinfo{pages}{304}.
%Type = Article
\bibitem[{Sch{\"u}tz et~al.(2022)Sch{\"u}tz, Seal, and
  Zeifang}]{SealSchuetzZeifang21}
\bibinfo{author}{J.~Sch{\"u}tz}, \bibinfo{author}{D.~C. Seal},
  \bibinfo{author}{J.~Zeifang},
\newblock \bibinfo{title}{Parallel-in-time high-order multiderivative {IMEX}
  solvers},
\newblock \bibinfo{journal}{Journal of Scientific Computing}
  \bibinfo{volume}{90} (\bibinfo{year}{2022}) \bibinfo{pages}{1--33}.
%Type = Article
\bibitem[{Gottlieb et~al.(2022)Gottlieb, Grant, Hu, and
  Shu}]{2022_Gottlieb_EtAl}
\bibinfo{author}{S.~Gottlieb}, \bibinfo{author}{Z.~J. Grant},
  \bibinfo{author}{J.~Hu}, \bibinfo{author}{R.~Shu},
\newblock \bibinfo{title}{High order {S}trong {S}tability {P}reserving
  multiderivative implicit and {IMEX} {R}unge--{K}utta methods with asymptotic
  preserving properties},
\newblock \bibinfo{journal}{SIAM Journal on Numerical Analysis}
  \bibinfo{volume}{60} (\bibinfo{year}{2022}) \bibinfo{pages}{423--449}.
%Type = Article
\bibitem[{Zor\'{\i}o et~al.(2017)Zor\'{\i}o, Baeza, and Mulet}]{ZorioEtAl}
\bibinfo{author}{D.~Zor\'{\i}o}, \bibinfo{author}{A.~Baeza},
  \bibinfo{author}{P.~Mulet},
\newblock \bibinfo{title}{An approximate {Lax--Wendroff}-type procedure for
  high order accurate schemes for hyperbolic conservation laws},
\newblock \bibinfo{journal}{Journal of Scientific Computing}
  \bibinfo{volume}{71} (\bibinfo{year}{2017}) \bibinfo{pages}{246--273}.
%Type = Article
\bibitem[{Carrillo and Par\'{e}s(2019)}]{CarrilloPares2019}
\bibinfo{author}{H.~Carrillo}, \bibinfo{author}{C.~Par\'{e}s},
\newblock \bibinfo{title}{Compact approximate {T}aylor methods for systems of
  conservation laws},
\newblock \bibinfo{journal}{Journal of Scientific Computing}
  \bibinfo{volume}{80} (\bibinfo{year}{2019}) \bibinfo{pages}{1832--1866}.
%Type = Article
\bibitem[{Chouchoulis et~al.(2022)Chouchoulis, Schütz, and
  Zeifang}]{2021_Chouchoulis_EtAl}
\bibinfo{author}{J.~Chouchoulis}, \bibinfo{author}{J.~Schütz},
  \bibinfo{author}{J.~Zeifang},
\newblock \bibinfo{title}{Jacobian-free explicit multiderivative
  {R}unge-{K}utta methods for hyperbolic conservation laws},
\newblock \bibinfo{journal}{Journal of Scientific Computing}
  \bibinfo{volume}{90} (\bibinfo{year}{2022}).
%Type = Book
\bibitem[{Quarteroni et~al.(2007)Quarteroni, Sacco, and
  Saleri}]{QuarteroniNumA}
\bibinfo{author}{A.~Quarteroni}, \bibinfo{author}{R.~Sacco},
  \bibinfo{author}{F.~Saleri}, \bibinfo{title}{Numerical Mathematics},
  \bibinfo{publisher}{Springer}, \bibinfo{year}{2007}.
%Type = Article
\bibitem[{Pareschi and Russo(2000)}]{pareschi2000implicit}
\bibinfo{author}{L.~Pareschi}, \bibinfo{author}{G.~Russo},
\newblock \bibinfo{title}{Implicit-explicit {{Runge-Kutta}} schemes for stiff
  systems of differential equations},
\newblock \bibinfo{journal}{Recent Trends in Numerical Analysis}
  \bibinfo{volume}{3} (\bibinfo{year}{2000}) \bibinfo{pages}{269--289}.
%Type = Article
\bibitem[{Sch\"{u}tz and Seal(2021)}]{SealSchuetz19}
\bibinfo{author}{J.~Sch\"{u}tz}, \bibinfo{author}{D.~Seal},
\newblock \bibinfo{title}{An asymptotic preserving semi-implicit
  multiderivative solver},
\newblock \bibinfo{journal}{Applied Numerical Mathematics}
  \bibinfo{volume}{160} (\bibinfo{year}{2021}) \bibinfo{pages}{84--101}.
%Type = Article
\bibitem[{Boscarino and Russo(2009)}]{RuBosc09}
\bibinfo{author}{S.~Boscarino}, \bibinfo{author}{G.~Russo},
\newblock \bibinfo{title}{On a class of uniformly accurate {IMEX} {Runge-Kutta}
  schemes and applications to hyperbolic systems with relaxation},
\newblock \bibinfo{journal}{SIAM Journal on Scientific Computing}
  \bibinfo{volume}{31} (\bibinfo{year}{2009}) \bibinfo{pages}{1926--1945}.
%Type = Article
\bibitem[{Baeza et~al.(2018)Baeza, Boscarino, Mulet, Russo, and
  Zor{\'i}o}]{2018_Baeza_EtAl_Arxiv}
\bibinfo{author}{A.~Baeza}, \bibinfo{author}{S.~Boscarino},
  \bibinfo{author}{P.~Mulet}, \bibinfo{author}{G.~Russo},
  \bibinfo{author}{D.~Zor{\'i}o},
\newblock \bibinfo{title}{On the stability of {A}pproximate {T}aylor methods
  for {ODE} and their relationship with {R}unge-{K}utta schemes},
\newblock \bibinfo{journal}{arXiv preprint arXiv:1804.03627}
  (\bibinfo{year}{2018}). \href{http://arxiv.org/abs/1804.03627}{{\tt
  arXiv:1804.03627}}.

\end{thebibliography}





% \end{thebibliography}
\end{document}